\UseRawInputEncoding
\documentclass[11pt, reqno]{amsart}

\usepackage[utf8]{inputenc}
\usepackage[T1]{fontenc}
\usepackage[english]{babel}
\usepackage{lmodern}

\usepackage{setspace}
\usepackage{lipsum}

\usepackage{eurosym}

\usepackage{enumitem}
\usepackage{graphicx}
\allowdisplaybreaks[4]
\usepackage{float}
\usepackage{amsmath}
\usepackage{mathtools}
\usepackage{amssymb}
\usepackage{amsthm}
\usepackage{bbm}

\newcommand{\N}{\mathbb{N}}

\newcommand{\R}{\mathbb{R}}

\theoremstyle{plain}
\newtheorem{theorem}{Theorem}
\newtheorem{proposition}[theorem]{Proposition}
\newtheorem{corollary}[theorem]{Corollary}
\newtheorem{lemma}[theorem]{Lemma}

\newtheorem{assumption}{Assumption}

\theoremstyle{definition}

\newtheorem{remark}[theorem]{Remark}

\usepackage{color}
\usepackage[square,sort,comma,numbers]{natbib}

\usepackage[a4paper,bindingoffset=0.2in,%
            left=0.5in,right=0.5in,top=1in,bottom=1in,%
            footskip=.25in ]{geometry}

\usepackage{fancyhdr}

\usepackage{ifthen}

\fancyheadoffset{\textwidth}
\usepackage{blindtext}
\usepackage[pdftex,pdfpagelabels=true]{hyperref}

\hypersetup{
    colorlinks=true,
    linkcolor=blue,
    filecolor=magenta,
    urlcolor=cyan,
}

\usepackage[title]{appendix}

\begin{document}

\title[]{Rigorous Derivation of the Degenerate Keller-Segel System  from a Moderately Interacting  Stochastic  Particle System.\\
Part II Propagation of Chaos}
\date{\today}
\author{Li Chen}
\address{School of Business Informatics and Mathematics, Universität Mannheim, 68131, Mannheim, Germany}
\email{li.chen@uni-mannheim.de}

\author{Veniamin Gvozdik}
\address{School of Business Informatics and Mathematics, Universität Mannheim, 68131, Mannheim, Germany}
\email{veniamin.gvozdik@uni-mannheim.de}

\author{Yue Li}
\address{School of mathematics and Statistics, Shandong Normal University, 250358, Jinan, P.R. China}
\email{liyue2011008@163.com}

\begin{abstract}

This work is a series of two articles. The main goal  is to  rigorously derive  the degenerate  Keller-Segel system (including the case with fractional Laplacian in the equation for chemical potential) in the sub-critical regime from a moderately interacting stochastic particle system.  In the first  article \cite{gvozdik2022partone},  we establish the classical solution theory of the degenerate parabolic-elliptic  Keller-Segel system  and its non-local version. In the second article, which is the current one, we derive a propagation of chaos result, where the classical solution theory obtained in the first article is used  to derive required estimates for the particle systems. Due to the degeneracy of the  non-linear diffusion and the singular aggregation effect in the system,  we perform  an  approximation of  the stochastic  particle system by using a cut-off interacting potential. An additional linear diffusion on the particle level is used as a parabolic regularization of the system. We present the propagation of chaos result with  two different types of cut-off scaling for the aggregation potential, namely  logarithmic and  algebraic scaling. For the logarithmic scaling the convergence of  trajectories is obtained in expectation, while for the algebraic scaling the convergence in the sense of probability is derived. The result with algebraic scaling is deduced by studying  the  dynamics of a carefully constructed  stopped process and applying a generalized version of  the law of large numbers. Consequently, the propagation of chaos in the weak sense follows directly from these convergence results and the vanishing viscosity argument of the Keller-Segel system. Using the relative entropy method and additional $L^q$ estimate for the $l$-th marginal of the $N$ particle distribution, we further derive quantitative propagation of chaos results in $L^q$ $(1\leq q<\infty)$ norm by interpolation inequality. 

\end{abstract}
\vskip 0.2in

\maketitle


\section{Introduction}

Interacting particle systems of mean-field type belong to the well-known and fast growing fields in applied analysis. They are used to construct microscopic models with different physical and biological backgrounds, such as  dynamics of ions in plasma, distribution of bacteria in a fluid or flocking behavior of animals. The corresponding mean-field limit is usually given by a non-local partial differential equation.
The main goal of this series of two papers is to  rigorously derive the degenerate Keller-Segel system on the whole space $\R^d$ for $d\geq 3$ in the sub-critical regime from a moderately interacting stochastic particle system.
The so-called Keller-Segel model, which describes the collective motion of cells which are attracted or repelled by a chemical substance and are able to emit it, has been introduced by \citep{keller1970initiation,keller1971model,patlak1953random}, which nowadays has many different modifications. There is a huge amount of valuable works in the analysis of such systems. We focus only on the following degenerate diffusion-aggregation version:
\begin{align}
\label{Keller_Segel_classical}
\begin{cases}
\partial_t u = - \nabla \cdot (u \nabla c) + \nabla \cdot (u\nabla p(u)),    \\
c= \Phi_{\vartheta} * u(t,x), \quad \nabla\Phi_{\vartheta}=C_{d,\vartheta}\frac{ x}{|x|^{\vartheta}}, \, \frac{d}{2}+1\leq \vartheta\leq d,
\\
u(0,x)=u_0(x), \ \ \ \ x\in \R^d, t>0,
\end{cases}
\end{align}
where $u$ represents the density of the cells, $c$ denotes the self-produced chemical potential, the nonlinear diffusion is given by $p(u)=\frac{m}{m-1}u^{m-1}$ for $m>1$, and $C_{d,\vartheta}$ is a constant depends on $\vartheta$ and $d$. In case $\vartheta=d$, the second equation is exactly $-\Delta c=u$. The porous media type nonlinear diffusion in  \eqref{Keller_Segel_classical} has been introduced in order to balance the aggregation effect in higher dimensions. Rigorous derivation of \eqref{Keller_Segel_classical} includes the analysis of Keller-Segel systems (the degenerate and the non-local versions) and  mean-field theory of stochastic interacting particle systems. In our first article \citep{gvozdik2022partone}, the well-posedness of the degenerate parabolic-elliptic  Keller-Segel system  as well as  its non-local version are established. It is a self-contained result focused on analysis of the partial differential equations, which will be cited directly in the current paper. In the rest of this paper we will present the result on propagation of chaos.

The idea of the mean-field theory and propagation of chaos comes initially from statistical physics in studying of the Vlasov equation. For the classical results related to this topic we refer to review papers \citep{golse2016dynamics,jabin2014review}.
For reviews of the mean-field theory of the stochastic interacting particle systems we refer to \citep{jabin2017mean,sznitman1991topics}. The classical method of the mean-field theory, especially for smooth interactions,  is based on the direct trajectory  estimates. However, when the interaction potential is singular, the classical solution theory of the corresponding particle system could be enhanced by additional assumptions on initial data or cut-off potentials. In case that the interaction potential is of Coulomb type, Lazarovici and Pickl in \citep{lazarovici2017mean} considered a particle system with cut-off potential and random initial data, and derived the convergence of particle trajectories in the sense of probability. This implies propagation of chaos directly. Further results can be found in \citep{boers2016mean,bolley2011stochastic,chen2017mean,chen2020combined,garcia2017microscopic,godinho2015propagation,lazarovici2017mean} with different cut-off techniques for different models.
Oelschl\"ager studied  moderately interacting particle systems in \citep{oelschlager1990large} in order to derive porous medium type equations.  It was later extended by Philipowski in \citep{philipowski2007interacting} with the logarithmic scaling to obtain the porous medium equation. This approach has been further generalized by \citep{chen2019rigorous}  to derive a cross-diffusion system, in \citep{chen2021rigorous} to obtain the SKT system where the moderate interaction appeared in the diffusion coefficient and in  \citep{chen2021analysis} to derive a non-local porous medium equation.
For more information about interacting particle systems with singular potentials as well as some generalizations see  \citep{jabin2017mean} and references therein. A direct study of the strong solution to the particle system with multi-component $L^p$ potential and the mean field limit result is recently given in \cite{hao2022strong}.
Recently, the strong convergence for propagation of chaos has been obtained through relative entropy (or modulated energy) method for singular potentials without cut-off, for example in \citep{CHH,serfaty2017mean,serfaty2020mean,jabin2018quantitative,bresch2019mean}.

We use $V^{\varepsilon} (x) := \frac{1}{\varepsilon^d } V(x / \varepsilon)$ with $\varepsilon>0$ as a mollification kernel in the rest of the paper, where $V\geq 0$ is a given radially symmetric smooth function such that $\int _{\R^d}V(x) dx =1$.
We use the following smoothed potential for the particle systems:
\begin{align}\label{smoothPhi}
	\Phi_\vartheta^{\varepsilon} := \begin{cases}
		\Phi_\vartheta *V^{\varepsilon}, \ &\frac{d}{2}+1\leq\vartheta<d,  \\
		\bar\Phi_d*V^{\varepsilon}, \ &\vartheta=d,
	\end{cases}
\quad\mbox{where}\quad
\bar\Phi_d (|x|) := \begin{cases}
\Phi_d(|x|), \ &|x|\geq \varepsilon,  \\
\Phi_d(\varepsilon), \ &|x|< \varepsilon.
\end{cases}
\end{align}
Here, we use an additional cut-off $\bar\Phi_d$ for the Coulomb potential case in order to easily handle the singularity. Actually, we can also use the cut-off given in \cite{lazarovici2017mean} directly. We decide to take this form because it fits better in our framework so that it has a parallel structure to the mollification of the nonlinear diffusion.
The non-linear function $p(u)$ in the diffusion is approximated by $p_{\lambda} \in C^3(\R_*)$ with
\begin{align}
\label{p_lambda_definition}
p_{\lambda}(r)  := \begin{cases}
p(2/\lambda), \ &\text{if} \ r>2/ \lambda,  \\
p(r), \ & \text{if} \ 0\leq r< 1/ \lambda.
\end{cases}
\end{align}
Let $(\Omega, \mathcal{F}, (\mathcal{F}_{t\geq 0}) , \mathbb{P})$ be a complete filtered probability space, $N\in \N$ and  $\{ (B_t^i)_{t\geq 0} \}_{i=1}^N$ be $d$-dimensional independent $\mathcal{F}_t$-Brownian motions. For $\sigma>0$, $\sqrt{2\sigma }  \,dB_t^i$  will be used as a parabolic regularization of the particle system. Furthermore, let $\{\zeta_i\}_{i=1}^N$ be independent identically distributed (i.i.d.) random variables, which are independent of $\{ (B_t^i)_{t\geq 0} \}_{i=1}^N$ with a common density function $u_0^\sigma$. This $u_0^\sigma$, which appears only in the parabolic regularized system, represents the approximation of $u_0\in L^1(\mathbb{R}^d)\cap L^\infty(\mathbb{R}^d)$ in \eqref{Keller_Segel_classical}. More precisely,
$(u_0^\sigma)_\sigma\subset C_0^\infty(\R^d)$ and
\begin{eqnarray*}
u_0^\sigma\rightarrow u_0 \mbox{ in } L^l(\R^d), \forall l\in [1,\infty) \mbox{ as } \sigma\rightarrow 0, \mbox{ and } \|u_0^\sigma\|_{L^q(\R^d)}\leq \|u_0\|_{L^q(\R^d)}, \forall q\in [1,\infty].
\end{eqnarray*}
The minimal assumption for initial datum $u_0\in L^1(\mathbb{R}^d)\cap L^\infty(\mathbb{R}^d)$ in \eqref{Keller_Segel_classical} has been used in \cite{gvozdik2022partone} to obtain solution theory for the corresponding partial differential equations. Since the mean-field limit result in this paper is obtained under the assumption of the existence of these solutions, this condition for initial data will not be used directly in the current paper.

System \eqref{Keller_Segel_classical} contains an aggregation and a non-linear diffusion, so we introduce $\varepsilon_{k}>0$ and $\varepsilon_p>0$ to be regularization parameters for both terms respectively. Therefore we propose the following regularized interacting particle system as the approximation of \eqref{Keller_Segel_classical}:
\begin{align}
\label{generalized_regularized_particle_model}
\begin{cases}
dX_t^{N,i, \varepsilon, \sigma} \!\!= \displaystyle\frac{1}{N} \displaystyle\sum_{j =1}^N \nabla \Phi_\vartheta^{\varepsilon_k}(X_t^{N,i, \varepsilon, \sigma} \!\! -X_t^{N,j, \varepsilon, \sigma}) dt - \nabla p_{\lambda} \Big(\displaystyle \frac{1}{N}\displaystyle \sum_{j=1}^N  V^{\varepsilon_p}(X_t^{N,i, \varepsilon, \sigma} \!\! -X_t^{N,j, \varepsilon, \sigma} ) \Big) dt + \sqrt{2\sigma }  \,dB_t^i,  \\
X_0^{N,i, \varepsilon, \sigma} =\zeta^i, \qquad 1\leq i\leq N.
\end{cases}\end{align}
System \eqref{generalized_regularized_particle_model} is an interacting stochastic particle system which describes the dynamics of $N$ random processes with i.i.d. initial data. With the notation of empirical measure $\mu_t^{N, \varepsilon, \sigma} := \frac{1}{N} \sum_{i=1}^N \delta_{X_t^{N,i, \varepsilon, \sigma} }$, where $\delta_x$ is the Delta distribution centered at $x$, system \eqref{generalized_regularized_particle_model} can be rewritten into
\begin{equation*}
dX_t^{N,i, \varepsilon, \sigma} = \nabla \Phi_\vartheta^{\varepsilon_k}*\mu_t^{N, \varepsilon, \sigma}(X_t^{N,i, \varepsilon, \sigma})dt - \nabla p_{\lambda} \big( V^{\varepsilon_p}*\mu_t^{N, \varepsilon, \sigma}(X_t^{N,i, \varepsilon, \sigma}) \big) dt + \sqrt{2\sigma }  \,dB_t^i,\quad X_0^{N,i, \varepsilon, \sigma} =\zeta^i.
\end{equation*}
In this setting, the interacting effects include two parts: $\Phi_\vartheta^{\varepsilon_k}$ means the mollified chemotactic attractive potential and $ V^{\varepsilon_p}$ represents the moderated interacting effect appeared in the non-linear degenerate diffusion term $p_\lambda$ for $\lambda,\varepsilon_p\rightarrow 0$.
With the above mollification and cut-offs, the classical theory of stochastic systems implies that system \eqref{generalized_regularized_particle_model} has a unique strong solution for given $\varepsilon_k$, $\varepsilon_p$ and $\lambda$.

As a mean field approximation of \eqref{generalized_regularized_particle_model}, we introduce the so-called intermediate particle system
\begin{align}
\label{generalized_intermediate_particle_model}
\begin{cases}
d\bar{X}_t^{i, \varepsilon, \sigma} = \nabla  \Phi_\vartheta^{\varepsilon_k}  *u^{\varepsilon, \sigma} ( t, \bar{X}_t^{i, \varepsilon, \sigma}) dt - \nabla p_{\lambda} \big(  V^{\varepsilon_p} * u^{\varepsilon, \sigma} (t, \bar{X}_t^{i, \varepsilon, \sigma}  ) \big) dt +\sqrt{2\sigma } dB_t^i ,\\
\bar{X}_0^{i, \varepsilon, \sigma} =\zeta^i,
\end{cases}
\end{align}
where $ u^{\varepsilon, \sigma}(t, x)$ is the probability density function of $\bar{X}_t^{i, \varepsilon, \sigma}$ which  solves the following system of  partial differential equations in the weak sense
\begin{align}
\label{generalized_equation_u_epsilon_sigma}
\begin{cases}\partial_t u^{\varepsilon, \sigma} =  \sigma \Delta u^{\varepsilon, \sigma} -\nabla\cdot ( u^{\varepsilon, \sigma} \nabla c^{\varepsilon, \sigma}) + \nabla \cdot (u^{\varepsilon, \sigma}\nabla p_{\lambda} (V^{\varepsilon_p} *u^{\varepsilon, \sigma} ) ),     \\
c^{\varepsilon, \sigma} = \Phi_\vartheta^{\varepsilon_k}* u^{\varepsilon, \sigma},  \\
u^{\varepsilon, \sigma}(0,x)=u_0^\sigma(x), \ \ \ \ x\in \R^d, t>0.
\end{cases}
\end{align}
If we let $\varepsilon_k,\varepsilon_p,\lambda\rightarrow 0$, the limiting particle system reads
\begin{align}
\label{generalized_particle_model}
\begin{cases} d\hat{X}_t^{i, \sigma} =   \nabla  \Phi _\vartheta *u^{\sigma} ( t, \hat{X}_t^{i, \sigma}) dt- \nabla p  (u^{ \sigma} (t, \hat{X}_t^{i,  \sigma}  ) )dt +\sqrt{2\sigma } dB_t^i,  \\
\hat{X}_0^{i, \sigma} =\zeta^i,
\end{cases}
\end{align}
where $ u^{\sigma}(t, x)$ is the probability density function of $\hat{X}_t^{i, \sigma}$ and  solves the following system of  partial differential equations in the weak sense
\begin{align}
\label{generalized_equation_u_sigma}
\begin{cases}\partial_t u^{ \sigma} = \sigma \Delta u^{\sigma} -\nabla \cdot  ( u^{ \sigma} \nabla c^\sigma) + \nabla  \cdot  (u^{ \sigma} \nabla p (u^{ \sigma}) ),  \\
c^{ \sigma} = \Phi_\vartheta*u^{ \sigma},  \\
u^{ \sigma}(0,x)=u_0^\sigma(x), \ \ \ \ x\in \R^d, t>0.
\end{cases}
\end{align}

The mean-field limit presented in the current paper is self-contained under the appropriate conditions on solutions of the corresponding PDE systems \eqref{generalized_equation_u_epsilon_sigma} and \eqref{generalized_equation_u_sigma}. These assumptions have all been given in \cite{gvozdik2022partone} for $\vartheta=d$ rigorously, while the case of $\frac{d}{2}+1\leq\vartheta<d$ requires the following similar assumptions

\begin{assumption}\label{ass}
 Assume that $u^{\varepsilon,\sigma},u^\sigma\in  L^\infty(0,T;(1+|x|^2)L^1(\R^d))\cap L^\infty(0,T;H^s(\R^d))$, $s>\frac{d}{2}+2$, are the solutions of problems \eqref{generalized_equation_u_epsilon_sigma} and \eqref{generalized_equation_u_sigma} separately, furthermore it holds that $ \|u^{\sigma}-u^{\varepsilon,\sigma}\|_{L^{\infty}(0,T; H^s(\R^d))}\leq C(T)(\varepsilon_k+\varepsilon_p)$.
\end{assumption}

Under these conditions we introduce the main result  in the following
\begin{theorem}
\label{coroll_X_N_and_X_hat}
Assume that $T>0$, $m=2$ or $m\geq 3$, $\sigma\geq 0$, and Assumption \ref{ass} holds, then the problem \eqref{generalized_particle_model} has a unique strong solution $\hat{X}_t^{i,  \sigma}$ with $u^\sigma(t,\cdot)$ as the density of its law, and furthermore it holds
\begin{enumerate}
\item[(a)] In case $\vartheta=d$, $\Phi_d^\varepsilon=\Phi_d*V^\varepsilon$, for $ \varepsilon_k = \big( {\alpha_k\ln N} \big)^{-\frac{1}{d}}
$, $\varepsilon_p = \big( {\alpha_p\ln N} \big)^{-\frac{1}{dm-d+2}}$, and $ \lambda = \frac{\varepsilon_p^d}{2}$ with $ 0 <
	\alpha_k + \alpha_p \ll 1$,
then there exists a constant $C>0$ independent of $N$ such that
\begin{align*}
\underset{t\in [0,T]}{\sup } \underset{i\in\{1, \ldots , N\}}{\max } \mathbb{E} \left[ \left| X_t^{N,i, \varepsilon, \sigma} - \hat{X}_t^{i,  \sigma} \right|^2 \right] \leq  C(\varepsilon_k +\varepsilon_p)^2;
\end{align*}

\item[(b)] In case $\frac{d}{2}+1\leq\vartheta\leq d$, for $\varepsilon_k \geq N^{-\beta_k}$, $\varepsilon_p \geq (\ln N)^{-\frac{1}{2+d(2+|m-3|)}}$, and $ \lambda = \frac{\varepsilon_p^d}{2}$ with $0<\beta_k< \frac{1}{2(\vartheta+1)}$, then there exists a constant $C>0$ independent of $N$ such that
\begin{align*}
&\sup_{t\in[0, T]} \mathbb{P} \Big( \max_{i \in\{1, \ldots , N\} } \left| X_t^{N,i, \varepsilon, \sigma} - \hat{X}_t^{i, \sigma} \right|
>\varepsilon_p^\frac{1}{2}  \Big) \leq C \varepsilon_p ^\frac12.
\end{align*}
\end{enumerate}
\end{theorem}

\begin{remark}
	The result in Theorem \ref{coroll_X_N_and_X_hat} (a) is also true for $\frac{d}{2}+1\leq\vartheta<d$ with similar arguments. We omit the proof details. If we compare the results in (b) to (a), there is an improvement of both smoothing parameters $\varepsilon_{k}$ and $\varepsilon_p$. On the one hand, the interaction mollification radius $\varepsilon_{k}$ is reduced to algebraic scaling. On the other hand, the mollification radius for the  nonlinear diffusion, $\varepsilon_p$, the smallness of $\alpha_p$ is not necessary any more. 
\end{remark}


\begin{corollary}[Propagation of chaos in the weak sense]
\label{propagation_of_chaos}
Let $l\in \mathbb{N}$ and consider a $l$-tuple $(X_t^{N,1, \varepsilon, \sigma}, \ldots , X_t^{N,l, \varepsilon, \sigma} )$. We  denote by  $P^l_{N, \varepsilon, \sigma}(t)$  the joint distribution of   $(X_t^{N,1, \varepsilon, \sigma}, \ldots , X_t^{N,l, \varepsilon, \sigma} )$.   Then it holds that
\begin{align*}
  P^l_{N, \varepsilon, \sigma}(t) \ \text{converges weakly to} \ P_\sigma^{\otimes l}(t)
\end{align*}
as $N \rightarrow \infty$; $\varepsilon_k, \varepsilon_p, \lambda \rightarrow 0$ in the sense of Theorem \ref{coroll_X_N_and_X_hat} (a) or \ref{coroll_X_N_and_X_hat} (b), where $P_\sigma(t)$ is a measure which is absolutely continuous with respect to the Lebesgue measure and has a probability density function $u^\sigma(t,x)$ which solves \eqref{generalized_equation_u_sigma} in the weak sense.
Furthermore, there is a subsequence of $(u^\sigma)$ which converges weakly to $u$, the weak solution of \eqref{Keller_Segel_classical}.
\end{corollary}

\begin{remark}
The condition $\sigma>0$ is only used in \cite{gvozdik2022partone} to obtain the regularity of the PDE solutions in Assumption \ref{ass}. The whole proof of the propagation of chaos works also for the case $\sigma=0$. In other words, if the solution of \eqref{Keller_Segel_classical} satisfies the same regularity as those have been given in Assumption \ref{ass}, one can derive the propagation of chaos result without introducing the parabolic regularity.
\end{remark}

The proof of the main result is going to be shown in three separate parts. Under the Assumption \ref{ass}, the well-posedness of problems \eqref{generalized_intermediate_particle_model} and \eqref{generalized_particle_model} is obtained by using It\^o's formula and duality method. Furthermore, the estimate between $u^\sigma-u^{\varepsilon,\sigma}$ implies directly the estimate for $ \max_{i\in\{1, \ldots , N\}}  |  \bar{X}_t^{i, \varepsilon, \sigma}  - \hat{X}_t^{i, \sigma}  | $. These results are given in Section \ref{intermediate_and_final_model}. The estimates for $ \max_{i \in\{1, \ldots , N\} } | X_t^{N,i, \varepsilon, \sigma} - \bar{X}_t^{i, \varepsilon, \sigma} | $ are given in logarithmic and algebraic scaling separately in Sections \ref{Section_conv_in_exp} and \ref{Section_conv_in_prob}. Hence, the result in Theorem \ref{coroll_X_N_and_X_hat} (a) is a direct consequence of those results from Sections \ref{intermediate_and_final_model} and \ref{Section_conv_in_exp}, while (b) follows from the results of Sections \ref{intermediate_and_final_model} and \ref{Section_conv_in_prob}.

For moderately interacting particle systems one can obtain the convergence of the trajectories in expectation with logarithmic scaling, for example \cite{philipowski2007interacting,chen2019rigorous}. Since the logarithmic scaling is relatively weak to describe the singularity of the interaction potential, more reasonable discussion such as algebraic scaling should be considered. However, due to the singularity of the potential one cannot expect that the trajectories of the regularized  particle system and the particle system of the mean-field type   remain close to each other on any finite time interval.  Nevertheless, one can obtain this convergence in the sense of probability, which is the  spirit of   \citep{boers2016mean}  and \citep{lazarovici2017mean} in deriving Vlasov-Poisson equations.  

One main contribution of this article is the proof of Theorem \ref{coroll_X_N_and_X_hat} (b). By introducing an algebraic cut-off for the interaction potential and a logarithmic cut-off for the degenerate diffusion, we construct a framework to obtain convergence rates in probability. This requires a delicate analysis of estimates involving this mixed scaling. We build up a suitable stopped process and then apply Markov's inequality, so that the problem is reduced to the  estimation of the expectation of the stopped process. Then we  study the time evolution of this expectation. Even for the dynamics of the stopped  process we do not have Lipschitz continuity, hence we  apply  Taylor's expansion at the point which corresponds to  \eqref{generalized_intermediate_particle_model}. Moreover, we  use the regularity of the solutions from the PDE analysis to absorb the singularity caused by the interaction potential. In the whole discussion we use several times a generalized version of  the law of large numbers which is presented in  Lemma \ref{lemma_estim_prob_bad_sets}.

Another main contribution lies in the following quantitative propagation of chaos result in the strong sense, which is obtained by applying interpolation between the relative entropy estimate and the $L^q$ estimate of the $l-$th marginal density. 

\begin{theorem}[Quantitative propagation of chaos result in $L^q(\R^{dl})$-norm]
\label{relativeLq}
Under the assumptions of Theorem  \ref{coroll_X_N_and_X_hat} (b),
let $u^{\varepsilon,\sigma}_{N,l}(t,x_1,\cdots,x_l)$ $(l\in\N)$ be the $l-$th marginal density of the joint density $u^{\varepsilon,\sigma}_N(t,x_1,\cdots,x_N)$ of $(X^{N,i,\varepsilon,\sigma}_{t})_{1\leq i\leq N}$,
then there exists a constant $C>0$ independent of $N$, $\varepsilon_k$ and $\varepsilon_p$ such that
\begin{align*}
&	\|u^{\varepsilon,\sigma}_{N,l}-{u^\sigma}^{\otimes l}\|_{L^\infty(0,T;L^1(\mathbb{R}^{dl}))}\leq C(l)\varepsilon_p^{\frac{4}{d+4}}.
\end{align*}
Furthermore, if $\varepsilon_p\geq (\frac{1}{8\hat{C}T}\ln N)^{-\frac{1}{3(dm-d+1)}}$
and $\varepsilon_k\geq(\frac{1}{8\hat{C}T}\ln N)^{-\frac{1}{2(\vartheta-1)}}$ with $\hat{C}$ being a positive constant independent of $N$, $\varepsilon_k$ and $\varepsilon_p$ (more details can be found in the proof), then we have
\begin{align*}
&\|u^{\varepsilon,\sigma}_{N,l}-{u^\sigma}^{\otimes l}\|_{L^\infty(0,T;L^q(\mathbb{R}^{dl}))}\leq C(l)(\varepsilon_k+\varepsilon_p)^{1-\frac{1}{q}}\quad {\mbox{for any }}1< q<2,\\
	&\|u^{\varepsilon,\sigma}_{N,l}-{u^\sigma}^{\otimes l}\|_{L^\infty(0,T;L^q(\mathbb{R}^{dl}))}\leq C(l)(\varepsilon_k+\varepsilon_p)\quad {\mbox{for any }}2\le q<\infty.
\end{align*}
\end{theorem}
This paper is organized as follows.
Section \ref{intermediate_and_final_model} is devoted to the error estimate for $ \max_{i\in\{1, \ldots , N\}}  |  \bar{X}_t^{i, \varepsilon, \sigma}  - \hat{X}_t^{i, \sigma}  | $. 
The estimates for $ \max_{i \in\{1, \ldots , N\} } | X_t^{N,i, \varepsilon, \sigma} - \bar{X}_t^{i, \varepsilon, \sigma} | $ are given in logarithmic and algebraic scaling in Sections \ref{Section_conv_in_exp} and \ref{Section_conv_in_prob} respectively.
In Sections \ref{propagationstrong}, we derive the quantitative propagation of chaos results in the strong sense.

\vspace{0.5cm}

\section{Particle models  \eqref{generalized_intermediate_particle_model} and \eqref{generalized_particle_model}}
\label{intermediate_and_final_model}

In this section we study the well-posedness of stochastic particle systems  \eqref{generalized_intermediate_particle_model} and \eqref{generalized_particle_model} and give the uniform in $\varepsilon$ estimate of the difference $ \bar{X}_t^{i, \varepsilon, \sigma}  - \hat{X}_t^{i, \sigma}$. Systems  \eqref{generalized_intermediate_particle_model} and \eqref{generalized_particle_model} should be considered together with its corresponding partial differential equations \eqref{generalized_equation_u_epsilon_sigma} and \eqref{generalized_equation_u_sigma}. The proofs for both problems are similar, we only present here the proof for \eqref{generalized_particle_model} and \eqref{generalized_equation_u_sigma}. Then we insert the regular solution of \eqref{generalized_equation_u_sigma} into the particle system \eqref{generalized_particle_model} to obtain a square integrable solution $\hat{X}^{i,\sigma}_t$ from the classical theory of SDEs. In the end, we identify the law of $\hat{X}^{i,\sigma}_t$ with the solutions to \eqref{generalized_equation_u_sigma}, where the whole process is achieved by using the duality method.

\begin{proposition}
	[Solvability of problem \eqref{generalized_particle_model}] Assume that $m=2$ or $m\geq 3$, and $u^\sigma\in L^\infty(0,T;L^1 (\R^d))\cap L^\infty(0,T;H^s(\R^d))$, $s>\frac{d}{2}+2$, is a solution of \eqref{generalized_equation_u_sigma}, then problem \eqref{generalized_particle_model} has a unique square integrable solution $\hat{X}^{i,\sigma}_t$ with $u^\sigma$ as the density of its distribution.
\end{proposition}

\begin{proof}
 Let $w^{\sigma}\in  L^\infty(0,T;L^1 (\R^d))\cap L^\infty(0,T;H^s(\R^d))$, $s>\frac{d}{2}+2$, be a solution of \eqref{generalized_equation_u_sigma}. For $i\in \{1, \ldots ,N\}$ we consider the following stochastic system
	\begin{align}
	\label{X_sigma_w}
	\begin{cases} d\hat{X}_t^{i, \sigma} =   \nabla  \Phi_{\vartheta}  *w^{\sigma} ( t, \hat{X}_t^{i, \sigma})dt - \nabla p  (w^{ \sigma} (t, \hat{X}_t^{i,  \sigma}  ) )dt +\sqrt{2\sigma } dB_t^i,  \\
	\hat{X}_0^{i, \sigma} =\zeta^i.
	\end{cases}
	\end{align}
We will use the following estimates: for $\frac{d}{2}+1\leq\vartheta\leq d$ it holds
\begin{align*}
&\| \nabla \Phi_\vartheta *w^{\sigma} \|_{L^{\infty}(0,T; L^{\infty}(\R^d))}\\
\leq & \| (\nabla \Phi_\vartheta \mathbb{I}_{B_1(0)}) *w^{\sigma} \|_{L^{\infty}(0,T; L^{\infty}(\R^d))} +\| (\nabla \Phi_\vartheta\mathbb{I}_{\R^d\backslash B_1(0)}) *w^{\sigma} \|_{L^{\infty}(0,T; L^{\infty}(\R^d))}\\
\leq & C(d,\vartheta)\|w^{\sigma} \|_{L^{\infty}(0,T; L^{\infty}(\R^d))}+C(d,\vartheta)\|w^{\sigma} \|_{L^{\infty}(0,T; L^1(\R^d))}\leq C.
\end{align*}
Furthermore, it holds for $\frac{d}{2}+1\leq\vartheta< d$ that
\begin{align*}
&\| D^2 \Phi_\vartheta *w^{\sigma} \|_{L^{\infty}(0,T; L^{\infty}(\R^d))} \\
\leq & \| (D^2 \Phi_\vartheta \mathbb{I}_{B_1(0)}) *w^{\sigma} \|_{L^{\infty}(0,T; L^{\infty}(\R^d))} +\| (D^2 \Phi_\vartheta\mathbb{I}_{\R^d\backslash B_1(0)}) *w^{\sigma} \|_{L^{\infty}(0,T; L^{\infty}(\R^d))}\\
\leq & C(d,\vartheta)\|w^{\sigma} \|_{L^{\infty}(0,T; L^{\infty}(\R^d))}+C(d,\vartheta)\|w^{\sigma} \|_{L^{\infty}(0,T; L^1(\R^d))}\leq C.
\end{align*}
The case $\vartheta=d$ needs to be handled differently, namely by using the Sobolev embedding theorem and the weak Young inequality we have
\begin{align}\label{ad1}
&\| D^2 \Phi_d *w^{\sigma} \|_{L^{\infty}(0,T; L^{\infty}(\R^d))}\leq C\| D \Phi_d *Dw^{\sigma} \|_{L^{\infty}(0,T; W^{s-1,p}(\R^d))}\nonumber\\
\leq & C(d) \|Dw^{\sigma} \|_{L^{\infty}(0,T; H^{s-1}(\R^d))}\leq C,
\end{align}
where $p=\frac{2d}{d-2}$, $s>\frac{d}{2}+2$, which implies $(s-1)p>d$.

Since $m\geq 2$ we obtain that
for all $x\in \R^d$,
\begin{eqnarray*}
&&|\nabla  \Phi_\vartheta  *w^{\sigma} ( t, x) - \nabla p  (w^{ \sigma} (t, x  ) )  | \leq  \| \nabla \Phi_\vartheta  *w^{\sigma} \|_{L^{\infty}(0,T; L^{\infty}(\R^d))} +  \| \nabla p (w^{\sigma}) \|_{L^{\infty}(0,T; L^{\infty}(\R^d))} \\
&\leq &\| \nabla \Phi_\vartheta *w^{\sigma} \|_{L^{\infty}(0,T; L^{\infty}(\R^d))} + \|\nabla w^{\sigma}  \|_{L^{\infty}(0,T; L^{\infty}(\R^d))}  \|p'\|_{L^{\infty}(0, \|w^{\sigma } \|_{L^{\infty} (0,T; L^{\infty}(\R^d))} )}\leq C,
\end{eqnarray*}	
and for all $x,y\in \R^d$,
\begin{align*}
	&\big| \nabla  \Phi_\vartheta  *w^{\sigma} ( t, x) - \nabla p  (w^{ \sigma} (t, x  ) ) -   \nabla  \Phi_\vartheta  *w^{\sigma} ( t, y) + \nabla p  (w^{ \sigma} (t, y  ) )     \big|\\
	\leq&  \left| \nabla  \Phi_\vartheta  *w^{\sigma} ( t, x)  -  \nabla  \Phi_\vartheta  *w^{\sigma} ( t, y)     \right| +  \left| \nabla p  (w^{ \sigma} (t, x  ) ) - \nabla p  (w^{ \sigma} (t, y  ) )  \right|\\
	\leq& \Big( \| D^2 \Phi_\vartheta  *w^{\sigma} \|_{L^{\infty}(0,T; L^{\infty}(\R^d))}  +\|p'\|_{L^{\infty}(0, \|w^{\sigma } \|_{L^{\infty} (0,T; L^{\infty}(\R^d))} )} \|D^2 w^{\sigma}\|_{L^{\infty}(0,T; L^{\infty}(\R^d))}     \\
	& \ \ \ \ \ +\|p''\|_{L^{\infty}(0, \|w^{\sigma } \|_{L^{\infty} (0,T; L^{\infty}(\R^d))} )}  \|\nabla  w^{\sigma}\|_{L^{\infty}(0,T; L^{\infty}(\R^d))}^2 \Big) |x-y|\leq C |x-y|.
\end{align*}
Then by classical result from stochastic differential equations, we deduce that there is a unique square integrable solution $\hat X_t^{i,\sigma}$. Furthermore, since the initial distribution of data $\zeta^i$ has a density function $u^\sigma_0$, we obtain that $\hat X_t^{i,\sigma}$ has a density function $u^{\sigma} \in L^{\infty}(0,T; L^1(\R^d))$.

It remains to prove that $u^{\sigma}$ is exactly the solution of \eqref{generalized_equation_u_sigma}, which is given here by $w^{\sigma}$.

Let $\varphi \in C_0^{\infty}([0,T)\times \R^d)$. Using It\^{o}'s formula and taking the expectation of \eqref{X_sigma_w} we obtain that
\begin{align*}
\int_{\R^d} \varphi(t, x) u^{\sigma}(t,x) dx =& \int_{\R^d} \varphi(0, x) u_0^\sigma(x) + \int_{\R^d} \int_0^t \partial_s \varphi(s, x) u^{\sigma}(s,x)  ds dx \\
&+ \int_{\R^d} \int_0^t \nabla \Phi_\vartheta  *w^{\sigma} ( s, x) \cdot \nabla \varphi(s, x) u^{\sigma}(s,x)  ds dx\\
&- \int_{\R^d} \int_0^t \nabla p  (w^{ \sigma} (s, x  ) ) \cdot \nabla \varphi(s, x) u^{\sigma}(s,x)  ds dx  +\int_{\R^d} \int_0^t \sigma \Delta \varphi(s, x) u^{\sigma}(s,x)  ds dx.
\end{align*}
Therefore, $u^{\sigma} \in L^\infty(0,T; L^1(\R^d))$ is a  weak solution (in the sense of distribution) to the linear partial differential equation
\begin{align}
\label{lin_eq_density}
\begin{cases}\partial_t u^{\sigma} = \sigma \Delta u^{\sigma} -\nabla \cdot  ( u^{\sigma} \nabla \Phi_\vartheta *w^{\sigma} ) + \nabla  \cdot  (u^{\sigma} \nabla p (w^{\sigma}) ),  \\
u^{\sigma}(0,x)=u_0^\sigma(x), \ \ \ \ x\in \R^d, t>0.
\end{cases}
\end{align}
Since $w^\sigma$ is the solution of \eqref{generalized_equation_u_sigma}, it is enough to prove that the weak solution to \eqref{lin_eq_density} is unique.

In the following we will show that if $u_0^\sigma\equiv 0$, then $u^{\sigma} \equiv 0$.
Let $0\leq  u^{\sigma} \in L^{\infty}(0,T; L^1(\R^d))$ be a solution to
\begin{align}
\label{integr_eq_density}
\int_0^T \int_{\R^d} u^{\sigma}\big(\partial_t \varphi + \sigma \Delta \varphi + \nabla \Phi_\vartheta *w^{\sigma}\cdot \nabla \varphi
- \nabla p(w^{\sigma}) \cdot \nabla \varphi\big) \, dx \, dt=0 , \ \ \ \ \forall \varphi \in C_0^{\infty}([0,T) \times \R^d).	
\end{align}
If we can show that $u^{\sigma} \in L^q((0,T) \times \R^d)$ for $q \in (1, \frac{d}{d-1})$, then for $g\in C_0^{\infty}((0,T)\times \R^d)$, it holds
\begin{align*}
\int_0^T \int_{\R^d} u^{\sigma} g \, dx \, dt= \int_0^T \int_{\R^d} u^{\sigma}\big(\partial_t \varphi + \sigma \Delta \varphi + \nabla \Phi_\vartheta *w^{\sigma}\cdot \nabla \varphi - \nabla p(w^{\sigma}) \cdot \nabla \varphi\big) \, dx \, dt =0,
\end{align*}
where we used that $\varphi \in W^{2,1}_{\tilde{q}} ( (0,T)\times \R^d )$ $(\frac{1}{q}+\frac{1}{\tilde q}=1)$ for $\tilde{q}>d$ is the strong solution of the following backward parabolic equation:
\begin{align*}
\begin{cases}
\partial_t \varphi + \sigma \Delta \varphi + \nabla \Phi_\vartheta *w^{\sigma}\cdot \nabla \varphi - \nabla p(w^{\sigma}) \cdot \nabla \varphi =g \ \text{in} \ (0,T)\times \R^d, \\
\varphi( T, \cdot )\equiv 0 \ \text{in} \ \R^d.
\end{cases}
\end{align*}
Therefore by the fundamental lemma of variational principle we get that
\begin{align*}
u^{\sigma} =0 \ \text{a.e.} \ \text{in} \ (0,T)\times \R^d.
\end{align*}
Finally, we concentrate on proving that $u^\sigma \in L^q((0,T) \times \R^d)$ for $q \in (1, \frac{d}{d-1})$.
For fixed $M>0$, we define  $g: [0, \infty) \rightarrow [0, \infty)$ such that
\begin{align*}
g(v):= \begin{cases} v^{q-1}, \ \text{for} \ 0\leq v \leq M^{\frac{1}{q-1}}, \\
M,  \ \text{for} \ v > M^{\frac{1}{q-1}}.
\end{cases}
\end{align*}
Since $u^{\sigma} \in L^{\infty}(0,T; L^1(\R^d))$, it follows that $g(u^{\sigma}) \in L^{\frac{1}{q-1}}((0,T) \times \R^d)$. Now consider a sequence of functions $(\bar{g}_n)_{n\in \N} \subset C_0^{\infty}((0,T) \times \R^d)$ such that
\begin{align}
\label{g_n_frst_prop}
&\bar{g}_n \rightarrow g(u^{\sigma})  \ \ \text{in} \ \ L^{\frac{1}{q-1}}((0,T) \times \R^d),\\
\label{g_n_scnd_prop}
&\| \bar{g}_n \|_{L^{\infty}((0,T)\times \R^d)} \leq \bar{C}, \ \ \forall n \in \N,\\
\label{g_n_thrd_prop}
&\bar{g}_n\overset{*}{\rightharpoonup} g(u^{\sigma}) \ \ \text{in} \ \ L^{\infty}((0,T)\times \R^d),
\end{align}
where $\bar{C}$ is a positive constant which depends on $M$.
From \eqref{g_n_frst_prop}, \eqref{g_n_scnd_prop} and the interpolation inequality we obtain that
\begin{align}
\label{g_n_frth_prop}
\bar{g}_n \rightarrow g(u^{\sigma}) \ \ \text{in} \ \ L^{r}((0,T) \times \R^d) \ \ \text{for} \ \ r\in \Big[\frac{1}{q-1}, \infty\Big).
\end{align}
Let us consider the backward heat equation
\begin{align*}
&\partial_t \varphi_n +\sigma \Delta \varphi_n = \bar{g}_n \ \ \text{in} \ \ (0,T)\times\R^d,\\
&\varphi_n(\cdot, T)=0 \ \ \text{in}\ \ \R^d.
\end{align*}
Since $C_0^{\infty}([0,T)\times \R^d)$ is dense in $\{ \varphi \in C_b^{2,1}([0,T]\times \R^d) \ | \ \varphi( T, \cdot )\equiv 0 \}$, it  implies that  \eqref{integr_eq_density}  holds for any $\varphi \in C_b^{2,1}([0,T]\times \R^d)$ such that $\varphi( T, \cdot )\equiv 0$. Choosing $\varphi_n$ in \eqref{integr_eq_density} and using the Gagliardo–Nirenberg–Sobolev inequality we have
\begin{align*}
\int_0^T \int_{\R^d} u^{\sigma} \bar{g}_n \, dx \, dt
&\leq \int_0^T \| u^{\sigma} \|_{L^1(\R^d)} \| \nabla \varphi_n \|_{L^{\infty}(\R^d)} \big( \| \nabla \Phi_\vartheta*w^{\sigma} \|_{L^{\infty}(\R^d)} + \| \nabla p(w^{\sigma}) \|_{L^{\infty}(\R^d)} \big) \, dt \\
&\leq C \int_0^T \| \nabla \varphi_n  \|_{L^{\infty}(\R^d)} \, dt
\leq C \int_0^T \| \varphi_n \|_{L^{\frac{1}{q-1}}(\R^d)}^{1-\theta} \| D^2 \varphi_n \|_{L^{\tilde{q}}(\R^d)}^{\theta}  \, dt,
\end{align*}
where $\tilde{q}=\frac{q}{q-1}>d$, $\theta = \frac{q-1+\frac{1}{d}}{\frac{1}{q} +q+\frac{2}{d} -2} \in (\frac{1}{2}, 1)$ and $C$ appeared in this proof is independent of $n$ and $M$.
Together with H\"older's inequality, we infer that
\begin{align*}
\int_0^T \int_{\R^d} u^{\sigma} \bar{g}_n \, dx \, dt
\leq& C\| \varphi_n \|_{L^1\big(0,T;L^\frac{1}{q-1}(\R^d)\big)}^{1-\theta}\| D^2 \varphi_n \|_{L^1(0,T;L^{\tilde{q}}(\R^d))}^\theta\\
\leq& C \| \varphi_n \|_{L^{\frac{1}{q-1}}((0,T)\times \R^d)}^{1-\theta} \| D^2 \varphi_n \|_{L^{\tilde{q}}((0,T)\times \R^d)}^{\theta}.
\end{align*}
Then using the parabolic regularity theory, we have
\begin{align*}
\int_0^T \int_{\R^d} u^{\sigma} \bar{g}_n \, dx \, dt &\leq C \| \bar{g}_n \|_{L^{\frac{1}{q-1}}((0,T)\times \R^d)}^{1-\theta} \| \bar{g}_n  \|_{L^{\tilde{q}}((0,T)\times \R^d)}^{\theta}.
\end{align*}
Taking $n \rightarrow \infty$ on both sides of the above inequality together with \eqref{g_n_thrd_prop}-\eqref{g_n_frth_prop}, we obtain that
\begin{align*}
\int_0^T \int_{\R^d} u^{\sigma} g(u^{\sigma}) \, dx \, dt \leq C \| g(u^{\sigma})  \|_{L^{\frac{1}{q-1}}((0,T)\times \R^d)}^{1-\theta} \| g(u^{\sigma}) \|_{L^{\tilde{q}}((0,T)\times \R^d)}^{\theta} \leq C \| g(u^{\sigma}) \|_{L^{\tilde{q}}((0,T)\times \R^d)}^{\theta}.
\end{align*}
Young's inequality implies
\begin{align*}
\int_0^T \int_{\R^d} u^{\sigma} g(u^{\sigma}) \, dx \, dt \leq C +\frac{1}{2} \int_0^T \int_{\R^d}  |g(u^{\sigma})|^{\tilde{q}} \, dx \, dt.
\end{align*}
From the definition of $g(u^{\sigma})$ and $\tilde{q}=\frac{q}{q-1}$, we have
\begin{align*}
0&\leq \int_{\{ u^{\sigma} \leq M^{\frac{1}{q-1}} \}} (u^{\sigma})^q \, dx \, dt + M \int_{\{ u^{\sigma}> M^{\frac{1}{q-1}} \}} u^{\sigma} \, dx \, dt \\
&\leq C + \frac{1}{2} \int_{\{ u^{\sigma}\leq M^{\frac{1}{q-1}} \}} (u^{\sigma})^q \, dx \, dt + \frac{1}{2}  \int_{\{u^{\sigma} > M^{\frac{1}{q-1}} \}} M^{\frac{q}{q-1}} \, dx \, dt.
\end{align*}
Moreover,
\begin{align*}
0\leq \frac{1}{2} \int_{\{ u^{\sigma}\leq M^{\frac{1}{q-1}} \}} (u^{\sigma})^q \, dx \, dt  \leq C + M \int_{\{ u^{\sigma} > M^{\frac{1}{q-1}} \}} \Big( \frac{1}{2} M^{\frac{1}{q-1}} - u^{\sigma} \Big) \, dx \, dt \leq C.
\end{align*}
Therefore, this allows us to take the limit $M\rightarrow \infty$ in the above inequality and so deduce
\begin{align*}
\| u^{\sigma} \|_{L^q((0,T)\times \R^d)} \leq C.
\end{align*}

\end{proof}

With the same idea we can prove the solvability of the problem \eqref{generalized_intermediate_particle_model} and obtain the following proposition
\begin{proposition}
	[Solvability of the problem \eqref{generalized_intermediate_particle_model}] Assume that $m=2$ or $m\geq 3$, and $u^{\varepsilon,\sigma}\in  L^\infty(0,T;L^1 (\R^d))\cap L^\infty(0,T;H^s(\R^d))$, $s>\frac{d}{2}+2$, is a solution of \eqref{generalized_equation_u_epsilon_sigma}, then the problem \eqref{generalized_intermediate_particle_model} has a unique square integrable solution $\bar{X}^{i,\varepsilon,\sigma}_t$ with $u^{\varepsilon,\sigma}$ as the density of its distribution.
\end{proposition}

Furthermore we obtain the error estimates between \eqref{generalized_intermediate_particle_model} and \eqref{generalized_particle_model} in the following
\begin{proposition}
	\label{lemma_X_bar_X_hat}
Assume that $m=2$ or $m\geq 3$, and Assumption \ref{ass} holds, then there exists a constant $C>0$ independent of $N, \lambda , \varepsilon_k$ and $\varepsilon_p$ such that
	\begin{align*}
	\underset{t\in [0,T]}{\sup } \mathbb{E} \Big[ \underset{i\in\{1, \ldots , N\}}{\max }  \big|  \bar{X}_t^{i, \varepsilon, \sigma}  - \hat{X}_t^{i, \sigma}  \big|^2 \Big] \leq  C( \varepsilon_k+\varepsilon_p )^2 \ \exp\big( C T\big),
	\end{align*}
	where $\bar{X}_t^{i, \varepsilon, \sigma} $ and $\hat{X}_t^{i, \sigma} $ are solutions of \eqref{generalized_intermediate_particle_model} and \eqref{generalized_particle_model} respectively.
\end{proposition}

\begin{proof}
By taking the difference of \eqref{generalized_intermediate_particle_model} and \eqref{generalized_particle_model}, and using  It\^o's formula, we deduce that
 \begin{align*}
 \left| \bar{X}_t^{i, \varepsilon, \sigma} -\hat{X}_t^{i, \sigma} \right|^2=& \int_0^t 2 (\bar{X}_s^{i, \varepsilon, \sigma} -\hat{X}_s^{i, \sigma})  \big( - \nabla p_{\lambda}\big(  V^{\varepsilon_p} * u^{\varepsilon, \sigma} (s, \bar{X}_s^{i, \varepsilon, \sigma} )\big) + \nabla p  \big( u^{ \sigma} (s, \hat{X}_s^{i,  \sigma}  )  \big)  \big) \, ds\\
 &+\int_0^t 2 (\bar{X}_s^{i, \varepsilon, \sigma} -\hat{X}_s^{i, \sigma})  \big(  \nabla \Phi_\vartheta^{\varepsilon_k} *u^{\varepsilon, \sigma} (s, \bar{X}_s^{i, \varepsilon, \sigma} ) - \nabla \Phi_\vartheta *u^{ \sigma} (s, \hat{X}_s^{i,  \sigma} )  \big) \, ds.
 \end{align*}
By inserting intermediate terms
 \begin{align*}
  |  \nabla \Phi_\vartheta^{\varepsilon_k} *u^{\varepsilon, \sigma} (s, \bar{X}_s^{i, \varepsilon, \sigma} ) - \nabla \Phi_\vartheta *u^{ \sigma} (s, \hat{X}_s^{i,  \sigma} ) | \leq& |\nabla \Phi_\vartheta *u^{ \sigma} (s, \hat{X}_s^{i,  \sigma} ) - \nabla \Phi_\vartheta *u^{ \varepsilon,\sigma} (s, \hat{X}_s^{i,  \sigma} ) |\\
  &+ | \nabla \Phi_\vartheta *u^{\varepsilon, \sigma} (s, \hat{X}_s^{i,  \sigma} ) - \nabla \Phi_\vartheta *u^{\varepsilon, \sigma} (s, \bar{X}_s^{i, \varepsilon, \sigma} ) |\\
  &+ |\nabla \Phi_\vartheta *u^{\varepsilon, \sigma} (s, \bar{X}_s^{i, \varepsilon, \sigma} ) - \nabla \Phi_\vartheta^{\varepsilon_k} *u^{\varepsilon, \sigma} (s, \bar{X}_s^{i, \varepsilon, \sigma} ) |,
 \end{align*}
 it follows from \cite[Lemma 13]{BCL} that for $\frac{d}{2}+1\leq\vartheta<d$
   \begin{align}
   \label{estimates_Keller_Segel_2}
   &\left| \int_0^t 2 (\bar{X}_s^{i, \varepsilon, \sigma} -\hat{X}_s^{i, \sigma})  \big(  \nabla \Phi_\vartheta^{\varepsilon_k} *u^{\varepsilon, \sigma} (s, \bar{X}_s^{i, \varepsilon, \sigma} ) - \nabla \Phi_\vartheta *u^{ \sigma} (s, \hat{X}_s^{i,  \sigma} )  \big) \, ds  \right| \notag\\
   \leq&  2\|\nabla\Phi_\vartheta* (u^{\sigma} -u^{\varepsilon, \sigma}) \|_{L^{\infty}(0,T; L^{\infty}(\R^d))} \int_0^t \left| \bar{X}_s^{i, \varepsilon, \sigma} -\hat{X}_s^{i, \sigma} \right| ds\notag \\
  &+ 2\| \Phi_\vartheta * D^2 u^{\varepsilon,\sigma}  \|_{L^{\infty}(0,T; L^{\infty}(\R^d))} \int_0^t  \left| \bar{X}_s^{i, \varepsilon, \sigma} -\hat{X}_s^{i, \sigma} \right|^2 ds \notag\\
  &+2 \| \Phi_\vartheta * D^2 u^{ \varepsilon,\sigma}  \|_{ L^{\infty}(0,T; L^{\infty}(\R^d))} \varepsilon_k   \int_0^t \left| \bar{X}_s^{i, \varepsilon, \sigma} -\hat{X}_s^{i, \sigma} \right| ds,
   \end{align}
and for $\vartheta=d$
\begin{align}
   \label{estimates_Keller_Segel_2d}
   &\left| \int_0^t 2 (\bar{X}_s^{i, \varepsilon, \sigma} -\hat{X}_s^{i, \sigma})  \big(  \nabla \Phi_d^{\varepsilon_k} *u^{\varepsilon, \sigma} (s, \bar{X}_s^{i, \varepsilon, \sigma} ) - \nabla \Phi_d *u^{ \sigma} (s, \hat{X}_s^{i,  \sigma} )  \big) \, ds  \right| \notag\\
   \leq&  2\|\nabla\Phi_d* (u^{\sigma} -u^{\varepsilon, \sigma}) \|_{L^{\infty}(0,T; L^{\infty}(\R^d))} \int_0^t \left| \bar{X}_s^{i, \varepsilon, \sigma} -\hat{X}_s^{i, \sigma} \right| ds\notag \\
  &+ 2\| \Phi_d * D^2 u^{\varepsilon,\sigma}  \|_{L^{\infty}(0,T; L^{\infty}(\R^d))} \int_0^t  \left| \bar{X}_s^{i, \varepsilon, \sigma} -\hat{X}_s^{i, \sigma} \right|^2 ds \notag\\
  &+2 \| \Phi_d * D^2 u^{ \varepsilon,\sigma}  \|_{ L^{\infty}(0,T; L^{\infty}(\R^d))} \varepsilon_k   \int_0^t \left| \bar{X}_s^{i, \varepsilon, \sigma} -\hat{X}_s^{i, \sigma} \right| ds \notag\\
  &+2\|u^{\varepsilon,\sigma}\|_{L^\infty(0,T;L^\infty(\R^d))}\|\nabla\Phi_d\|_{L^1(B_{\varepsilon_k}(0))}\int_0^t\left| \bar{X}_s^{i, \varepsilon, \sigma} -\hat{X}_s^{i, \sigma} \right| ds.
   \end{align}
Using the definition of $\nabla\Phi_d$, we can derive that $\|\nabla\Phi_d\|_{L^1(B_{\varepsilon_k}(0))}\leq C\varepsilon_k$.
By the similar argument as that one in \eqref{ad1}, we infer from Assumption \ref{ass} that, for $\frac{d}{2}+1\le\nu\le d$, 
$\|\nabla\Phi_\vartheta* (u^{\sigma} -u^{\varepsilon, \sigma}) \|_{L^{\infty}(0,T; L^{\infty}(\R^d))}\le C(\varepsilon_k+\varepsilon_p)$.
The other term can be handled similarly. However due to the nonlinearity, more terms are involved.
\begin{align*}
\big| \nabla p  \big( u^{ \sigma} (s, \hat{X}_s^{i,  \sigma}  )  \big)  - \nabla p_{\lambda}\big(  V^{\varepsilon_p} * u^{\varepsilon, \sigma}(s, \bar{X}_s^{i, \varepsilon, \sigma} ) \big)  \big| \leq& \big| \nabla p  \big( u^{ \sigma} (s, \hat{X}_s^{i,  \sigma}  )  \big) -  \nabla p  \big( u^{ \sigma} (s, \bar{X}_s^{i, \varepsilon, \sigma}  )  \big)  \big|\\
&+ \big|   \nabla p  \big( u^{ \sigma} (s, \bar{X}_s^{i, \varepsilon, \sigma}  )  \big)  - \nabla p  \big( u^{ \varepsilon,  \sigma} (s,\bar{X}_s^{i, \varepsilon, \sigma}  )  \big)  \big|\\
&+\big|  \nabla p  \big( u^{ \varepsilon,  \sigma} (s, \bar{X}_s^{i, \varepsilon, \sigma}  )  \big)   -  \nabla  p_{\lambda} \big( u^{ \varepsilon,  \sigma} (s, \bar{X}_s^{i, \varepsilon, \sigma}  )  \big)   \big| \\
&+ \big|  \nabla p_{\lambda}  \big( u^{ \varepsilon,  \sigma} (s,\bar{X}_s^{i, \varepsilon, \sigma}  )  \big)   - \nabla p_{\lambda}\big(  V^{\varepsilon_p} * u^{\varepsilon, \sigma} (s, \bar{X}_s^{i, \varepsilon, \sigma} ) \big)   \big|.
\end{align*}
We rewrite the first two term on the right-hand side as follows:
\begin{align*}
 \nabla p  \big( u^{ \sigma} (s, \hat{X}_s^{i,  \sigma}  )  \big) -  \nabla p  \big( u^{ \sigma} (s, \bar{X}_s^{i, \varepsilon, \sigma}  )  \big)  =&   p' \big( u^{ \sigma} (s, \hat{X}_s^{i,  \sigma}  ) \big) \big(\nabla   u^{ \sigma} (s, \hat{X}_s^{i,  \sigma}  ) -  \nabla   u^{ \sigma} (s, \bar{X}_s^{i, \varepsilon, \sigma}  ) \big) \\
 &+ \big(  p' \big( u^{ \sigma} (s, \hat{X}_s^{i,  \sigma}  ) \big) -  p'  \big( u^{ \sigma} (s, \bar{X}_s^{i, \varepsilon, \sigma}  )  \big)  \big)  \nabla   u^{ \sigma} (s, \bar{X}_s^{i, \varepsilon, \sigma}  ) ,
 \end{align*}
 and
 \begin{align*}
 \nabla p  \big( u^{ \sigma} (s, \bar{X}_s^{i, \varepsilon, \sigma}  )  \big)  - \nabla p  \big( u^{ \varepsilon,  \sigma} (s, \bar{X}_s^{i, \varepsilon, \sigma}  )  \big)  =& p'  \big( u^{ \sigma} (s, \bar{X}_s^{i, \varepsilon, \sigma}  )  \big) \big( \nabla u^{ \sigma} (s, \bar{X}_s^{i, \varepsilon, \sigma}  ) - \nabla u^{ \varepsilon,  \sigma} (s, \bar{X}_s^{i, \varepsilon, \sigma}  )  \big)\\
 &+ \big(p'  \big( u^{ \sigma} (s, \bar{X}_s^{i, \varepsilon, \sigma}  )  \big) - p'  \big( u^{ \varepsilon,  \sigma} (s, \bar{X}_s^{i, \varepsilon, \sigma}  )  \big)  \big) \nabla u^{ \varepsilon,  \sigma} (s, \bar{X}_s^{i, \varepsilon, \sigma}  ).
 \end{align*}
The third term is zero for small enough $\lambda$, since $\|u^{\varepsilon, \sigma} \|_{L^{\infty}(0,T; W^{1,\infty}(\R^d))}\leq C$.

The fourth term can be written as
 \begin{align*}
 &\nabla p_{\lambda}  \big( u^{ \varepsilon,  \sigma} (s,\bar{X}_s^{i, \varepsilon, \sigma}  )  \big)   - \nabla p_{\lambda}\big(  V^{\varepsilon_p} * u^{\varepsilon, \sigma} (s, \bar{X}_s^{i, \varepsilon, \sigma} ) \big)\\
  =&  p_{\lambda}'  \big( u^{ \varepsilon,  \sigma} (s,\bar{X}_s^{i, \varepsilon, \sigma}  )  \big) \big( \nabla u^{ \varepsilon,  \sigma} (s,\bar{X}_s^{i, \varepsilon, \sigma}  )   -    V^{\varepsilon_p} *  \nabla  u^{\varepsilon, \sigma} (s, \bar{X}_s^{i, \varepsilon, \sigma} )  \big) \\
&+ \Big( p_{\lambda}'  \big( u^{ \varepsilon,  \sigma} (s,\bar{X}_s^{i, \varepsilon, \sigma}  )  \big) -  p_{\lambda}' \big(  V^{\varepsilon_p} * u^{\varepsilon, \sigma} (s, \bar{X}_s^{i, \varepsilon, \sigma} ) \big) \Big)    V^{\varepsilon_p} * \nabla u^{\varepsilon, \sigma} (s, \bar{X}_s^{i, \varepsilon, \sigma} ).
 \end{align*}
Hence,
\begin{align*}
 &\mathbb{E} \Big[  \underset{i\in\{1, \ldots , N\}}{\max }  \int_0^t 2 (\bar{X}_s^{i, \varepsilon, \sigma} -\hat{X}_s^{i, \sigma})  \Big( - \nabla p_{\lambda}\big(  V^{\varepsilon_p} * u^{\varepsilon, \sigma} (s, \bar{X}_s^{i, \varepsilon, \sigma} ) \big) + \nabla p  \big( u^{ \sigma} (s,\hat{X}_s^{i,  \sigma}  )  \big)  \Big) \, ds  \Big] \\
\leq&  2 C_{3a} \int_0^t  \mathbb{E} \big[  \underset{i\in\{1, \ldots , N\}}{\max }  |\bar{X}_s^{i, \varepsilon, \sigma} -\hat{X}_s^{i, \sigma}|^2  \big] \, ds  + 2C_{3b} \int_0^t \mathbb{E} \big[  \underset{i\in\{1, \ldots , N\}}{\max }  |\bar{X}_s^{i, \varepsilon, \sigma} -\hat{X}_s^{i, \sigma}| \big] ds\\
&  +2 C_{3d}  \int_0^t  \mathbb{E} \big[  \underset{i\in\{1, \ldots , N\}}{\max }  |\bar{X}_s^{i, \varepsilon, \sigma} -\hat{X}_s^{i, \sigma}|   \big] ds ,
\end{align*}
where
\begin{align*}
 C_{3a} =&  \|p'\|_{L^{\infty}(0, \|u^{\sigma } \|_{L^{\infty} ((0,T)\times\R^d))} )} \|D^2 u^{\sigma}\|_{L^{\infty}((0,T)\times\R^d))} + \|p''\|_{L^{\infty}(0, \|u^{\sigma } \|_{L^{\infty} ((0,T)\times\R^d))} )}  \|\nabla  u^{\sigma}\|^2_{L^{\infty}((0,T)\times\R^d))}   ,\\
C_{3b}=& \|p'\|_{L^{\infty}(0, \| u^{\sigma}\|_{L^{\infty} (0,T;  L^{\infty}(\R^d) )})  }     \|  \nabla \left( u^{ \sigma}  -  u^{ \varepsilon,  \sigma} \right)     \|_{L^{\infty}(0,T; L^{\infty}(\R^d))}\\   &+  \|\nabla u^{ \varepsilon,  \sigma}  \|_{L^{\infty}((0,T)\times\R^d)}  \|p''\|_{L^{\infty}(0,  \max\{  \|u^{\sigma}\|_{L^{\infty} ((0,T)\times\R^d)} , \|u^{\varepsilon, \sigma}\|_{L^{\infty} ((0,T)\times\R^d))} \} )  }  \|   u^{ \sigma}  -  u^{ \varepsilon,  \sigma}     \|_{L^{\infty}((0,T)\times\R^d))},  \\
 C_{3d} =&  \varepsilon_p\|p'\|_{L^{\infty}(0, \| u^{\varepsilon, \sigma}\|_{L^{\infty} (0,T; L^{\infty}(\R^d))})  }  \|D^2 u^{\varepsilon, \sigma} \|_{L^{\infty}(0,T; L^{\infty}(\R^d) )}   \\
& + \varepsilon_p \|p''\|_{L^{\infty}(0, \| u^{\varepsilon, \sigma}\|_{L^{\infty} (0,T; L^{\infty}(\R^d))})  }  \| \nabla u^{\varepsilon, \sigma} \|^2_{L^{\infty}(0,T; L^{\infty}(\R^d)) }.
\end{align*}
Combining the estimates above and the assumptions of this lemma, we obtain that there exists a constant $C>0$ such that
\begin{align*}
  \mathbb{E} \Big[ \underset{i\in\{1, \ldots , N\}}{\max }   \left|\bar{X}_t^{i, \varepsilon, \sigma} -\hat{X}_t^{i, \sigma} \right|^2 \Big]
   \leq& C \int_0^t   \mathbb{E} \Big[ \underset{i\in\{1, \ldots , N\}}{\max }  \left|\bar{X}_s^{i, \varepsilon, \sigma} -\hat{X}_s^{i, \sigma} \right|^2 \Big]  \, ds\\
  &+ C( \varepsilon_k+\varepsilon_p ) \int_0^t   \mathbb{E} \Big[ \underset{i\in\{1, \ldots , N\}}{\max }  \left|\bar{X}_s^{i, \varepsilon, \sigma} -\hat{X}_s^{i, \sigma} \right| \Big]  \, ds.
\end{align*}
Therefore, Gr\"onwall's inequality yields
\begin{align*}
\underset{t\in [0,T]}{\sup }  \mathbb{E} \Big[ \underset{i\in\{1, \ldots , N\}}{\max }  \left|\bar{X}_t^{i, \varepsilon, \sigma} -\hat{X}_t^{i, \sigma} \right|^2 \Big] \leq   C( \varepsilon_k+\varepsilon_p )^2 \ \exp\big( C T\big).
\end{align*}
\end{proof}

\section{Convergence in Expectation with the Logarithmic Cut-off}
\label{Section_conv_in_exp}

In this section, we concentrate on the estimate for the difference of the particle model \eqref{generalized_regularized_particle_model} and the intermediate particle model \eqref{generalized_intermediate_particle_model} in the case $\vartheta=d$. Similar results can be obtained for $\frac{d}{2}+1\leq\vartheta<d$. The proof in this section is directly done for mollified potential $\Phi_d^\varepsilon=\Phi_d*V^\varepsilon$, an additional cut-off such as $\bar\Phi_d$ is not needed. The estimate is obtained in the sense of expectation with the logarithmic cut-off, hence the result in Theorem \ref{coroll_X_N_and_X_hat} (a) is a direct consequence.
\begin{proposition}
	\label{theorem_log}
	Assume $\vartheta=d$, $X_t^{N,i, \varepsilon, \sigma} $ and $\bar{X}_t^{i, \varepsilon, \sigma}$ are solutions of \eqref{generalized_regularized_particle_model} and  \eqref{generalized_intermediate_particle_model} respectively, then for $\beta \in (0,1)$ and $N$ large enough there exists a  constant $C>0$ independent of $N$  such that
	\begin{align*}
	\underset{t\in [0,T]}{\sup } \underset{i\in\{1, \ldots , N\}}{\max } \mathbb{E} \big[ \big| X_t^{N,i, \varepsilon, \sigma} - \bar{X}_t^{i, \varepsilon, \sigma} \big|^2 \big] \leq CN^{-\beta},
	\end{align*}
	where the cut-off parameters are chosen such that
	\begin{align*}
	\varepsilon_k = \big( \alpha_k\ln N \big)^{-\frac{1}{d}}, \varepsilon_p = \big( \alpha_p\ln N \big)^{-\frac{1}{dm-d+2}}, \ \lambda = \frac{\varepsilon_p^d}{2}, \ \text{ with } \ 0 <
	\alpha_k + \alpha_p < \frac{1 - \delta \cdot  \frac{2dm-2d+2}{dm-d+2}  - \beta }{CT}.
	\end{align*}
	Here $\delta>0$ is an arbitrary small constant and $C$ is a positive constant which depends on $V$, $m$ and $d$.
\end{proposition}

First of all, we derive the estimate for the mollified fundamental solution $\Phi_d^\varepsilon=\Phi_d*V^\varepsilon$.
\begin{lemma}
	\label{lemma_L_infty_estimates_phi}
	There exists a positive constant $C>0$ which depends only on the given mollification kernel $V$ and $d$ such that
	\begin{align*}
	\|\nabla   \Phi_d^{\varepsilon} \|_{L^{\infty}(\R^d)} \leq \frac{C}{\varepsilon^{d-1}}, \ \ \ \ \|D^2   \Phi_d^{\varepsilon} \|_{L^{\infty}(\R^d)} \leq    \frac{C}{\varepsilon^d},  \ \ \ \ \|D^3   \Phi_d^{\varepsilon} \|_{L^{\infty}(\R^d)} \leq    \frac{C}{\varepsilon^{d+1}}.
	\end{align*}
In a similar way, we show that the results also hold for $\Phi^\varepsilon_\nu$ with $\frac{d}{2}+1\le \nu<d$ and $\bar{\Phi}_d\ast V^\varepsilon$.
\end{lemma}
\begin{proof}
	Since $\nabla \Phi_d = C_d \frac{x}{|x|^d}$, we obtain that there exists a constant $C>0$ such that
	\begin{align*}
	\int_{|y|< \varepsilon} \frac{1}{|x-y|^{d-1}}    dy &=   \int_{ {\substack{|y|< \varepsilon \\|x-y|>\varepsilon }} } \frac{1}{|x-y|^{d-1}}    dy + \int_{ {\substack{|y|< \varepsilon \\|x-y|<\varepsilon }} } \frac{1}{|x-y|^{d-1}}    dy  \\
	&\leq C \Big( \frac{1}{\varepsilon^{d-1}} Vol(|y|< \varepsilon) +  \int_{|x-y|< \varepsilon} \frac{1}{|x-y|^{d-1}} d(x-y)\Big)\\
	&\leq C   \Big( \frac{\varepsilon^d}{\varepsilon^{d-1}} + \int_0^{\varepsilon} \frac{r^{d-1}}{r^{d-1}} dr \Big)\leq C  \varepsilon,
	\end{align*}
	which implies
	\begin{align*}
	\|\nabla  \Phi_d^{\varepsilon} \|_{L^{\infty}(\R^d)} \leq C_d \frac{1}{\varepsilon^d} \Big\|\int_{|y|< \varepsilon} \frac{1}{|x-y|^{d-1}} \Big| V \Big( \frac{y}{\varepsilon} \Big) \Big|   dy\Big\|_{L^{\infty}(\R^d)} \leq  C_d \frac{ \| V \|_{L^{\infty}(\R^d)}  }{\varepsilon^d}  \varepsilon \leq \frac{C}{\varepsilon^{d-1}}.
	\end{align*}
	Using the same method we deduce that
	\begin{align*}
	\|D^2  \Phi_d^{\varepsilon} \|_{L^{\infty}(\R^d)}  \leq  C_d \frac{1}{\varepsilon^{d+1}} \Big\|\int_{|y|< \varepsilon} \frac{1}{|x-y|^{d-1}} \Big|\nabla  V \Big( \frac{y}{\varepsilon} \Big) \Big|   dy\Big\|_{L^{\infty}(\R^d)}\leq \frac{C}{\varepsilon^{d}},
	\end{align*}
	and
	\begin{align*}
	\|D^3  \Phi_d^{\varepsilon} \|_{L^{\infty}(\R^d)} \leq C_d \frac{1}{\varepsilon^{d+2}} \Big\|\int_{|y|< \varepsilon} \frac{1}{|x-y|^{d-1}} \Big|D^2 V \Big( \frac{y}{\varepsilon} \Big) \Big|   dy\Big\|_{L^{\infty}(\R^d)}\leq \frac{C}{\varepsilon^{d+1}}.
	\end{align*}
\end{proof}

Now we start with the proof of the main result of this section.

\begin{proof}[Proof of Proposition \ref{theorem_log}]
For the difference $X_t^{N,i, \varepsilon, \sigma} - \bar{X}_t^{i, \varepsilon, \sigma}$, by using  It\^o's  formula we deduce that
\begin{align*}
 &\mathbb{E}\Big[\Big| X_t^{N,i, \varepsilon, \sigma} - \bar{X}_t^{i, \varepsilon, \sigma} \Big|^2 \Big]\\
 \leq & \mathbb{E}\Big[\int_0^t 2 (X_s^{N,i, \varepsilon, \sigma} - \bar{X}_s^{i, \varepsilon, \sigma}) \Big(   \frac{1}{N} \sum_{j =1}^N \nabla \Phi_d^{\varepsilon_k}(X_s^{N,i, \varepsilon, \sigma} -X_s^{N,j, \varepsilon, \sigma}) - \nabla  \Phi_d^{\varepsilon_k}  *u^{\varepsilon, \sigma} ( s, \bar{X}_s^{i, \varepsilon, \sigma})  \Big) ds\Big] \\
 & +\mathbb{E}\Big[\int_0^t 2 (X_s^{N,i, \varepsilon, \sigma} - \bar{X}_s^{i, \varepsilon, \sigma})  \Big( -\nabla p_{\lambda} \Big(  \frac{1}{N} \sum_{j=1}^N  V^{\varepsilon_p}(X_s^{N,i, \varepsilon, \sigma} -X_s^{N,j, \varepsilon, \sigma} ) \Big)   + \nabla p_{\lambda}\big(  V^{\varepsilon_p} * u^{\varepsilon, \sigma} (s, \bar{X}_s^{i, \varepsilon, \sigma} ) \big)  \Big) \, ds\Big].
\end{align*}
We first present the estimate for the expectation of the aggregation term
\begin{align*}
&\mathbb{E} \Big[ \int_0^t 2 (X_s^{N,i, \varepsilon, \sigma} - \bar{X}_s^{i, \varepsilon, \sigma}) \Big(   \frac{1}{N} \sum_{j =1}^N \nabla \Phi_d^{\varepsilon_k}(X_s^{N,i, \varepsilon, \sigma} -X_s^{N,j, \varepsilon, \sigma}) - \nabla  \Phi_d^{\varepsilon_k}  *u^{\varepsilon, \sigma} ( s, \bar{X}_s^{i, \varepsilon, \sigma})  \Big) ds \Big]\\
\leq& \mathbb{E} \Big[ \int_0^t 2 (X_s^{N,i, \varepsilon, \sigma} - \bar{X}_s^{i, \varepsilon, \sigma}) ( A_1+A_2 )ds\Big],
\end{align*}
where
\begin{align*}
  &A_1 :=  \frac{1}{N} \sum_{j =1}^N \nabla \Phi_d^{\varepsilon_k}(X_s^{N,i, \varepsilon, \sigma} -X_s^{N,j, \varepsilon, \sigma}) -  \frac{1}{N} \sum_{j =1}^N \nabla \Phi_d^{\varepsilon_k}( \bar{X}_s^{i, \varepsilon, \sigma} - \bar{X}_s^{j, \varepsilon, \sigma} ),\\
  &A_2:= \frac{1}{N} \sum_{j =1}^N \nabla \Phi_d^{\varepsilon_k}( \bar{X}_s^{i, \varepsilon, \sigma} - \bar{X}_s^{j, \varepsilon, \sigma} ) -\nabla  \Phi_d^{\varepsilon_k}  *u^{\varepsilon, \sigma} ( s, \bar{X}_s^{i, \varepsilon, \sigma}) .
\end{align*}
Using Lemma \ref{lemma_L_infty_estimates_phi} we get the following estimate for  $A_1$
\begin{align*}
|A_1| \leq  \frac{C}{N \varepsilon_k^d}\sum_{j =1}^N \big( |X_s^{N,i, \varepsilon, \sigma}-\bar{X}_s^{i, \varepsilon, \sigma} | + |X_s^{N,j, \varepsilon, \sigma} - \bar{X}_s^{j, \varepsilon, \sigma} | \big).
\end{align*}
Combining estimates for $|A_1|$ and Young's inequality we obtain that
 \begin{align*}
 \mathbb{E} \Big[ \int_0^t 2 (X_s^{N,i, \varepsilon, \sigma} - \bar{X}_s^{i, \varepsilon, \sigma}) ( A_1)\,ds\Big]
    &\leq\frac{C}{\varepsilon_k^d} \int_0^t\max_{i\in\{1,\ldots,N\}} \mathbb{E} \left[
    \left| X_s^{N,i, \varepsilon, \sigma} - \bar{X}_s^{i, \varepsilon, \sigma} \right|^2\right]ds.
 \end{align*}
Furthermore,
 \begin{align*}
&\mathbb{E} \Big[\int_0^t 2 (X_s^{N,i, \varepsilon, \sigma} - \bar{X}_s^{i, \varepsilon, \sigma}) ( A_2)ds\Big]\\
\leq&\frac{C}{\varepsilon_k^d} \int_0^t \max_{i\in\{1,\ldots,N\}}\mathbb{E} \big[ \big| X_s^{N,i, \varepsilon, \sigma} - \bar{X}_s^{i, \varepsilon, \sigma} \big|^2\big] ds + \varepsilon_k^d\int_0^t \mathbb{E} [ |A_2 |^2 ] ds.
\end{align*}
 Now we handle the term involving $A_2$. For simplicity we introduce
 \begin{align*}
 \hat{Z}_{ij}:=   \nabla \Phi_d^{\varepsilon_k}( \bar{X}_s^{i, \varepsilon, \sigma} - \bar{X}_s^{j, \varepsilon, \sigma} ) - \nabla \Phi_d^{\varepsilon_k}  *u^{\varepsilon, \sigma} ( s, \bar{X}_s^{i, \varepsilon, \sigma}).
\end{align*}
Therefore, we obtain that
\begin{align*}
|A_2|^2  = \Big| \frac{1}{N} \sum_{j=1}^N \hat{Z}_{ij} \Big|^2= \frac{1}{N^2}\sum_{j=1}^N \sum_{k=1}^N \hat{Z}_{ij} \cdot \hat{Z}_{ik}.
\end{align*}
For $i\neq j$, $i\neq k$ and $j\neq k$, since $\{ \bar{X}_s^{j, \varepsilon, \sigma}\}_{j=1}^N$ are independent, we obtain that
\begin{align*}
&\mathbb{E}[\hat{Z}_{ij} \cdot \hat{Z}_{ik}] \\
=& \mathbb{E}[ \big( \nabla \Phi_d^{\varepsilon_k}( \bar{X}_s^{i, \varepsilon, \sigma} - \bar{X}_s^{j, \varepsilon, \sigma} )  -  \nabla \Phi_d^{\varepsilon_k}  *u^{\varepsilon, \sigma} ( s, \bar{X}_s^{i, \varepsilon, \sigma}) \big) \cdot ( \nabla \Phi_d^{\varepsilon_k}( \bar{X}_s^{i, \varepsilon, \sigma} - \bar{X}_s^{k, \varepsilon, \sigma} )  -  \nabla \Phi_d^{\varepsilon_k}  *u^{\varepsilon, \sigma} ( s, \bar{X}_s^{i, \varepsilon, \sigma}) ) ] \\
=& \int_{\R^d} \int_{\R^d} \int_{\R^d}   ( \nabla \Phi_d^{\varepsilon_k}( x-y )  -  \nabla \Phi_d^{\varepsilon_k}  *u^{\varepsilon, \sigma} ( s, x ) ) \\
& \ \ \ \ \ \ \ \ \ \ \ \ \ \ \ \  \cdot ( \nabla \Phi_d^{\varepsilon_k}( x-z )  -  \nabla \Phi_d^{\varepsilon_k}  *u^{\varepsilon, \sigma} ( s, x ) )  u^{\varepsilon, \sigma}(s,x) u^{\varepsilon, \sigma}(s,y) u^{\varepsilon, \sigma}(s,z) \, dx \, dy \, dz =0.
\end{align*}
For $i\neq j$ and $i=k$, we have
\begin{align*}
&\mathbb{E}[\hat{Z}_{ij}\cdot\hat{Z}_{ii}]\\
=& \mathbb{E}[ -( \nabla \Phi_d^{\varepsilon_k}( \bar{X}_s^{i, \varepsilon, \sigma} - \bar{X}_s^{j, \varepsilon, \sigma} )  -  \nabla \Phi_d^{\varepsilon_k}  *u^{\varepsilon, \sigma} ( s, \bar{X}_s^{i, \varepsilon, \sigma}) ) \cdot \nabla \Phi_d^{\varepsilon_k}  *u^{\varepsilon, \sigma} ( s, \bar{X}_s^{i, \varepsilon, \sigma}) ]\\
=&0.
\end{align*}
Similarly, for $i=j$ and $i\neq k$
\begin{align*}
\mathbb{E}[\hat{Z}_{ii}\cdot\hat{Z}_{ik}]=0.
\end{align*}
Therefore, it holds
\begin{align}
\label{indep_Z_aggreg}
 \mathbb{E}[ |A_2|^2] =  \frac{1}{N^2} \mathbb{E} \Big[ \sum_{j=1}^N \sum_{k=1}^N \hat{Z}_{ij} \cdot \hat{Z}_{ik} \Big]
 \leq \frac{1}{N^2}\mathbb{E}\Big[\sum^N_{j=1}|\hat{Z}_{ij}|^2 \Big]+\frac{1}{N^2}\mathbb{E}[|\hat{Z_{ii}}|^2]
  \leq  \frac{C}{N \varepsilon_k^{2(d-1)}}.
\end{align}
In the last estimate the following inequality has been used
   \begin{align*}
   \hat{Z}_{ij}^2 &= \big(  \nabla \Phi_d^{\varepsilon_k}( \bar{X}_s^{i, \varepsilon, \sigma} - \bar{X}_s^{j, \varepsilon, \sigma} ) - \nabla \Phi_d^{\varepsilon_k}  *u^{\varepsilon, \sigma} ( s, \bar{X}_s^{i, \varepsilon, \sigma})  \big)^2  \leq  \frac{C}{\varepsilon_k^{2(d-1)}}.
   \end{align*}
As a summary, the expectation for the aggregation term can be estimated in the following
\begin{align}
\label{estimate_Keller_Segel}
&\mathbb{E} \Big[ \int_0^t 2 (X_s^{N,i, \varepsilon, \sigma} - \bar{X}_s^{i, \varepsilon, \sigma}) \Big(   \frac{1}{N} \sum_{j =1}^N \nabla \Phi_d^{\varepsilon_k}(X_s^{N,i, \varepsilon, \sigma} -X_s^{N,j, \varepsilon, \sigma}) - \nabla  \Phi_d^{\varepsilon_k}  *u^{\varepsilon, \sigma} ( s, \bar{X}_s^{i, \varepsilon, \sigma})  \Big) ds \Big]\notag\\
  \leq&  \frac{C}{\varepsilon_k^d}   \int_0^t \underset{i\in\{1, \ldots , N\}}{\max }  \mathbb{E} \Big[ \left| X_s^{N,i, \varepsilon, \sigma} - \bar{X}_s^{i, \varepsilon, \sigma} \right|^2 \Big]   ds  + \frac{Ct}{ N \varepsilon_k^{d-2}}.
\end{align}

Next we study the estimates for the diffusion term. Due to the nonlinear composition, this term is much more complicated than the aggregation term. More precisely, the difference can be splitted into the following differences
 \begin{align*}
 & - \nabla  p_{\lambda} \Big( \frac{1}{N} \sum_{j=1}^N  V^{\varepsilon_p}(X_s^{N,i, \varepsilon, \sigma} -X_s^{N,j, \varepsilon, \sigma} )  \Big)  +  \nabla  p_{\lambda}  \big(  V^{\varepsilon_p} * u^{\varepsilon, \sigma} (s,  \bar{X}_s^{i, \varepsilon, \sigma}) \big)  =: I_1+I_2,
 \end{align*}
 where
 \begin{align*}
   I_1=& \Big( p'_{\lambda}\Big( \frac{1}{N}\sum_{j=1}^N  V^{\varepsilon_p}(\bar{X}_s^{i, \varepsilon, \sigma}  -\bar{X}_s^{j, \varepsilon, \sigma} )  \Big)  - p'_{\lambda}\Big( \frac{1}{N} \sum_{j=1}^N  V^{\varepsilon_p}(X_s^{N,i, \varepsilon, \sigma} -X_s^{N,j, \varepsilon, \sigma} ) \Big)  \Big) \frac{1}{N}\sum_{j=1}^N  \nabla V^{\varepsilon_p}(\bar{X}_s^{i, \varepsilon, \sigma}  -\bar{X}_s^{j, \varepsilon, \sigma} ) \\
   & + p'_{\lambda}\Big( \frac{1}{N} \sum_{j=1}^N  V^{\varepsilon_p}(X_s^{N,i, \varepsilon, \sigma} -X_s^{N,j, \varepsilon, \sigma} ) \Big) \Big(\frac{1}{N}\sum_{j=1}^N  \big(\nabla V^{\varepsilon_p}(\bar{X}_s^{i, \varepsilon, \sigma}  -\bar{X}_s^{j, \varepsilon, \sigma} )   -    \nabla V^{\varepsilon_p}(X_s^{N,i, \varepsilon, \sigma} -X_s^{N,j, \varepsilon, \sigma} ) \big) \Big),
\end{align*}
and
\begin{align*}
 I_2=& \Big( p'_{\lambda} ( V^{\varepsilon_p} * u^{\varepsilon, \sigma} (s, \bar{X}_s^{i, \varepsilon, \sigma} ) - p_{\lambda}'  \Big( \frac{1}{N}  \sum_{j=1}^N  V^{\varepsilon_p}(\bar{X}_s^{i, \varepsilon, \sigma}  -\bar{X}_s^{j, \varepsilon, \sigma} )  \Big) \Big)  \frac{1}{N}\sum_{j=1}^N  \nabla V^{\varepsilon_p}(\bar{X}_s^{i, \varepsilon, \sigma}  -\bar{X}_s^{j, \varepsilon, \sigma} ) \\
 &+ \Big(  \nabla  V^{\varepsilon_p} * u^{\varepsilon, \sigma} (s, \bar{X}_s^{i, \varepsilon, \sigma}  ) -  \frac{1}{N}  \sum_{j=1}^N  \nabla V^{\varepsilon_p}(\bar{X}_s^{i, \varepsilon, \sigma}  -\bar{X}_s^{j, \varepsilon, \sigma} ) \Big)  p_{\lambda}' \big(  V^{\varepsilon_p} * u^{\varepsilon, \sigma} (s, \bar{X}_s^{i, \varepsilon, \sigma}  ) \big)\\
 =:&I_{2,1} +I_{2,2}.
  \end{align*}
The term  $I_1$ is bounded as follows
    \begin{align*}
    |I_1| &\leq C_{I_1}  \Big(  \big| \bar{X}_s^{i, \varepsilon, \sigma}  - X_s^{N,i, \varepsilon, \sigma} \big| + \frac{1}{N} \sum_{j=1}^N \big| X_s^{N,j, \varepsilon, \sigma} - \bar{X}_s^{j, \varepsilon, \sigma} \big| \Big),
    \end{align*}
where
\begin{align*}
C_{I_1} :=    \frac{\| \nabla V\|^2_{L^{\infty}(\R^d)}}{\varepsilon_p^{2(d+1)} }  \ \|p_{\lambda}''\|_{L^{\infty} (0, \| V^{\varepsilon_p}\|_{L^{\infty}(\R^d)})}  +  \|p_{\lambda}'\|_{L^{\infty} (0, \| V^{\varepsilon_p}\|_{L^{\infty}(\R^d)})}   \frac{\|D^2 V\|_{L^{\infty}(\R^d)}}{\varepsilon_p^{d+2}}.
\end{align*}
This gives immediately
 \begin{align*}
 \mathbb{E} \Big[ \int_0^t  2( X_s^{N,i, \varepsilon, \sigma} - \bar{X}_s^{i, \varepsilon, \sigma} )  I_1\, ds \Big] \leq  4 C_{I_1}
 \int_0^t \underset{i\in\{1, \ldots , N\}}{\max }\mathbb{E} \big[  \big| X_s^{N,i, \varepsilon, \sigma} - \bar{X}_s^{i, \varepsilon, \sigma} \big| ^2 \big]ds .
 \end{align*}
Next, we estimate the terms with $I_{2,1}$ and $I_{2,2}$.
  \begin{align*}
  & \big| \bar{X}_s^{i, \varepsilon, \sigma} -  X_s^{N,i, \varepsilon, \sigma} \big| \   |I_{2,1}|  \\
   \leq&  \big|\bar{X}_s^{i, \varepsilon, \sigma} -  X_s^{N,i, \varepsilon, \sigma} \big|  \frac{\|\nabla V \|_{L^{\infty}(\R^d)}}{\varepsilon_p^{d+1}}  \|p_{\lambda}''\|_{L^{\infty}(0, \|V^{\varepsilon_p} \|_{L^{\infty}(\R^d)}) }   \cdot \Big| V^{\varepsilon_p} * u^{\varepsilon, \sigma} (s, \bar{X}_s^{i, \varepsilon, \sigma} ) - \frac{1}{N}  \sum_{j=1}^N  V^{\varepsilon_p}(\bar{X}_s^{i, \varepsilon, \sigma}  -\bar{X}_s^{j, \varepsilon, \sigma} ) \Big|\\
   \leq& \frac{\|\nabla V \|_{L^{\infty}(\R^d)}}{\varepsilon_p^{d+1}}  \|p_{\lambda}''\|_{L^{\infty}(0, \|V^{\varepsilon_p} \|_{L^{\infty}(\R^d)}) } \left| \bar{X}_s^{i, \varepsilon, \sigma} -  X_s^{N,i, \varepsilon, \sigma} \right| \   \frac{1}{N}  \bigg| \sum_{j=1}^N \tilde{Z}_{ij} \bigg|\\
   \leq&    \big| \bar{X}_s^{i, \varepsilon, \sigma} -  X_s^{N,i, \varepsilon, \sigma} \big|^2 +  \frac{\|\nabla V \|_{L^{\infty}(\R^d)}^2}{4\varepsilon_p^{2(d+1)}}  \|p_{\lambda}''\|^2_{L^{\infty}(0, \|V^{\varepsilon_p} \|_{L^{\infty}(\R^d)}) }  \frac{1}{N^2}   \sum_{j=1}^N \sum_{k=1}^N \tilde{Z}_{ij} \tilde{Z}_{ik},
  \end{align*}
 where $\tilde{Z}_{ij} := V^{\varepsilon_p} * u^{\varepsilon, \sigma} (s, \bar{X}_s^{i, \varepsilon, \sigma} ) - V^{\varepsilon_p}(\bar{X}_s^{i, \varepsilon, \sigma}  -\bar{X}_s^{j, \varepsilon, \sigma} ) $.
  Similarly as in \eqref{indep_Z_aggreg} we deduce that
   \begin{align*}
    \mathbb{E} \Big[ \sum_{j=1}^N \sum_{k=1}^N \tilde{Z}_{ij} \tilde{Z}_{ik} \Big]  \leq  CN\frac{\|V\|^2_{L^{\infty}(\R^d)}}{\varepsilon_p^{2d}},
   \end{align*}
where the following fact has been used in the last estimate
   \begin{align*}
   \tilde{Z}_{ij}^2 &= \big( V^{\varepsilon_p} * u^{\varepsilon, \sigma} (s, \bar{X}_s^{i, \varepsilon, \sigma} ) - V^{\varepsilon_p}(\bar{X}_s^{i, \varepsilon, \sigma}  -\bar{X}_s^{j, \varepsilon, \sigma} )  \big)^2  \leq C  \frac{\|V\|^2_{L^{\infty}(\R^d)}}{\varepsilon_p^{2d}}. \ \
   \end{align*}
We obtain that
   \begin{align*}
  &  \mathbb{E} \Big[ \int_0^t 2(\bar{X}_s^{i, \varepsilon, \sigma} -  X_s^{N,i, \varepsilon, \sigma} )   I_{2,1} \, ds \Big]\\
   \leq&  \int_0^t \mathbb{E} \bigg[ \left| \bar{X}_s^{i, \varepsilon, \sigma} -  X_s^{N,i, \varepsilon, \sigma} \right|^2 \bigg] \, ds +  C   \frac{\|\nabla V \|_{L^{\infty}(\R^d)}^2}{\varepsilon_p^{2(2d+1)}}  \|p_{\lambda}''\|_{L^{\infty}(0, \|V^{\varepsilon_p} \|_{L^{\infty}(\R^d)}) }^2 \frac{1}{N}    \|V\|^2_{L^{\infty}(\R^d)}  t.
   \end{align*}

The procedure to estimate $\mathbb{E} \big[ \int_0^t 2(\bar{X}_s^{i, \varepsilon, \sigma} -  X_s^{N,i, \varepsilon, \sigma} )   I_{2,2} \, ds \big]$ is similar to the one used for the estimation of $\mathbb{E} \big[ \int_0^t 2(\bar{X}_s^{i, \varepsilon, \sigma} -  X_s^{N,i, \varepsilon, \sigma} )   I_{2,1} \, ds \big]$, namely we obtain that
\begin{align*}
&\mathbb{E} \Big[ \int_0^t 2(\bar{X}_s^{i, \varepsilon, \sigma} -  X_s^{N,i, \varepsilon, \sigma} )   I_{2,2} \, ds \Big]\\
\leq& \int_0^t \mathbb{E} \bigg[ \left| \bar{X}_s^{i, \varepsilon, \sigma} -  X_s^{N,i, \varepsilon, \sigma} \right|^2 \bigg] \, ds +      C\|p_{\lambda}'\|_{L^{\infty}(0, \| u^{\varepsilon, \sigma} \|_{L^{\infty}((0,T)\times \R^d)}) }^2 \frac{1}{N}    \frac{\|\nabla V\|^2_{L^{\infty}(\R^d)}}{\varepsilon_p^{2(d+1)}}t.
\end{align*}
Combining the above estimates we obtain that
\begin{align*}
\mathbb{E} \Big[ \int_0^t 2(\bar{X}_s^{i, \varepsilon, \sigma} -  X_s^{N,i, \varepsilon, \sigma} )   I_2\, ds \Big]  \leq  \int_0^t \mathbb{E} \bigg[ \left| \bar{X}_s^{i, \varepsilon, \sigma} -  X_s^{N,i, \varepsilon, \sigma} \right|^2 \bigg] \, ds  + C_{I_2},
\end{align*}
where
\begin{align*}
C_{I_2}&:=  \frac{C t  \|\nabla V \|_{L^{\infty}(\R^d)}^2}{N \varepsilon_p^{2(d+1)}}  \Big( \frac{1}{\varepsilon_p^{2d}}  \|p_{\lambda}''\|_{L^{\infty}(0, \|V^{\varepsilon_p} \|_{L^{\infty}(\R^d)}) }^2 \|V\|^2_{L^{\infty}(\R^d)}  +  \|p_{\lambda}'\|_{L^{\infty}(0, \| u^{\varepsilon, \sigma} \|_{L^{\infty}((0,T)\times \R^d)}) }^2  \Big) .
\end{align*}
Therefore the estimates for the non-linear diffusion term is
  \begin{align}
  \label{estimate_Porous_Medium}
   &\mathbb{E} \Big[  \int_0^t 2 (X_s^{N,i, \varepsilon, \sigma} - \bar{X}_s^{i, \varepsilon, \sigma})
   \Big( \nabla  p_{\lambda} \big(  V^{\varepsilon_p} * u^{\varepsilon, \sigma} (s, \bar{X}_s^{i, \varepsilon, \sigma}  )\big)  - \nabla  p_{\lambda}  \Big( \frac{1}{N} \sum_{j=1}^N  V^{\varepsilon_p}(X_s^{N,i, \varepsilon, \sigma} -X_s^{N,j, \varepsilon, \sigma} ) \Big) \Big)   \, ds \Big]  \notag\\
   \leq& (4C_{I_1} +2) \int_0^t \underset{i\in\{1, \ldots , N\}}{\max }\mathbb{E} \bigg[  \left| X_s^{N,i, \varepsilon, \sigma} - \bar{X}_s^{i, \varepsilon, \sigma} \right| ^2 \bigg]ds + C_{I_2}.
  \end{align}


Combining the estimates in \eqref{estimate_Keller_Segel}   and  \eqref{estimate_Porous_Medium} we obtain that
\begin{align*}
\underset{i\in\{1, \ldots , N\}}{\max }  \mathbb{E} \left[ \left| X_t^{N,i, \varepsilon, \sigma} - \bar{X}_t^{i, \varepsilon, \sigma} \right|^2 \right]
 \leq&  \frac{C }{\varepsilon_k^d} \int_0^t \underset{i\in\{1, \ldots , N\}}{\max }  \mathbb{E} \left[ \left| X_s^{N,i, \varepsilon, \sigma} - \bar{X}_s^{i, \varepsilon, \sigma} \right|^2 \right]   ds  + \frac{Ct}{ N \varepsilon_k^{d-2}} \\
&+ (4C_{I_1} +2) \int_0^t \underset{i\in\{1, \ldots , N\}}{\max }\mathbb{E} \left[  \left| X_s^{N,i, \varepsilon, \sigma} - \bar{X}_s^{i, \varepsilon, \sigma} \right| ^2 \right]ds + C_{I_2}\\
\leq& C_{G_1} \int_0^t \underset{i\in\{1, \ldots , N\}}{\max }\mathbb{E} \left[  \left| X_s^{N,i, \varepsilon, \sigma} - \bar{X}_s^{i, \varepsilon, \sigma} \right| ^2 \right]ds + C_{G_2},
\end{align*}
 and the constants are given in the following
  \begin{align*}
  C_{G_1} &:= \frac{C}{\varepsilon_k^d}  + 4 C_{I_1} +2 \\
  &=  \frac{C}{\varepsilon_k^d} + 4 \frac{\| \nabla V\|^2_{L^{\infty}(\R^d)}}{\varepsilon_p^{2(d+1)} } \|p''_\lambda\|_{L^{\infty}(0, \|V^{\varepsilon_p} \|_{L^{\infty}(\R^d)}) } +  4\|p'_\lambda\|_{L^{\infty}(0, \|V^{\varepsilon_p} \|_{L^{\infty}(\R^d)}) } \frac{\|D^2 V\|_{L^{\infty}(\R^d)}}{\varepsilon_p^{d+2}} +2,
  \end{align*}
and
\begin{align*}
C_{G_2} &:=  \frac{Ct}{ N \varepsilon_k^{d-2}} + C_{I_2}\\
& =\frac{Ct}{ N \varepsilon_k^{d-2}} +  \frac{C t  \|\nabla V \|_{L^{\infty}(\R^d)}^2}{N \varepsilon_p^{2(d+1)}}  \Big( \frac{1}{\varepsilon_p^{2d}}  \|p''_\lambda\|_{L^{\infty}(0, \|V^{\varepsilon_p} \|_{L^{\infty}(\R^d)}) }^2 \|V\|^2_{L^{\infty}(\R^d)}  +  \|p'_\lambda\|_{L^{\infty}(0, \| u^{\varepsilon, \sigma} \|_{L^{\infty}((0,T)\times \R^d)}) }^2  \Big).
\end{align*}
Now we apply Grönwall's inequality, which results in
  \begin{align*}
  \underset{t\in [0,T]}{\sup } \underset{i\in\{1, \ldots , N\}}{\max } \mathbb{E} \big[ \big| X_t^{N,i, \varepsilon, \sigma} - \bar{X}_t^{i, \varepsilon, \sigma} \big|^2 \big] \leq C_{G_2} \ \exp\big( C_{G_1} T\big).
\end{align*}
By choosing $\lambda = \frac{\varepsilon^d_p}{2}$, we obtain	
\begin{align*}
&\|p'_\lambda\|_{L^{\infty}(0, \|V^{\varepsilon_p} \|_{L^{\infty}(\R^d)}) } \leq  C(m) \|V^{\varepsilon_p} \|_{L^{\infty}(\R^d)}^{m-2}  = C \frac{\|V \|_{L^{\infty}(\R^d)}^{m-2}}{\varepsilon_p^{d(m-2)}}, \\
&\|p''_\lambda\|_{L^{\infty}(0, \|V^{\varepsilon_p} \|_{L^{\infty}(\R^d)}) } \leq  C(m) \|V^{\varepsilon_p} \|_{L^{\infty}(\R^d)}^{m-3}  = C \frac{\|V \|_{L^{\infty}(\R^d)}^{m-3}}{\varepsilon_p^{d(m-3)}},
\end{align*}
which consequently implies that
\begin{align*}
& C_{G_1} \leq C\Big(  1+ \frac{1}{\varepsilon_k^d} +  \frac{1}{\varepsilon_p^{2(d+1)} }  \ \frac{1}{\varepsilon_p^{d(m-3)} } +  \frac{1}{\varepsilon_p^{d(m-2)} } \frac{1}{\varepsilon_p^{d+2}} \Big) \leq C \Big( \frac{1}{\varepsilon_k^d} +\frac{1}{\varepsilon_p^{dm-d+2}}  \Big)   \\
 &C_{G_2} \leq C \Big( \frac{t}{ N \varepsilon_k^{d-2}} +  \frac{ t }{N \varepsilon_p^{2(d+1)}}  \Big( \frac{1}{\varepsilon_p^{2d}}  \frac{1}{\varepsilon_p^{2d(m-3)} }   +  1 \Big) \Big) \leq C \Big( \frac{t}{ N \varepsilon_k^{d-2}} + \frac{t}{ N \varepsilon_p^{2dm -2d+2}} \Big)
\end{align*}
for $m\geq 3$ or $m=2$. Therefore, the following estimate holds
\begin{align*}
  \underset{t\in [0,T]}{\sup } \underset{i\in\{1, \ldots , N\}}{\max } \mathbb{E} \big[ \big| X_t^{N,i, \varepsilon, \sigma} - \bar{X}_t^{i, \varepsilon, \sigma} \big|^2 \big] \leq CT \Big( \frac{1}{ N \varepsilon_k^{d-2}} + \frac{1}{ N \varepsilon_p^{2dm -2d+2}} \Big)  \ \exp\Big( CT \Big( \frac{1}{\varepsilon_k^d} +\frac{1}{\varepsilon_p^{dm-d+2}}  \Big) \Big).
\end{align*}
It remains to prove that for $0<\beta<1$ it holds that
\begin{align*}
  CT \Big( \frac{1}{ N \varepsilon_k^{d-2}} + \frac{1}{ N \varepsilon_p^{2dm -2d+2}} \Big)  \ \exp\Big( CT \Big( \frac{1}{\varepsilon_k^d} +\frac{1}{\varepsilon_p^{dm-d+2}}  \Big) \Big) = O(N^{-\beta}) \ \text{as} \ N \rightarrow \infty.
  \end{align*}
For $\varepsilon_k = \big( \alpha_k{\ln N} \big)^{-\frac{1}{d}}$ and $\varepsilon_p = \big(\alpha_p \ln N \big)^{-\frac{1}{dm-d+2}} $  where $\alpha_k, \alpha_p >0$,
it holds
\begin{align*}
\frac{1}{N \varepsilon_k^{d-2}} \exp\Big(CT \Big( \frac{1}{\varepsilon_k^d} + \frac{1}{\varepsilon_p^{dm-d+2}}\Big)  \Big) &= (\alpha_k)^{\frac{d-2}{d}} (\ln N)^{\frac{d-2}{d}}N^{-1 + CT ( \alpha_k +  \alpha_p)}
\end{align*}
and
\begin{align*}
\frac{1}{N \varepsilon_p^{2dm-2d+2}} \exp\Big( CT \Big( \frac{1}{\varepsilon_k^d} + \frac{1}{\varepsilon_p^{dm-d+2}}\Big)  \Big) &= (\alpha_p)^{\frac{2dm-2d+2}{dm-d+2}} (\ln N)^{\frac{2dm-2d+2}{dm-d+2}} N^{-1  + CT (\alpha_k + \alpha_p) }.
\end{align*}
Since $\frac{d-2}{d} \leq \frac{2dm-2d+2}{dm-d+2}$, and for arbitrary $0<\delta\ll 1$ it holds that $\ln N\leq N^\delta$ for big enough $N$, it is enough to choose $\alpha_k$ and $\alpha_p$ such that
\begin{align*}
-1 + \delta \cdot  \frac{2dm-2d+2}{dm-d+2} + CT (\alpha_k +  \alpha_p)   <-\beta,
\end{align*}
which is exactly the assumption for $\alpha_k$ and $\alpha_p$ in Proposition \ref{theorem_log}.
\end{proof}

\begin{proof}
	[Proof of Theorem \ref{coroll_X_N_and_X_hat} (a)] The result follows directly from Propositions \ref{lemma_X_bar_X_hat} and \ref{theorem_log} by using the following triangle inequality,
	$$
	| X_t^{N,i, \varepsilon, \sigma} - \hat{X}_t^{i, \varepsilon, \sigma} |\leq | X_t^{N,i, \varepsilon, \sigma} - \bar{X}_t^{i, \varepsilon, \sigma} | + | \bar{X}_t^{i, \varepsilon, \sigma} - \hat{X}_t^{i, \varepsilon, \sigma}|, \quad\forall i\in\{1,\ldots, N\}.
	$$
\end{proof}

\section{Convergence in Probability with the Algebraic cut-off}
\label{Section_conv_in_prob}
In this section, we prove the key technical results in this paper. Since the logarithmic cut-off scaling obtained in the previous section is too weak to describe the singularity of the interaction potential in the particle level, an algebraic cut-off should be considered to capture  the singularity in the microscopic level. In this meaning, a convergence result in the expectation is too strong to be expected. Instead, we will prove the convergence in the sense of probability.

\begin{proposition}[Error estimates between \eqref{generalized_regularized_particle_model} and \eqref{generalized_intermediate_particle_model} in the algebraic scaling]
	\label{theorem_algebraic}
	Assume $\frac{d}{2}+1\leq \vartheta\leq d$, $\varepsilon_k \geq N^{-\beta_k}$, $\varepsilon_p \geq(\ln N)^{-\frac{1}{2+d(2+|m-3|)}}$, $\lambda = \frac{\varepsilon_p^d}{2}$, $0<\beta_k<\frac{1}{2(\vartheta+1)}$, and $X_t^{N,i, \varepsilon, \sigma} $ and $\bar{X}_t^{i, \varepsilon, \sigma}$ are solutions of \eqref{generalized_regularized_particle_model} and  \eqref{generalized_intermediate_particle_model} respectively,
	then for an arbitrary $\gamma>0$ there exists a constant $C(\gamma)>0$, which is independent of $N$, such that
	\begin{align*}
		&\sup_{0\leq t\leq T} \mathbb{P} \Big( \max_{i \in\{1, \ldots , N\} } \left| X_t^{N,i, \varepsilon, \sigma} - \bar{X}_t^{i, \varepsilon, \sigma} \right| >N^{-a} \Big) \leq C(\gamma) N^{-\gamma},
	\end{align*}
	where $a<\frac12$.
\end{proposition}

In order to prove this proposition, we need the following three lemmas.
The first one is a special version of the law of large numbers.
Although its proof can be found in \cite[Section 3]{CHJ}, we present the full proof for the convenience of the reader.
\begin{lemma} [A special version of the law of large numbers]
	\label{lemma_estim_prob_bad_sets}
	Let $(\bar{Y}^i)_{i \in \{1, \ldots, N \}}$ be a collection of i.i.d. random  variable with the  density function $v\in L^1(\R^d)$, $U \in L^{\infty}(\R^d) $ be a given function.  Define
	\begin{align*}
		&\mathcal{A}_{\theta}^i(U,v) := \Big\{ \omega \in \Omega :  \Big| \frac{1}{N} \sum_{j =1}^N U( \bar{Y}^i- \bar{Y}^j ) - U *v(\bar{Y}^i)   \Big| > \frac{1}{N^{\theta}}  \Big\},\\
		&\mathcal{A}_{\theta}^N(U,v) := \bigcup_{i=1}^N\mathcal{A}_{\theta}^i(U,v).
	\end{align*}
Then for arbitrary $\tilde{\kappa}\in \N$  and $\theta\in (0, \frac{1}{2})$, it holds
	\begin{align*}
		\mathbb{P}(\mathcal{A}_{\theta}^N(U,v)) \leq N \max_{i \in \{ 1, \ldots N \}}\mathbb{P}(\mathcal{A}_{\theta}^i(U,v)) \leq N^{2 \tilde{\kappa} (\theta-\frac{1}{2}) +1} C(\tilde{\kappa}) \left(  \|  U  \|_{ L^{\infty}(\R^d)}^{2\tilde{\kappa}} +  \|  U *v \|_{L^{\infty}(\R^d)}^{2\tilde{\kappa}} \right).
	\end{align*}
\end{lemma}

\begin{proof} First of all, the Markov's inequality implies that
	\begin{align*}
		\mathbb{P}(\mathcal{A}_{\theta}^i(U,v))  \leq N^{2\tilde{\kappa}\theta} \mathbb{E} \Big[ \Big| \frac{1}{N} \sum_{j =1}^N U( \bar{Y}^i- \bar{Y}^j ) - U *v(\bar{Y}^i)  \Big|^{2\tilde{\kappa}} \Big] =  N^{2\tilde{\kappa}\theta} \mathbb{E} \Big[ \Big(   \frac{1}{N^2}  \sum_{j,l =1}^N  h (\bar{Y}^i, \bar{Y}^j)  h (\bar{Y}^i, \bar{Y}^l) \Big)^{\tilde{\kappa}} \Big],
	\end{align*}
	where $\displaystyle  h (\bar{Y}^i, \bar{Y}^j):=   U( \bar{Y}^i- \bar{Y}^j ) - U *v(\bar{Y}^i) $.
	We now look at terms in $\mathbb{E}\big[\big(\sum_{j,l = 1}^N h (\bar{Y}^i, \bar{Y}^j) h (\bar{Y}^i, \bar{Y}^l) \big)^{\tilde{\kappa}}\big]$ where one index $j\in \{1,\ldots,N\}$ only appears once:
	\begin{align*}
		&\mathbb{E}\bigg[ h (\bar{Y}^i, \bar{Y}^j) \prod_{\stackrel{m=1}{\ell_m\not= j}}^{2\tilde{\kappa}-1}h (\bar{Y}^i, \bar{Y}^{\ell_m})\bigg] =\int_{\R^d} v(x)   \int_{\R^d} v(y) h(x,y) \, dy \,  \mathbb{E} \bigg[ \prod_{\stackrel{m=1}{\ell_m\not= j}}^{2\tilde{\kappa}-1}h (x, \bar{Y}^{\ell_m})  \bigg]  \, d x =0,
	\end{align*}
where we have used the fact that
	\begin{align*}
		\int_{\R^d} v(y) h(x,y) \, dy
		&=  \int_{\R^d} v(y) U(x-y ) \, dy - \int_{\R^d }\int_{\R^d} v(y) U(x-z)v(z) \, dz   \, dy=0.
	\end{align*}
	In order to estimate $\mathbb{E}\big[\big(\sum_{j,l = 1}^N h (\bar{Y}^i, \bar{Y}^j) h (\bar{Y}^i, \bar{Y}^l) \big)^{\tilde{\kappa}}\big]$, we need an upper bound for the terms which do not vanish. These terms have the following form:
	\begin{align*}
		\mathcal{N}:= \Big\{\prod_{j=1}^{2\tilde{\kappa}} h (\bar{Y}^i, \bar{Y}^{i_j}) :~i_j \in \{1,\ldots,N\}\text{  such that all appearing indices $i_j$ those appear at least twice}\Big\}.
	\end{align*}
	So, the number of elements in $\mathcal{N}$ is bounded by
	$C(\tilde \kappa)N^{\tilde{\kappa}}$. Together with the estimate
	$$
	\mathbb{E}\Big(\prod_{j=1}^{2\tilde{\kappa}} h (\bar{Y}^i, \bar{Y}^{i_j}) \Big) \leq C(\tilde\kappa)\left(  \|  U  \|_{ L^{\infty}(\R^d)}^{2\tilde{\kappa}} +  \|  U *v \|_{ L^{\infty}(\R^d)}^{2\tilde{\kappa}} \right),
	$$
	we obtain
	\begin{align*}
		\mathbb{P}(\mathcal{A}_{\theta}^i(U))  &\leq    N^{2\tilde{\kappa}\theta} \mathbb{E} \bigg[ \Big(   \frac{1}{N^2}  \sum_{j,l =1}^N  h (\bar{Y}^i, \bar{Y}^j) h (\bar{Y}^i, \bar{Y}^l) \Big)^{\tilde{\kappa}} \bigg]  \\
		&\leq N^{2 \tilde{\kappa} \theta } \frac{1}{N^{2\tilde{\kappa}}} C(\tilde{\kappa}) N^{\tilde{\kappa}} \big(  \|  U  \|_{ L^{\infty}(\R^d)}^{2\tilde{\kappa}} +  \|  U *v \|_{ L^{\infty}(\R^d)}^{2\tilde{\kappa}} \big)\\
		&= N^{2 \tilde{\kappa} (\theta-\frac{1}{2}) } C(\tilde{\kappa}) \big(  \|  U  \|_{ L^{\infty}(\R^d)}^{2\tilde{\kappa}} +  \|  U *v \|_{ L^{\infty}(\R^d)}^{2\tilde{\kappa}} \big).
	\end{align*}
	This implies the desired result.
\end{proof}

The following uniform in $\varepsilon$ estimates are important in the discussion of this section.
\begin{lemma}
\label{add_regular_u_eps_sigma}
For $w \in L^1(\R^d)\cap L^\infty(\R^d)$, $\frac{d}{2}+1\leq\vartheta<d$, there exists constant $C$ independent of $\varepsilon$ such that
\begin{align*}
\big\|  |D^2\Phi_\vartheta^{\varepsilon}| * w \big\|_{L^{\infty}(\R^d)}+\big\|  |\nabla \Phi_\vartheta^{\varepsilon}| * w\big\|_{L^{\infty}(\R^d)}+\big\|  |\nabla \Phi_d^{\varepsilon}| * w\big\|_{L^{\infty}(\R^d)}\leq C,
\end{align*}
and furthermore, we have
\begin{align*}
\big\|  |D^2\Phi_d^{\varepsilon}| * w \big\|_{L^{\infty}(\R^d)}\leq -C\ln \varepsilon,
\end{align*}
\end{lemma}
\begin{proof}
	Due to the definition of $\Phi_\vartheta$, we have the following estimate
	\begin{align*}
	&\Big|	\int_{\R^d}|D^2\Phi^\varepsilon_\vartheta|(z)w(x-z)dz\Big|\leq \int_{\R^d} \int_{\R^d}V^\varepsilon (y)|D^2\Phi_\vartheta(z-y)| |w(x-z)|dydz\\
	\leq &\int_{\R^d}\int_{\R^d} V^\varepsilon (y)\frac{C(d,\vartheta)}{|z-y|^\vartheta}|w(x-z)|dydz =\int_{\R^d}\int_{\R^d} V^\varepsilon (y)\frac{C(d,\vartheta)}{|z-y|^\vartheta}(\mathbb{I}_{|z-y|\leq 1}+\mathbb{I}_{|y-z|>1})|w(x-z)|dydz\\
	\leq & C \|V^\varepsilon\|_{L^1(\R^d)}\|w\|_{L^\infty(\R^d)} + C\|V^\varepsilon\|_{L^1(\R^d)}\|w\|_{L^1(\R^d)}\leq C.
	\end{align*}
The estimate for $\big\|  |\nabla \Phi_\vartheta^{\varepsilon}| * w\big\|_{L^{\infty}(\R^d)}$ and $\big\|  |\nabla \Phi_d^{\varepsilon}| * w\big\|_{L^{\infty}(\R^d)}$ can be obtained similarly.
Further the definition of $\Phi^\varepsilon_d$ implies that for small $\varepsilon$,
\begin{align*}
&\Big|	\int_{\R^d}|D^2\Phi^\varepsilon_d|(z)w(x-z)dz\Big|\\
\leq& \int_{\R^d} \int_{|y|>1}V^\varepsilon (z-y)|D^2\bar\Phi_d(y)| |w(x-z)|dydz
 +\int_{\R^d} \int_{\varepsilon<|y|\leq 1}V^\varepsilon (z-y)|D^2\bar\Phi_d(y)| |w(x-z)|dydz\\
\leq &C(d)\|V^\varepsilon\|_{L^1(\R^d)}\|w\|_{L^1(\R^d)}+\int_{\R^d} \int_{\varepsilon<|y|\leq 1}V^\varepsilon (z-y)\frac{C(d)}{|y|^d} |w(x-z)|dydz\\
\leq & C-C \|V^\varepsilon\|_{L^1(\R^d)}\|w\|_{L^\infty(\R^d)} \ln \varepsilon \leq -C\ln\varepsilon,
\end{align*}
where we have used $\||y|^{-d}\mathbb{I}_{\{\varepsilon<|y|\le 1 \}}\|_{L^1(\mathbb{R}^d)}=C\int^1_\varepsilon \frac{r^{d-1}}{r^{d}}dr=-C\ln \varepsilon$.
\end{proof}

The following estimates for the mollified interaction potential will be used frequently in the rest of this section.
\begin{lemma}
	\label{lemma_L_infty_estimates_phi_theta}
	There exists a positive constant $C>0$ which depends only on the given mollification kernel $V$ and $d$ such that, for any $\frac{d}{2}+1\le\nu\le d$,
	\begin{align*}
	\|\nabla   \Phi_{\vartheta}^{\varepsilon} \|_{L^{\infty}(\R^d)} \leq \frac{C}{\varepsilon^{\vartheta-1}}, \ \ \ \ \|D^2   \Phi_{\vartheta}^{\varepsilon} \|_{L^{\infty}(\R^d)} \leq    \frac{C}{\varepsilon^\vartheta},  \ \ \ \ \|D^3   \Phi_{\vartheta}^{\varepsilon} \|_{L^{\infty}(\R^d)} \leq    \frac{C}{\varepsilon^{\vartheta+1}}.
	\end{align*}
\end{lemma}
\begin{proof}
	Since $\nabla \Phi_{\vartheta} = C_{d,\vartheta}\frac{x}{|x|^\vartheta}$, we obtain that there exists a constant $C>0$ such that
	\begin{align*}
	\int_{|y|< \varepsilon} \frac{1}{|x-y|^{\vartheta-1}}    dy &=   \int_{ {\substack{|y|< \varepsilon \\|x-y|>\varepsilon }} } \frac{1}{|x-y|^{\vartheta-1}}    dy + \int_{ {\substack{|y|< \varepsilon \\|x-y|<\varepsilon }} } \frac{1}{|x-y|^{\vartheta-1}}    dy  \\
	&\leq C \Big( \frac{1}{\varepsilon^{\vartheta-1}} Vol(|y|< \varepsilon) +  \int_{|x-y|< \varepsilon} \frac{1}{|x-y|^{\vartheta-1}} d(x-y)\Big)\\
	&\leq C   \Big( \frac{\varepsilon^d}{\varepsilon^{\vartheta-1}} + \int_0^{\varepsilon} \frac{r^{d-1}}{r^{\vartheta-1}} dr \Big)\leq C  \varepsilon^{d-\vartheta+1},
	\end{align*}
	which implies for $\frac{d}{2}+1\leq\vartheta< d$ that
	\begin{align*}
	\|\nabla  \Phi_\vartheta*V^{\varepsilon} \|_{L^{\infty}(\R^d)}
	\leq \frac{C}{\varepsilon^d} \Big\|\int_{|y|< \varepsilon} \frac{1}{|x-y|^{\vartheta-1}} \Big| V \Big( \frac{y}{\varepsilon} \Big) \Big|   dy\Big\|_{L^{\infty}(\R^d)} \leq  C \frac{ \| V \|_{L^{\infty}(\R^d)}  }{\varepsilon^d}  \varepsilon^{d-\vartheta+1} \leq \frac{C}{\varepsilon^{\vartheta-1}},
	\end{align*}
and for $\vartheta= d$
\begin{align*}
\|\nabla  \Phi_d^{\varepsilon} \|_{L^{\infty}(\R^d)} \leq
\|\nabla  \Phi_d\ast V^{\varepsilon} \|_{L^{\infty}(\R^d)}	
+\|(\nabla  \Phi_d\mathbb{I}_{B_\varepsilon(0)})\ast V^{\varepsilon} \|_{L^{\infty}(\R^d)}
\leq \frac{C}{\varepsilon^{d-1}}.
\end{align*}
	Using the same method we deduce that
	\begin{align*}
	\|D^2  \Phi_\vartheta^{\varepsilon} \|_{L^{\infty}(\R^d)}  \leq  C \frac{1}{\varepsilon^{d+1}} \Big\|\int_{|y|< \varepsilon} \frac{1}{|x-y|^{\vartheta-1}} \Big|\nabla  V \Big( \frac{y}{\varepsilon} \Big) \Big|   dy\Big\|_{L^{\infty}(\R^d)}\leq \frac{C}{\varepsilon^{\vartheta}},
	\end{align*}
	and
	\begin{align*}
	\|D^3  \Phi_\vartheta^{\varepsilon} \|_{L^{\infty}(\R^d)} \leq C \frac{1}{\varepsilon^{d+2}} \Big\|\int_{|y|< \varepsilon} \frac{1}{|x-y|^{\vartheta-1}} \Big|D^2 V \Big( \frac{y}{\varepsilon} \Big) \Big|   dy\Big\|_{L^{\infty}(\R^d)}\leq \frac{C}{\varepsilon^{\vartheta+1}}.
	\end{align*}
\end{proof}

\begin{proof}[Proof of Proposition \ref{theorem_algebraic}]
We consider, for an appropriate  $a\in (0, \frac{1}{2})$ to be chosen later, the following stopping time and the stopped process
\begin{align*}
&\tau(\omega)= \inf \Big\{ t\in (0,T) \Big|  \max_{i \in \{1, \ldots, N \} } \Big| X_t^{N,i, \varepsilon, \sigma} - \bar{X}_t^{i, \varepsilon, \sigma} \Big| \geq N^{-a}  \Big\}, \\
&  S_t(\omega)= N^{ a \kappa } \max_{i \in\{1, \ldots , N\}} \left| X_{t\wedge \tau}^{N,i, \varepsilon, \sigma} - \bar{X}_{t\wedge \tau}^{i, \varepsilon, \sigma} \right|^{\kappa}, \quad \mbox{ where } \kappa>0 \mbox{ is an arbitrary given number}.
\end{align*}
Thanks to the stopping time, we have $S_t\le 1$.
By applying Markov's inequality we get that
\begin{align*}
&\sup_{0\leq t\leq T} \mathbb{P} \Big( \max_{i \in\{1, \ldots , N\}} \big| X_t^{N,i, \varepsilon, \sigma} - \bar{X}_t^{i, \varepsilon, \sigma} \big| >N^{-a} \Big) \leq  \sup_{0\leq t\leq T}  \mathbb{P} \left( S_t =1 \right) \leq \sup_{0\leq t\leq T} \mathbb{E} \left[ S_t \right].
\end{align*}
The main task in the following is to estimate $\sup_{0\leq t\leq T}\mathbb{E}[S_t]$. Within this proof we use $C$ as a generic constant {which} depends only on $\kappa$, $m$, and $T$, which might change line by line.

Recall the problems \eqref{generalized_regularized_particle_model} and \eqref{generalized_intermediate_particle_model}, since the initial data of $X_t^{N,i, \varepsilon, \sigma} $ and $\bar{X}_t^{i, \varepsilon, \sigma}$ are given by the same random variables, we deduce that
  \begin{align*}
 &\mathbb{E}[S_t] =N^{ a \kappa } \mathbb{E}\Big[\max_{i \in\{1, \ldots , N\}} \left| X_{t\wedge \tau}^{N,i, \varepsilon, \sigma} - \bar{X}_{t\wedge \tau}^{i, \varepsilon, \sigma} \right|^{\kappa} \Big] \leq  C \mathbb{E}[I_1 + I_2 + II_1 + II_2],
 \end{align*}
where
 \begin{align*}
 &I_1 =   N^{a} \max_{i \in\{1, \ldots , N\}} \int_0^{t\wedge \tau} \Big| \frac{1}{N} \sum_{j =1}^N  \big( \nabla \Phi_\vartheta^{\varepsilon_k}(X_s^{N,i, \varepsilon, \sigma} -X_s^{N,j, \varepsilon, \sigma}) - \nabla \Phi_\vartheta^{\varepsilon_k}( \bar{X}_s^{i, \varepsilon, \sigma}- \bar{X}_s^{j, \varepsilon, \sigma} )   \big)    \Big|\cdot S_s^\frac{\kappa-1}{\kappa}   ds, \\
&I_2 =   N^{a}  \max_{i \in\{1, \ldots , N\}} \int_0^{t\wedge \tau} \Big| \frac{1}{N} \sum_{j =1}^N  \nabla \Phi_\vartheta^{\varepsilon_k}( \bar{X}_s^{i, \varepsilon, \sigma}- \bar{X}_s^{j, \varepsilon, \sigma} ) - \nabla  \Phi_\vartheta^{\varepsilon_k}  *u^{\varepsilon, \sigma} ( s, \bar{X}_s^{i, \varepsilon, \sigma})     \Big|\cdot S_s^\frac{\kappa-1}{\kappa}   ds,\\
&II_1 =   N^{a}  \max_{i \in\{1, \ldots , N\}}   \int_0^{t\wedge \tau} \Big| \nabla p_{\lambda} \Big(  \frac{1}{N} \sum_{j=1}^N  V^{\varepsilon_p}(X_s^{N,i, \varepsilon, \sigma} -X_s^{N,j, \varepsilon, \sigma} ) \Big)  - \nabla p_{\lambda} \Big(  \frac{1}{N} \sum_{j=1}^N  V^{\varepsilon_p}( \bar{X}_s^{i, \varepsilon, \sigma}- \bar{X}_s^{j, \varepsilon, \sigma}  ) \Big)   \Big|\cdot S_s^\frac{\kappa-1}{\kappa} ds, \\
&II_2 =   N^{a } \max_{i \in\{1, \ldots , N\}}   \int_0^{t\wedge \tau} \Big| \nabla p_{\lambda} \Big(  \frac{1}{N} \sum_{j=1}^N  V^{\varepsilon_p}( \bar{X}_s^{i, \varepsilon, \sigma}- \bar{X}_s^{j, \varepsilon, \sigma}  ) \Big)  - \nabla p_{\lambda}\big(  V^{\varepsilon_p} * u^{\varepsilon, \sigma} (s, \bar{X}_s^{i, \varepsilon, \sigma} ) \big) \Big|\cdot S_s^\frac{\kappa-1}{\kappa}  ds.
 \end{align*}

Now we estimate the right-hand side in expectation term by term.

\underline{\textit{Step 1 Estimation of the aggregation terms $\mathbb{E}[I_1]+\mathbb{E}[I_2]$}.}

Due to the lack of Lipschitz continuity of $\nabla \Phi_\vartheta^{\varepsilon_k}$, the key idea in this step is to expand the difference in $I_1$ at the point $\bar{X}_s^{i, \varepsilon, \sigma}- \bar{X}_s^{j, \varepsilon, \sigma}$ by using Taylor's expansion. And then using the law of large numbers from Lemma \ref{lemma_estim_prob_bad_sets} to estimate $ D^2  \Phi_\vartheta^{\varepsilon_k}( \bar{X}_s^{i, \varepsilon, \sigma}- \bar{X}_s^{j, \varepsilon, \sigma} ) $ as a replacement of the Lipschitz constant. More precisely,  we have
\begin{align*}
I_1 &\leq C  N^{a }  \max_{i \in\{1, \ldots , N\}} \int_0^{t\wedge \tau} \Big| \frac{1}{N} \sum_{j =1}^N   D^2  \Phi_\vartheta^{\varepsilon_k}( \bar{X}_s^{i, \varepsilon, \sigma}- \bar{X}_s^{j, \varepsilon, \sigma} )    \left( X_s^{N,i, \varepsilon, \sigma} - \bar{X}_s^{i, \varepsilon, \sigma}  - X_s^{N,j, \varepsilon, \sigma} + \bar{X}_s^{j, \varepsilon, \sigma}      \right) \Big|\cdot S_s^\frac{\kappa-1}{\kappa}   ds\\
&\quad + C N^{a }  \max_{i \in\{1, \ldots , N\}} \int_0^{t\wedge \tau}  \frac{1}{N}  \sum_{j =1}^N  \| D^3 \Phi_\vartheta^{\varepsilon_k} \|_{L^{\infty}( \R^d)} \left| X_s^{N,i, \varepsilon, \sigma} - \bar{X}_s^{i, \varepsilon, \sigma}  - X_s^{N,j, \varepsilon, \sigma} + \bar{X}_s^{j, \varepsilon, \sigma}   \right|^{2} \cdot S_s^\frac{\kappa-1}{\kappa}  ds\\
&\leq CN^{a } \int_0^{t\wedge \tau} \max_{i \in\{1, \ldots , N\}} \Big( \frac{1}{N} \sum_{j =1}^N  \left| D^2  \Phi_\vartheta^{\varepsilon_k}( \bar{X}_s^{i, \varepsilon, \sigma}- \bar{X}_s^{j, \varepsilon, \sigma} ) \right|  \Big)    \max_{j \in\{1, \ldots , N\}}  \left| X_s^{N,j, \varepsilon, \sigma} - \bar{X}_s^{j, \varepsilon, \sigma}   \right| \cdot S_s^\frac{\kappa-1}{\kappa}   ds \\
&\quad + C \int_0^{t\wedge \tau}  \| D^3 \Phi_\vartheta^{\varepsilon_k} \|_{L^{\infty}( \R^d)}  \max_{i \in\{1, \ldots , N\}}  \left| X_s^{N,i, \varepsilon, \sigma} - \bar{X}_s^{i, \varepsilon, \sigma}    \right|  S_s^\frac{1}{\kappa}  \cdot S_s^\frac{\kappa-1}{\kappa}  ds\\
&\leq C\int_0^{t\wedge \tau} \max_{i \in\{1, \ldots , N\}} \Big( \frac{1}{N} \sum_{j =1}^N  \left| D^2  \Phi_\vartheta^{\varepsilon_k}( \bar{X}_s^{i, \varepsilon, \sigma}- \bar{X}_s^{j, \varepsilon, \sigma} ) \right|  \Big)    S_s ds \\
&\quad + C  N^{-a }   \int_0^{t\wedge \tau}  \| D^3 \Phi_\vartheta^{\varepsilon_k} \|_{L^{\infty}( \R^d)}   S_s   ds.
\end{align*}
For the first integral, we estimate its expectation by using Lemma \ref{lemma_estim_prob_bad_sets} in the following
\begin{align*}
&\mathbb{E}\bigg[\int_0^{t\wedge \tau} \max_{i \in\{1, \ldots , N\}} \Big( \frac{1}{N} \sum_{j =1}^N  \left| D^2  \Phi_\vartheta^{\varepsilon_k}( \bar{X}_s^{i, \varepsilon, \sigma}- \bar{X}_s^{j, \varepsilon, \sigma} ) \right|  \Big)   S_s ds\bigg]\\
 \leq&\, \mathbb{E}\bigg[ \int_0^{t\wedge \tau} \max_{i \in\{1, \ldots , N\}}  \left|  |D^2 \Phi_\vartheta^{\varepsilon_k} | *u^{\varepsilon, \sigma} ( s , \bar{X}_s^{i, \varepsilon, \sigma})    \right| S_s ds \bigg]  \\
&+  \mathbb{E}\bigg[ \int_0^{t\wedge \tau} \max_{i \in\{1, \ldots , N\}}  \Big| \frac{1}{N} \sum_{j =1}^N   |D^2  \Phi_\vartheta^{\varepsilon_k}( \bar{X}_s^{i, \varepsilon, \sigma}- \bar{X}_s^{j, \varepsilon, \sigma} )| - |D^2 \Phi_\vartheta^{\varepsilon_k}|  *u^{\varepsilon, \sigma} ( s , \bar{X}_s^{i, \varepsilon, \sigma})    \Big|S_s ds \bigg]\\
\leq&\, \mathbb{E}\bigg[ \int_0^{t\wedge \tau} \max_{i \in\{1, \ldots , N\}}  \left|  |D^2 \Phi_\vartheta^{\varepsilon_k} | *u^{\varepsilon, \sigma} ( s , \bar{X}_s^{i, \varepsilon, \sigma})    \right| S_s ds \bigg]  \\
&+  \mathbb{E}\bigg[ \int_0^{t\wedge \tau} \max_{i \in\{1, \ldots , N\}}  \Big| \frac{1}{N} \sum_{j =1}^N   |D^2  \Phi_\vartheta^{\varepsilon_k}( \bar{X}_s^{i, \varepsilon, \sigma}- \bar{X}_s^{j, \varepsilon, \sigma} )| - |D^2 \Phi_\vartheta^{\varepsilon_k}|  *u^{\varepsilon, \sigma} ( s , \bar{X}_s^{i, \varepsilon, \sigma})    \Big|\mathbb{I}_{(\mathcal{A}_{0}^{(1)})^c(s)}S_s ds \bigg]\\
&+  \mathbb{E}\bigg[ \int_0^{t\wedge \tau} \max_{i \in\{1, \ldots , N\}}  \Big| \frac{1}{N} \sum_{j =1}^N   |D^2  \Phi_\vartheta^{\varepsilon_k}( \bar{X}_s^{i, \varepsilon, \sigma}- \bar{X}_s^{j, \varepsilon, \sigma} )| - |D^2 \Phi_\vartheta^{\varepsilon_k}|  *u^{\varepsilon, \sigma} ( s , \bar{X}_s^{i, \varepsilon, \sigma})    \Big|\mathbb{I}_{\mathcal{A}_{0}^{(1)}(s)}S_s ds \bigg]\\
\leq& C \int_0^t \Big( \big\|  |D^2\Phi_\vartheta^{\varepsilon_k}| * u^{\varepsilon, \sigma}  \big\|_{L^{\infty}(0,T; L^{\infty}(\R^d))} +1 \Big) \mathbb{E}[S_s] ds \\
&+C \Big(  \|  D^2 \Phi_\vartheta^{\varepsilon_k}   \|_{L^{\infty}(\R^d)}  + \big\|  |D^2 \Phi_\vartheta^{\varepsilon_k}| * u^{\varepsilon, \sigma}  \big\|_{ L^{\infty} ( 0,T; L^{\infty}(\R^d))} \Big)  \int_0^t  \mathbb{P}( \mathcal{A}_{0}^{(1)}(s) ) ds,
\end{align*}
where in Lemma \ref{lemma_estim_prob_bad_sets} the i.i.d. variables are chosen as
\begin{align*}
\mathcal{A}_{0}^{(1)}(s) := \mathcal{A}_{0}^N(|D^2  \Phi_\vartheta^{\varepsilon_k} (\cdot)|,u^{\varepsilon, \sigma}(s, \cdot))\ \text{with} \ (\bar{Y}^i)_{i \in \{1, \ldots, N \}} = (\bar{X}_s^{i, \varepsilon, \sigma})_{i \in \{1, \ldots, N \}}.
\end{align*}
So, together with Lemma \ref{add_regular_u_eps_sigma}, we obtain the following bound for $\mathbb{E}[I_1]$:
\begin{align}
\label{I1}
\mathbb{E} \left[ I_1 \right]
\leq& C \Big( -\ln \varepsilon_{k}  + \| D^3 \Phi_\vartheta^{\varepsilon_k} \|_{L^{\infty}( \R^d)}   \frac{1}{N^{ a}}  \Big) \int_0^t \mathbb{E}\left[  S_s \right] ds   + C    \|  D^2 \Phi_\vartheta^{\varepsilon_k}   \|_{L^{\infty}(\R^d)}      \int_0^t  \mathbb{P}( \mathcal{A}_{0}^{(1)}(s) )    ds,
\end{align}
where the constant with $-\ln \varepsilon_{k}$ is only needed for the case $\vartheta=d$, while for $\frac{d}{2}+1\leq\vartheta<d$ it is a uniform constant.

The estimate for $\mathbb{E}[I_2]$ is a direct application of Lemma \ref{lemma_estim_prob_bad_sets} by choosing
\begin{align*}
\mathcal{A}_{\theta_{(2)}}^{(2)}(s) := \mathcal{A}_{\theta_{(2)}}^N(\nabla  \Phi_\vartheta^{\varepsilon_k} (\cdot),u^{\varepsilon, \sigma}(s, \cdot))\ \text{with} \ (\bar{Y}^i)_{i \in \{1, \ldots, N \}} = (\bar{X}_s^{i, \varepsilon, \sigma})_{i \in \{1, \ldots, N \}}.
\end{align*}
Namely, we have
\begin{align}
\label{I2}
\mathbb{E}[I_2]  \leq& \int_0^t \mathbb{E}\left[  S_s \right] ds \nonumber\\
& + C N^{a \kappa }\mathbb{E} \Big[ \max_{i \in\{1, \ldots , N\}}   \int_0^{t\wedge \tau}  \Big| \frac{1}{N} \sum_{j =1}^N  \nabla \Phi_\vartheta^{\varepsilon_k}( \bar{X}_s^{i, \varepsilon, \sigma}- \bar{X}_s^{j, \varepsilon, \sigma} ) - \nabla  \Phi_\vartheta^{\varepsilon_k}  *u^{\varepsilon, \sigma} ( s, \bar{X}_s^{i, \varepsilon, \sigma})     \Big|^\kappa ds\Big] \nonumber\\
\leq& \int_0^t \mathbb{E}\left[  S_s \right] ds+C N^{(a -\theta_{(2)} ) \kappa} +  CN^{a \kappa}   \|  \nabla \Phi_\vartheta^{\varepsilon_k} \|_{ L^{\infty}(\R^d)}^{\kappa}  \int_0^t  \mathbb{P}(   \mathcal{A}_{\theta_{(2)}}^{(2)}(s)  ) ds.
\end{align}
Therefore, by choosing $\varepsilon_k \geq N^{- \beta_k}$ in \eqref{I1} and \eqref{I2} and applying Lemma \ref{lemma_L_infty_estimates_phi} we derive the estimate for the aggregation term
\begin{align}
\label{main_est_aggreg_term}
\mathbb{E} [I_1+I_2  ] \notag 
\leq& C \big( 1+\ln N+ N^{(\vartheta+1)\beta_k - a } \big) \int_0^t \mathbb{E}\left[  S_s \right] ds
+ C N^{ \kappa (a  -  \theta_{(2)} )  } \notag \\
&+ C  N^{\vartheta\beta_k }    \int_0^t  \mathbb{P}( \mathcal{A}_{0}^{(1)}(s) )   ds
 + C    N^{ \kappa ( (\vartheta-1) \beta_k +a )}    \int_0^t \mathbb{P}(   \mathcal{A}_{\theta_{(2)}}^{(2)}(s)  ) ds,
\end{align}
where we have used the estimates in Lemma \ref{lemma_L_infty_estimates_phi_theta}.

\underline{\textit{Step 2 Estimation of the porous medium terms $\mathbb{E}[II_1]+\mathbb{E}[II_2]$ .}}

A further decomposition $II_1$ by triangle inequality implies that
\begin{align}
\label{II1}
II_{1} &\leq C N^{\alpha}\max_{i\in\{1,\ldots,N \}} \Big(  \int_0^{t \wedge \tau} |A_{1,1}| \cdot S_s^\frac{\kappa-1}{\kappa} ds
+ \int_0^{t \wedge \tau} |A_{1,2}| \cdot S_s^\frac{\kappa-1}{\kappa} ds \Big),
\end{align}
where
\begin{align*}
   A_{1,1}=& \Big( p'_{\lambda}\Big( \frac{1}{N}\sum_{j=1}^N  V^{\varepsilon_p}(\bar{X}_s^{i, \varepsilon, \sigma}  -\bar{X}_s^{j, \varepsilon, \sigma} )  \Big)  - p'_{\lambda}\Big( \frac{1}{N} \sum_{j=1}^N  V^{\varepsilon_p}(X_s^{N,i, \varepsilon, \sigma} -X_s^{N,j, \varepsilon, \sigma} ) \Big)  \Big)
   \frac{1}{N}\sum_{j=1}^N  \nabla V^{\varepsilon_p}(\bar{X}_s^{i, \varepsilon, \sigma}  -\bar{X}_s^{j, \varepsilon, \sigma} ),  \\
   A_{1,2}=& \Big( \frac{1}{N}\sum_{j=1}^N  \big(\nabla V^{\varepsilon_p}(\bar{X}_s^{i, \varepsilon, \sigma}  -\bar{X}_s^{j, \varepsilon, \sigma} )
    -    \nabla V^{\varepsilon_p}(X_s^{N,i, \varepsilon, \sigma} -X_s^{N,j, \varepsilon, \sigma} )\big)  \Big)
    \cdot   p'_{\lambda}\Big( \frac{1}{N} \sum_{j=1}^N  V^{\varepsilon_p}(X_s^{N,i, \varepsilon, \sigma} -X_s^{N,j, \varepsilon, \sigma} ) \Big).
\end{align*}
It can be exactly estimated like before in the following
\begin{align}\label{estim_A_1}
\mathbb{E}[II_{1}] &\leq C_{I_1} \int_0^t \mathbb{E}\left[  S_s \right] ds
\end{align}
where
\begin{align*}
C_{I_1} :=    \frac{\| \nabla V\|^2_{L^{\infty}(\R^d)}}{\varepsilon_p^{2(d+1)} }  \ \|p_{\lambda}''\|_{L^{\infty} (0, \| V^{\varepsilon_p}\|_{L^{\infty}(\R^d)})}  +  \|p_{\lambda}'\|_{L^{\infty} (0, \| V^{\varepsilon_p}\|_{L^{\infty}(\R^d)})}   \frac{\|D^2 V\|_{L^{\infty}(\R^d)}}{\varepsilon_p^{d+2}}\leq C \varepsilon_p^{-2-d(2+|m-3|)}.
\end{align*}
The term $II_2$ includes only the i.i.d. random processes $\{\bar{X}_s^{i, \varepsilon, \sigma}\}_{i=1}^N$, therefore the main technique is the domain decomposition of $\Omega$ and the application of the law of large numbers in Lemma \ref{lemma_estim_prob_bad_sets}.
\begin{align}
\label{II2}\mathbb{E}[II_{2}] &\leq \int_0^t \mathbb{E}\left[  S_s \right] ds\\
\nonumber &\quad + N^{a\kappa } \max_{i \in\{1, \ldots , N\}}   \int_0^{t\wedge \tau} \Big| \nabla p_{\lambda} \Big(  \frac{1}{N} \sum_{j=1}^N  V^{\varepsilon_p}( \bar{X}_s^{i, \varepsilon, \sigma}- \bar{X}_s^{j, \varepsilon, \sigma}  ) \Big)  - \nabla p_{\lambda}\big(  V^{\varepsilon_p} * u^{\varepsilon, \sigma} (s, \bar{X}_s^{i, \varepsilon, \sigma} ) \big) \Big|^{\kappa}  ds\\
&\nonumber\leq \int_0^t \mathbb{E}\left[  S_s \right] ds+
C\mathbb{E} \Big[ N^{a \kappa} \max_{i \in\{1, \ldots , N\}} \int_0^{t\wedge \tau} |A_{2,1}|^{\kappa} ds \Big]
+C\mathbb{E} \Big[ N^{a \kappa} \max_{i \in\{1, \ldots , N\}} \int_0^{t\wedge \tau} |A_{2,2}|^{\kappa} ds \Big],
\end{align}
where
\begin{align}
&A_{2,1}:= \bigg( p'_{\lambda} ( V^{\varepsilon_p} * u^{\varepsilon, \sigma} (s, \bar{X}_s^{i, \varepsilon, \sigma} ) - p_{\lambda}'  \Big( \frac{1}{N}  \sum_{j=1}^N  V^{\varepsilon_p}(\bar{X}_s^{i, \varepsilon, \sigma}  -\bar{X}_s^{j, \varepsilon, \sigma} )  \Big) \bigg)  \frac{1}{N}\sum_{j=1}^N  \nabla V^{\varepsilon_p}(\bar{X}_s^{i, \varepsilon, \sigma}  -\bar{X}_s^{j, \varepsilon, \sigma} ), \\
&A_{2,2}:= \Big(  \nabla  V^{\varepsilon_p} * u^{\varepsilon, \sigma} (s, \bar{X}_s^{i, \varepsilon, \sigma}  ) -  \frac{1}{N}  \sum_{j=1}^N  \nabla V^{\varepsilon_p}(\bar{X}_s^{i, \varepsilon, \sigma}  -\bar{X}_s^{j, \varepsilon, \sigma} ) \Big)  p_{\lambda}' \big(  V^{\varepsilon_p} * u^{\varepsilon, \sigma} (s, \bar{X}_s^{i, \varepsilon, \sigma}  ) \big).
\end{align}

We apply Lemma \ref{lemma_estim_prob_bad_sets} and the mean value theorem in the following
\begin{align}\label{A21}
&\mathbb{E} \Big[ N^{a \kappa} \max_{i \in\{1, \ldots , N\}} \int_0^{t\wedge \tau} |A_{2,1}|^{\kappa} ds \Big] \notag\\
\leq& C  \varepsilon_p^{-d \kappa |m-3|} \varepsilon_p^{-(d+1)\kappa} \cdot N^{a \kappa}  \cdot  \int_0^t \mathbb{E} \Big[   \max_{i \in\{1, \ldots , N\}} \Big|  \frac{1}{N}\sum_{j=1}^N  V^{\varepsilon_p}(\bar{X}_s^{i, \varepsilon, \sigma}  -\bar{X}_s^{j, \varepsilon, \sigma} )    -   V^{\varepsilon_p} * u^{\varepsilon, \sigma} (s, \bar{X}_s^{i, \varepsilon, \sigma}  )  \Big|^{\kappa} \Big] ds  \notag\\
\leq& C  \Big( \varepsilon_p^{-d |m-3|-(d+1)} \cdot  N^{a}\Big)^\kappa \Big( N^{-\theta_{(3)} \kappa} +  \varepsilon_p^{- d \kappa} \int_0^t  \mathbb{P}(\mathcal{A}_{ \theta_{(3)} }^{(3)}(s) ) ds  \Big),
\end{align}
where $\mathcal{A}_{\theta_{(3)}}^{(3)}(s) := \mathcal{A}_{\theta_{(3)}}^N(  V^{\varepsilon_p} (\cdot),u^{\varepsilon, \sigma}(s, \cdot))$.
Similarly, since $\| p_{\lambda}' \big(  V^{\varepsilon_p} * u^{\varepsilon, \sigma})\|_{L^\infty(0,T;L^\infty(\R^d))}\leq C$ we have
\begin{align}\label{A22}
\mathbb{E} \Big[ N^{a \kappa} \max_{i \in\{1, \ldots , N\}} \int_0^{t\wedge \tau} |A_{2,2}|^{\kappa} ds \Big]
\leq C   N^{a \kappa} \Big( N^{-\theta_{(4)} \kappa} +  \varepsilon_p^{- (d+1)\kappa} \int_0^t \mathbb{P}( \mathcal{A}_{\theta_{(4)}}^{(4)}(s) )  ds \Big),
\end{align}
where $\mathcal{A}_{\theta_{(4)}}^{(4)}(s) := \mathcal{A}_{ \theta_{(4)} }^N(\nabla  V^{\varepsilon_p} (\cdot),u^{\varepsilon, \sigma}(s, \cdot))$.

Plugging \eqref{A21}-\eqref{A22} into \eqref{II2}, we infer that
\begin{align}
\label{estim_A_2}
\mathbb{E} [II_2] \leq & \int_0^t \mathbb{E}\left[  S_s \right] ds+C \bigg( \Big( \varepsilon_p^{-d |m-3|-(d+1)} \cdot  N^{a-\theta_{(3)}}\Big)^\kappa+  N^{ \kappa( a -\theta_{(4)} )  }  \bigg) \notag \\
&+ C    \Big( \varepsilon_p^{-d |m-3|-(2d+1)} \cdot  N^{a}\Big)^\kappa  \int_0^t \mathbb{P}(\mathcal{A}_{ \theta_{(3)} }^{(3)}(s) )ds
+ C N^{a \kappa}  \varepsilon_p^{- (d+1)\kappa} \int_0^t \mathbb{P}( \mathcal{A}_{\theta_{(4)}}^{(4)}(s) )  ds.
\end{align}
Combining \eqref{estim_A_1} and \eqref{estim_A_2}, we derive the estimate for the non-linear term

\begin{align}
\label{estim_porous_med_conv_in_prob}
& \mathbb{E} \Big[ II_1+ II_2\Big]\leq  C\varepsilon_p^{-2-d(2+|m-3|)}\int_0^t \mathbb{E}\left[  S_s \right] ds+C \bigg( \Big( \varepsilon_p^{-d |m-3|-(d+1)} \cdot  N^{a-\theta_{(3)}}\Big)^\kappa+  N^{ \kappa( a -\theta_{(4)} )  }  \bigg) \notag \\
&+ C    \Big( \varepsilon_p^{-d |m-3|-(2d+1)} \cdot  N^{a}\Big)^\kappa  \int_0^t \mathbb{P}(\mathcal{A}_{ \theta_{(3)} }^{(3)}(s) )ds
+ C N^{a \kappa}  \varepsilon_p^{- (d+1)\kappa} \int_0^t \mathbb{P}( \mathcal{A}_{\theta_{(4)}}^{(4)}(s) )  ds.
\end{align}

\underline{\textit{Step 3 Final estimation for $\mathbb{E}[S_t]$. }}

Combining estimates from steps 1 and 2 we obtain that
\begin{align}
\label{final_estimate_con_in_prob}
\mathbb{E}[S_t] \leq & C \big( \ln N+ N^{(\vartheta+1)\beta_k - a } +\varepsilon_p^{-2-d(2+|m-3|)}\big) \int_0^t \mathbb{E}\left[  S_s \right] ds\nonumber\\
&+ C \bigg(N^{ \kappa (a  -  \theta_{(2)} )  }  + \Big( \varepsilon_p^{-d |m-3|-(d+1)} \cdot  N^{a-\theta_{(3)}}\Big)^\kappa+  N^{ \kappa( a -\theta_{(4)} )  } \bigg) \nonumber\\
&+ C  N^{\vartheta\beta_k }    \int_0^t  \mathbb{P}( \mathcal{A}_{0}^{(1)}(s) )   ds
+ C    N^{ \kappa ( (\vartheta-1) \beta_k +a )}    \int_0^t \mathbb{P}(   \mathcal{A}_{\theta_{(2)}}^{(2)}(s)  ) ds\nonumber\\
&+ C    \Big( \varepsilon_p^{-d |m-3|-(2d+1)} \cdot  N^{a}\Big)^\kappa  \int_0^t \mathbb{P}(\mathcal{A}_{ \theta_{(3)} }^{(3)}(s) )ds
+ C N^{a \kappa}  \varepsilon_p^{- (d+1)\kappa} \int_0^t \mathbb{P}( \mathcal{A}_{\theta_{(4)}}^{(4)}(s) )  ds.
\end{align}
In order to use Gr\"onwall's inequality, we need to fix the relation among $\beta_k$, $\varepsilon_p$, and $a$ such that the inequality implies that $\sup_{0\leq t\leq T}\mathbb{E}[S_t]$ is bounded uniformly. Namely we need to make sure that the coefficients of $\int_0^t  \mathbb{E} [S_s] ds$ have to be bounded at most by $\ln N$, the constant terms in the last line have to give arbitrary rate, and the probabilities of ``bad sets'' $\mathcal{A}$ are small enough.

The following restrictions have to be true
\begin{align}\label{restr_a_keller_segel_E}
&(\vartheta+1)\beta_k \leq  a, \quad \varepsilon_p^{-2-d(2+|m-3|)}\leq \ln N, \quad a< \theta_{(2)}< \frac{1}{2}, \quad a<\theta_{(3)}< \frac{1}{2}, \quad a< \theta_{(4)}< \frac{1}{2}.
\end{align}
Then for an arbitrary $\tilde\gamma>0$  we can choose $\kappa>0$ such that
\begin{align*}
N^{ \kappa (a  -  \theta_{(2)} )  }  + \Big( \varepsilon_p^{-d |m-3|-(d+1)} \cdot  N^{a-\theta_{(3)}}\Big)^\kappa+  N^{ \kappa( a -\theta_{(4)} )  }
\leq C N^{-\tilde \gamma}.
\end{align*}
For the probabilities of ``bad sets'' $\mathcal{A}$, we apply Lemma \ref{lemma_estim_prob_bad_sets} and obtain by choosing $\varepsilon_k= N^{-\beta_k}$ that
\begin{align*}
&C  N^{\vartheta\beta_k }    \int_0^t  \mathbb{P}( \mathcal{A}_{0}^{(1)}(s) )   ds
+ C    N^{ \kappa ( (\vartheta-1) \beta_k +a )}    \int_0^t \mathbb{P}(   \mathcal{A}_{\theta_{(2)}}^{(2)}(s)  ) ds
\nonumber\\
\nonumber&+ C    \Big( \varepsilon_p^{-d |m-3|-(2d+1)} \cdot  N^{a}\Big)^\kappa  \int_0^t \mathbb{P}(\mathcal{A}_{ \theta_{(3)} }^{(3)}(s) )ds
+ C N^{a \kappa}  \varepsilon_p^{- (d+1)\kappa} \int_0^t \mathbb{P}( \mathcal{A}_{\theta_{(4)}}^{(4)}(s) )  ds\\
\leq&C N^{\vartheta \beta_k }  \cdot N^{2 \tilde{\kappa}_{(1)} ( -\frac{1}{2} + \vartheta \beta_k) +1 }
  +  C N^{ \kappa ( (\vartheta-1) \beta_k +a )}  N^{2 \tilde{\kappa}_{(2)} ( ( \theta_{(2)}-\frac{1}{2}) +(\vartheta-1)  \beta_k )  +1 }\\
  &+ C    \Big( \varepsilon_p^{-d |m-3|-(2d+1)} \cdot  N^{a}\Big)^\kappa  N^{2 \tilde{\kappa}_{(3)} (  \theta_{(3)}-\frac{1}{2})   +1 }\varepsilon_p^{-2 \tilde{\kappa}_{(3)}d}
  + C N^{a \kappa}  \varepsilon_p^{- (d+1)\kappa} N^{2 \tilde{\kappa}_{(4)} (  \theta_{(4)}-\frac{1}{2})   +1 }\varepsilon_p^{-2 \tilde{\kappa}_{(4)}(d+1)}.
\end{align*}
These gives further restrictions of $\beta_k$ and $a$ and the possible choices of $\theta_{(2)} $-$\theta_{(4)}$ such that
\begin{align}
&0< \frac{1}{2} - \vartheta \beta_k,&
&0<\theta_{(2)}<\frac{1}{2}-(\vartheta-1)  \beta_k,&
&0< \theta_{(3)} < \frac{1}{2}, &0< \theta_{(4)} < \frac{1}{2}.\label{rest_thetas_and_a_keller_segel_prob}
\end{align}
With these conditions, for the above fixed $\kappa$, we can choose $\tilde{\kappa}_{(1)}$-$\tilde{\kappa}_{(4)}$ big enough such that the term on ``bad'' sets are bounded by $ C N^{-\tilde\gamma}$.

Obviously, the above restrictions imply that $\beta_k$ and $a$ have to be chosen with
\begin{equation}\label{main_condition_conv_in_prob}
(\vartheta+1)\beta_k \leq  a<\frac12.
\end{equation}
As a summary, if the condition \eqref{main_condition_conv_in_prob} holds, then we have
\begin{align*}
\mathbb{E}[S_t] \leq C\ln N \int_0^t \mathbb{E}[S_s] \, ds + C N^{-\tilde\gamma}.
\end{align*}
Consequently, Gr\"onwall's inequality yields
\begin{align*}
\sup_{0\leq t\leq T} \mathbb{E}[S_t] \leq C N^{-\tilde\gamma+C(T)} .
\end{align*}
Taking $\tilde\gamma=\gamma+C(T)$, together with Markov’s inequality, it holds, for those $a$ in \eqref{main_condition_conv_in_prob}, that
\begin{align}\label{PN}
\sup_{0\leq t\leq T} \mathbb{P} \Big( \max_{i \in\{1, \ldots , N\}} \big| X_t^{N,i, \varepsilon, \sigma} - \bar{X}_t^{i, \varepsilon, \sigma} \big| >N^{-a} \Big) \leq
C N^{-\gamma}.
\end{align}
Actually for any $\tilde{a}<a$, it is easy to see
\begin{align*}
\sup_{0\leq t\leq T} \mathbb{P} \Big( \max_{i \in\{1, \ldots , N\}} \big| X_t^{N,i, \varepsilon, \sigma} - \bar{X}_t^{i, \varepsilon, \sigma} \big|
>N^{-\tilde{a}} \Big) \leq
\sup_{0\leq t\leq T} \mathbb{P} \Big( \max_{i \in\{1, \ldots , N\}} \big| X_t^{N,i, \varepsilon, \sigma} - \bar{X}_t^{i, \varepsilon, \sigma} \big| >N^{-a} \Big) \leq
C N^{-\gamma}.
\end{align*}
Similarly, we have
\begin{align*}
\sup_{0\leq t\leq T} \mathbb{P} \Big( \max_{i \in\{1, \ldots , N\}} \big| X_t^{N,i, \varepsilon, \sigma} - \bar{X}_t^{i, \varepsilon, \sigma} \big|
>(\ln N)^{-\frac{1}{4+2d(2+|m-3|)}} \Big) \leq
C (\varepsilon_k+\varepsilon_p)^\frac{1}{2}
{\color{red} \leq C \varepsilon_p^\frac{1}{2}.}
\end{align*}
\end{proof}

\begin{proof}
	[Proof of Theorem \ref{coroll_X_N_and_X_hat} (b)] The convergence in expectation implies the convergence in the sense of probability. Therefore, from Proposition \ref{lemma_X_bar_X_hat} we obtain
	\begin{align*}
	&\sup_{0\leq t\leq T}\mathbb{P}\Big(\max_{i\in\{1,\ldots,N\}}|\bar{X}_t^{i, \varepsilon, \sigma} -\hat{X}_t^{i, \sigma}|>
    \varepsilon_p^\frac{1}{2} \Big)\\
	\leq&  \varepsilon_p^{-\frac{1}{2}}\sup_{0\leq t\leq T}\mathbb{E}\Big(\max_{i\in\{1,\ldots,N\}}|\bar{X}_t^{i, \varepsilon, \sigma}
    -\hat{X}_t^{i, \sigma}| \Big)\\
	\leq& C \varepsilon_p^{-\frac{1}{2}}(\varepsilon_k+\varepsilon_p)
	\leq C(\varepsilon_k+\varepsilon_p)^\frac{1}{2}{\color{red} \leq C \varepsilon_p^\frac{1}{2},}
	\end{align*}
	which together with the results in Proposition \ref{theorem_algebraic} gives directly the result in Theorem \ref{coroll_X_N_and_X_hat} $(b)$.
\end{proof}

\section{Propagation of chaos in the strong sense}\label{propagationstrong}
In this section, we derive the quantitative propagation of chaos results in the strong sense by the relative entropy and interpolation inequality.
The letter $C$ appeared in this section is a positive constant independent of $N$, $\varepsilon_k$ and $\varepsilon_p$.
 
\begin{theorem}[Quantitative propagation of chaos result in $L^1(\R^{dl})$-norm]
	\label{relativeL1}
	Under the assumptions of Theorem  \ref{coroll_X_N_and_X_hat} (b),
	we suppose that $u^{\varepsilon,\sigma}_{N,l}(t,x_1,\cdots,x_l)$ $(l\in\N)$ is the $l-$th marginal density of the joint density $u^{\varepsilon,\sigma}_N(t,x_1,\cdots,x_N)$ of $(X^{N,i,\varepsilon,\sigma}_{t})_{1\leq i\leq N}$,
	then there exists a positive constant $C$ independent of $N$, $\varepsilon_k$ and $\varepsilon_p$ such that
	\begin{align*}
		\|u^{\varepsilon,\sigma}_{N,l}-{u^\sigma}^{\otimes l}\|_{L^\infty(0,T;L^1(\mathbb{R}^{dl}))}\leq C(l)\varepsilon_p^{\frac{4}{d+4}}.
	\end{align*}
\end{theorem}
A similar argument can be applied to infer that this result holds for Theorem  \ref{coroll_X_N_and_X_hat} (a) as well.

\begin{proof}[Proof of Theorem \ref{relativeL1}]
The Kolmogorov forward equation of \eqref{generalized_regularized_particle_model} reads as
\begin{align}
\label{relativeuN}
\begin{cases}
&\partial_t u_N^{\varepsilon,\sigma}=\sigma\displaystyle\sum_{i=1}^N\Delta_{x_i}u^{\varepsilon,\sigma}_N
-\displaystyle\sum_{i=1}^N\nabla_{x_i}\cdot \Big(\frac{1}{N}\sum_{j=1}^N\nabla\Phi^{\varepsilon_k}_\vartheta(x_i-x_j)u^{\varepsilon,\sigma}_N \Big) \\
&\qquad\qquad+\displaystyle\sum_{i=1}^N\nabla_{x_i}\cdot\Big(\nabla p_\lambda\Big(\frac{1}{N}\sum_{j=1}^NV^{\varepsilon_p}(x_i-x_j) \Big)u^{\varepsilon,\sigma}_N \Big),     \\
&u_N^{\varepsilon,\sigma}(0,x_1,\cdots,x_N)={u_0^\sigma}^{\otimes N}=u_0^\sigma(x_1)\cdots u_0^\sigma(x_N),
\end{cases}
\end{align}
where $u_N^{\varepsilon,\sigma}(t,x_1,\cdots,x_N)$ is the joint law of $N$ particles $(X_t^{N,i,\varepsilon,\sigma})_{1\leq i\leq N}$.
For fixed $\varepsilon$ and $\sigma$, the global existence and uniqueness of a classical solution to the linear parabolic problem \eqref{relativeuN} can be obtained
as a trivial consequence of classical parabolic theory.
The key to derive a quantitative convergence result is the error estimate between $u_N^{\varepsilon,\sigma}$ and the product measure with density ${u^{\varepsilon,\sigma}}^{\otimes N}$, where $u^{\varepsilon,\sigma}$ is the solution to \eqref{generalized_equation_u_epsilon_sigma}.
It is easy to see
\begin{align}
\label{relativeuvare}
\begin{cases}
&\partial_t {u^{\varepsilon,\sigma}}^{\otimes N}=\sigma\displaystyle\sum_{i=1}^N\Delta_{x_i}{u^{\varepsilon,\sigma}}^{\otimes N}
-\displaystyle\sum_{i=1}^N\nabla_{x_i}\cdot \big(\nabla\Phi^{\varepsilon_k}_\vartheta\ast u^{\varepsilon,\sigma}(x_i){u^{\varepsilon,\sigma}}^{\otimes N} \big) \\
&\qquad\qquad+\displaystyle\sum_{i=1}^N\nabla_{x_i}\cdot\big(\nabla p_\lambda\big(V^{\varepsilon_p}\ast u^{\varepsilon,\sigma}(x_i) \big){u^{\varepsilon,\sigma}}^{\otimes N} \big),     \\
&{u^{\varepsilon,\sigma}}^{\otimes N}(0,x_1,\cdots,x_N)=u_0^\sigma(x_1)\cdots u_0^\sigma(x_N).
\end{cases}
\end{align}
For $l\in\N$ and two probability density functions $f$ and $g$ on $\R^{dl}$, let
\begin{align*}
\mathcal{H}_l(f|g)
:=\int_{\R^{dl}}f\log\frac{f}{g}\,dx_1\cdots dx_l.
\end{align*}
We define the relative entropy between $u^{\varepsilon,\sigma}_N$ and ${u^{\varepsilon,\sigma}}^{\otimes N}$ as follows
\begin{align*}
\mathcal{H}(u_N^{\varepsilon,\sigma}|{u^{\varepsilon,\sigma}}^{\otimes N}):=\frac{1}{N}\mathcal{H}_N(u_N^{\varepsilon,\sigma}|{u^{\varepsilon,\sigma}}^{\otimes N})(t)
=\frac{1}{N}\int_{\R^{dN}}u_N^{\varepsilon,\sigma}\log\frac{u_N^{\varepsilon,\sigma}}{{u^{\varepsilon,\sigma}}^{\otimes N}}\,dx_1\cdots dx_N.
\end{align*}
From \eqref{relativeuN} and \eqref{relativeuvare}, we deal with the evolution of the relative entropy as follows:
\begin{align}\label{rela1}
&\frac{d}{dt}\mathcal{H}(u_N^{\varepsilon,\sigma}|{u^{\varepsilon,\sigma}}^{\otimes N})\nonumber\\
=&\frac{\sigma}{N}\sum_{i=1}^N\int_{\R^{dN}}\Big(-\nabla_{x_i}u^{\varepsilon,\sigma}_N+\frac{u^{\varepsilon,\sigma}_N}{{u^{\varepsilon,\sigma}}^{\otimes N}}\nabla_{x_i}{u^{\varepsilon,\sigma}}^{\otimes N} \Big)
\cdot \nabla_{x_i}\log \frac{u^{\varepsilon,\sigma}_N}{{u^{\varepsilon,\sigma}}^{\otimes N}} dx_1\cdots dx_N\nonumber\\
&+\frac{1}{N}\sum_{i=1}^N\int_{\R^{dN}}\Big(\frac{1}{N}\sum_{j=1}^N\nabla\Phi^{\varepsilon_k}_\vartheta(x_i-x_j)
-\nabla\Phi^{\varepsilon_k}_\vartheta\ast u^{\varepsilon,\sigma}(x_i) \Big)u^{\varepsilon,\sigma}_N\cdot\nabla_{x_i}\log \frac{u^{\varepsilon,\sigma}_N}{{u^{\varepsilon,\sigma}}^{\otimes N}}
dx_1\cdots dx_N\nonumber\\
&-\frac{1}{N}\sum_{i=1}^N\int_{\R^{dN}}\Big(\nabla p_\lambda \Big(\frac{1}{N}\sum_{j=1}^N V^{\varepsilon_p}(x_i-x_j)\Big)
-\nabla p_\lambda(V^{\varepsilon_p}\ast u^{\varepsilon,\sigma}(x_i))  \Big)u^{\varepsilon,\sigma}_N\cdot \nabla_{x_i}\log\frac{u^{\varepsilon,\sigma}_N}{{u^{\varepsilon,\sigma}}^{\otimes N}}
dx_1\cdots dx_N\nonumber\\
\le& -\frac{\sigma}{2N}\sum_{i=1}^N\int_{\R^{dN}} u^{\varepsilon,\sigma}_N\Big|\nabla_{x_i}\log\frac{u^{\varepsilon,\sigma}_N}{{u^{\varepsilon,\sigma}}^{\otimes N}} \Big|^2 dx_1\cdots dx_N
\nonumber\\
&+\frac{C}{N}\sum_{i=1}^N\int_{\R^{dN}}\Big|\frac{1}{N}\sum_{j=1}^N\nabla\Phi^{\varepsilon_k}_\vartheta(x_i-x_j)
-\nabla\Phi^{\varepsilon_k}_\vartheta\ast u^{\varepsilon,\sigma}(x_i) \Big|^2 u^{\varepsilon,\sigma}_N dx_1\cdots dx_N\nonumber\\
&+\frac{C}{N}\sum_{i=1}^N\int_{\R^{dN}}\Big|\nabla p_\lambda\Big(\frac{1}{N}\sum_{j=1}^N V^{\varepsilon_p}(x_i-x_j) \Big)
-\nabla p_\lambda(V^{\varepsilon_p}\ast u^{\varepsilon,\sigma}(x_i)) \Big|^2 u^{\varepsilon,\sigma}_N dx_1\cdots dx_N\nonumber\\
=:& -\frac{\sigma}{2N}\sum_{i=1}^N \int_{\R^{dN}}u^{\varepsilon,\sigma}_N\Big|\nabla_{x_i}\log\frac{u^{\varepsilon,\sigma}_N}{{u^{\varepsilon,\sigma}}^{\otimes N}} \Big|^2 dx_1\cdots dx_N
+I_1+I_2.
\end{align}
The term $I_1$ can be divided into three parts
\begin{align}\label{rela2}
I_1\le& C\mathbb{E}\Big[\frac{1}{N}\sum_{i=1}^N\Big|\frac{1}{N}\sum_{j=1}^N\nabla\Phi^{\varepsilon_k}_\vartheta (X^{N,i,\varepsilon,\sigma}_t-X^{N,j,\varepsilon,\sigma}_t) 
-\frac{1}{N}\sum_{j=1}^N\nabla\Phi^{\varepsilon_k}_\vartheta(\bar{X}^{i,\varepsilon,\sigma}_t-\bar{X}^{j,\varepsilon,\sigma}_t)\Big|^2 \Big]\nonumber\\
&+C\mathbb{E}\Big[\frac{1}{N}\sum_{i=1}^N\Big|\frac{1}{N}\sum_{j=1}^N\nabla\Phi^{\varepsilon_k}_\vartheta(\bar{X}^{i,\varepsilon,\sigma}_t-\bar{X}^{j,\varepsilon,\sigma}_t)
-\nabla\Phi^{\varepsilon_k}_\vartheta\ast u^{\varepsilon,\sigma}(\bar{X}^{i,\varepsilon,\sigma}_t)\Big|^2 \Big]\nonumber\\
&+C\mathbb{E}\Big[\frac{1}{N}\sum_{i=1}^N|\nabla\Phi^{\varepsilon_k}_\vartheta\ast u^{\varepsilon,\sigma}(\bar{X}^{i,\varepsilon,\sigma}_t)
-\nabla\Phi^{\varepsilon_k}_\vartheta\ast u^{\varepsilon,\sigma}(X^{N,i,\varepsilon,\sigma}_t)|^2 \Big]\nonumber\\
=:&I_{11}+I_{12}+I_{13}.
\end{align}
It follows from Proposition \ref{theorem_algebraic} that, for an arbitrary $\gamma>0$, there exists a constant $C(\gamma)>0$ independent of $N$ such that
$\sup_{0\le t\le T}\mathbb{P}(\mathcal{A}(t))\le C(\gamma)N^{-\gamma}$, where 
\begin{align*}
\mathcal{A}(t)=\Big\{w\in \Omega:\max_{1\le i\le N}|X^{N,i,\varepsilon,\sigma}_t-\bar{X}^{i,\varepsilon,\sigma}_t|>N^{-a} \Big\}\quad{\mbox{with }} 
0<a<\frac{1}{2}.
\end{align*}
This, together with Lemma \ref{lemma_L_infty_estimates_phi_theta}, implies that
\begin{align*}
I_{11}\le&C\mathbb{E}\Big[\frac{1}{N}\sum_{i=1}^N\Big|\frac{1}{N}\sum_{j=1}^N\nabla\Phi^{\varepsilon_k}_\vartheta (X^{N,i,\varepsilon,\sigma}_t-X^{N,j,\varepsilon,\sigma}_t) 
-\frac{1}{N}\sum_{j=1}^N\nabla\Phi^{\varepsilon_k}_\vartheta(\bar{X}^{i,\varepsilon,\sigma}_t-\bar{X}^{j,\varepsilon,\sigma}_t)\Big|^2\mathbb{I}_{\mathcal{A}(t)} \Big]\\
&+C\mathbb{E}\Big[\frac{1}{N}\sum_{i=1}^N\Big|\frac{1}{N}\sum_{j=1}^N\nabla\Phi^{\varepsilon_k}_\vartheta (X^{N,i,\varepsilon,\sigma}_t-X^{N,j,\varepsilon,\sigma}_t) 
-\frac{1}{N}\sum_{j=1}^N\nabla\Phi^{\varepsilon_k}_\vartheta(\bar{X}^{i,\varepsilon,\sigma}_t-\bar{X}^{j,\varepsilon,\sigma}_t)\Big|^2\mathbb{I}_{\mathcal{A}^c(t)} \Big]\\
\le& C\|\nabla\Phi^{\varepsilon_k}_\vartheta\|^2_{L^\infty(\R^d)}\mathbb{P}(\mathcal{A}(t))
+C\|D^2\Phi^{\varepsilon_k}_\vartheta\|^2_{L^\infty(\R^d)}\mathbb{E}\Big[\max_{1\le i\le N}|X^{N,i,\varepsilon,\sigma}_t-\bar{X}^{i,\varepsilon,\sigma}_t|^2\mathbb{I}_{\mathcal{A}^c(t)} \Big]\\
\le& C\varepsilon_k^{-2(\vartheta-1)}N^{-\gamma}+C\varepsilon_k^{-2\vartheta}N^{-2a}.
\end{align*}
A similar argument proves that
\begin{align*}
I_{13}\le  C\varepsilon_k^{-2(\vartheta-1)}N^{-\gamma}+C\varepsilon_k^{-2\vartheta}N^{-2a}.
\end{align*}
We apply Lemma \ref{lemma_estim_prob_bad_sets} by choosing
\begin{align*}
\mathcal{A}_{\theta_{(5)}}^{(5)}(t) := \mathcal{A}_{\theta_{(5)}}^N(\nabla  \Phi_\vartheta^{\varepsilon_k} (\cdot),u^{\varepsilon, \sigma}(t, \cdot))\ \text{with} \ (\bar{Y}^i)_{i \in \{1, \ldots, N \}} = (\bar{X}_t^{i, \varepsilon, \sigma})_{i \in \{1, \ldots, N \}},
\end{align*}
to deduce that
\begin{align*}
I_{12}\le& C\mathbb{E}\Big[\frac{1}{N}\sum_{i=1}^N\Big|\frac{1}{N}\sum_{j=1}^N\nabla\Phi^{\varepsilon_k}_\vartheta(\bar{X}^{i,\varepsilon,\sigma}_t-\bar{X}^{j,\varepsilon,\sigma}_t)
-\nabla\Phi^{\varepsilon_k}_\vartheta\ast u^{\varepsilon,\sigma}(\bar{X}^{i,\varepsilon,\sigma}_t)\Big|^2\mathbb{I}_{\mathcal{A}_{\theta_{(5)}}^{(5)}(t)} \Big]\\
&+C\mathbb{E}\Big[\frac{1}{N}\sum_{i=1}^N\Big|\frac{1}{N}\sum_{j=1}^N\nabla\Phi^{\varepsilon_k}_\vartheta(\bar{X}^{i,\varepsilon,\sigma}_t-\bar{X}^{j,\varepsilon,\sigma}_t)
-\nabla\Phi^{\varepsilon_k}_\vartheta\ast u^{\varepsilon,\sigma}(\bar{X}^{i,\varepsilon,\sigma}_t)\Big|^2\mathbb{I}_{(\mathcal{A}_{{\theta_{(5)}}}^{(5)})^c(t)} \Big]\\
\le& C\|\nabla\Phi^{\varepsilon_k}_\vartheta\|^2_{L^\infty(\R^d)}
\mathbb{P}(\mathcal{A}_{\theta_{(5)}}^{(5)}(t))+CN^{-2\theta_{(5)}}\\
\le& C\varepsilon_k^{-2(\vartheta-1)(1+\tilde{\kappa}_{(5)})}N^{2\tilde{\kappa}_{(5)}(\theta_{(5)}-\frac{1}{2})+1}+CN^{-2\theta_{(5)}}.
\end{align*}
We combine the estimates for $I_{11}, I_{12}$ and $I_{13}$ to obtain that
\begin{align}\label{rela3}
I_1\le C\varepsilon_k^{-2(\vartheta-1)}N^{-\gamma}+C\varepsilon_k^{-2\vartheta}N^{-2a}
+C\varepsilon_k^{-2(\vartheta-1)(1+\tilde{\kappa}_{(5)})}N^{2\tilde{\kappa}_{(5)}(\theta_{(5)}-\frac{1}{2})+1}+CN^{-2\theta_{(5)}}.
\end{align}
Now we turn to the term $I_2$
\begin{align*}
I_2\le& C\mathbb{E}\Big[\frac{1}{N}\sum_{i=1}^N\Big|\Big(p'_\lambda\Big(\frac{1}{N}\sum_{j=1}^NV^{\varepsilon_p}(X^{N,i,\varepsilon,\sigma}_t-X^{N,j,\varepsilon,\sigma}_t)\Big)
-p'_\lambda\Big(\frac{1}{N}\sum_{j=1}^N V^{\varepsilon_p}(\bar{X}^{i,\varepsilon,\sigma}_t-\bar{X}^{j,\varepsilon,\sigma}_t)  \Big) \Big)\\
&\qquad\qquad\frac{1}{N}\sum_{j=1}^N\nabla V^{\varepsilon_p}(X^{N,i,\varepsilon,\sigma}_t-X^{N,j,\varepsilon,\sigma}_t)\Big|^2 \Big]\\
&+ C\mathbb{E}\Big[\frac{1}{N}\sum_{i=1}^N\Big| p'_\lambda\Big(\frac{1}{N}\sum_{j=1}^N V^{\varepsilon_p}(\bar{X}^{i,\varepsilon,\sigma}_t-\bar{X}^{j,\varepsilon,\sigma}_t) \Big)\\
&\qquad\qquad\frac{1}{N}\sum_{j=1}^N\big(\nabla V^{\varepsilon_p}(X^{N,i,\varepsilon,\sigma}_t-X^{N,j,\varepsilon,\sigma}_t)
-\nabla V^{\varepsilon_p}(\bar{X}^{i,\varepsilon,\sigma}_t-\bar{X}^{j,\varepsilon,\sigma}_t) \big)  \Big|^2\Big]\\
&+C\mathbb{E}\Big[\frac{1}{N}\sum_{i=1}^N\Big|\Big(p'_\lambda\Big(\frac{1}{N}\sum_{j=1}^NV^{\varepsilon_p}(\bar{X}^{i,\varepsilon,\sigma}_t-\bar{X}^{j,\varepsilon,\sigma}_t) \Big)
-p'_\lambda(V^{\varepsilon_p}\ast u^{\varepsilon,\sigma}(\bar{X}^{i,\varepsilon,\sigma}_t)) \Big)\\
&\qquad\qquad \frac{1}{N}\sum_{j=1}^N\nabla V^{\varepsilon_p}(\bar{X}^{i,\varepsilon,\sigma}_t-\bar{X}^{j,\varepsilon,\sigma}_t) \Big|^2 \Big]\\
&+C\mathbb{E}\Big[\frac{1}{N}\sum_{i=1}^N\Big|\big(p'_\lambda(V^{\varepsilon_p}\ast u^{\varepsilon,\sigma}(\bar{X}^{i,\varepsilon,\sigma}_t))
-p'_\lambda(V^{\varepsilon_p}\ast u^{\varepsilon,\sigma}(X^{N,i,\varepsilon,\sigma}_t)) \big)\\
&\qquad\qquad \frac{1}{N}\sum_{j=1}^N\nabla V^{\varepsilon_p}(\bar{X}^{i,\varepsilon,\sigma}_t-\bar{X}^{j,\varepsilon,\sigma}_t) \Big|^2 \Big]\\
&+C\mathbb{E}\Big[\frac{1}{N}\sum_{i=1}^N\Big|p'_\lambda(V^{\varepsilon_p}\ast u^{\varepsilon,\sigma}(X_t^{N,i,\varepsilon,\sigma}))
\Big(\frac{1}{N}\sum_{j=1}^N\nabla V^{\varepsilon_p}(\bar{X}^{i,\varepsilon,\sigma}_t-\bar{X}^{j,\varepsilon,\sigma}_t)
-\nabla V^{\varepsilon_p}\ast u^{\varepsilon,\sigma}(\bar{X}^{i,\varepsilon,\sigma}_t) \Big) \Big|^2 \Big]\\
&+C\mathbb{E}\Big[\frac{1}{N}\sum_{i=1}^N\Big|p'_\lambda(V^{\varepsilon_p}\ast u^{\varepsilon,\sigma}(X_t^{N,i,\varepsilon,\sigma}))
\Big(\nabla V^{\varepsilon_p}\ast u^{\varepsilon,\sigma}(\bar{X}^{i,\varepsilon,\sigma}_t)
-\nabla V^{\varepsilon_p}\ast u^{\varepsilon,\sigma}({X}^{N,i,\varepsilon,\sigma}_t) \Big) \Big|^2 \Big]\\
=:& \sum_{m=1}^6I_{2m}.
\end{align*}
By the same argument as that one used for the term $I_{11}$, we derive from Proposition \ref{theorem_algebraic} that
\begin{align*}
I_{21}\le& C\|p'_\lambda\|^2_{L^\infty(0,\|V^{\varepsilon_p}\|_{L^\infty(\R^d)})}\|\nabla V^{\varepsilon_p}\|^2_{L^\infty(\R^d)}\mathbb{P}(\mathcal{A}(t))\\
&+C\|\nabla V^{\varepsilon_p}\|^2_{L^\infty(\R^d)}\|p''_\lambda\|^2_{L^\infty(0,\|V^{\varepsilon_p}\|_{L^\infty(\R^d)})}\|\nabla V^{\varepsilon_p}\|^2_{L^\infty(\R^d)}
\mathbb{E}\Big[\max_{1\le i\le N}|X^{N,i,\varepsilon,\sigma}_t-\bar{X}^{i,\varepsilon,\sigma}_t|^2\mathbb{I}_{\mathcal{A}^c(t)} \Big]\\
\le& C\varepsilon_p^{-2(dm-d+1)}N^{-\gamma}+C\varepsilon_p^{-4-2d(|m-3|+2)}N^{-2a},
\end{align*}
where the assumption that $\lambda=\frac{\varepsilon_p^d}{2}$ has been used.
Similarly, we get
\begin{align*}
I_{22}+I_{24}+I_{26}
\le C\varepsilon_p^{-2(dm-d+1)}N^{-\gamma}
+C\varepsilon_p^{-2(dm-d+2)}N^{-2a}.
\end{align*}
Using Lemma \ref{lemma_estim_prob_bad_sets}, we have
\begin{align*}
I_{23}\le& C\|\nabla V^{\varepsilon_p}\|^2_{L^\infty(\R^d)}\|p''_\lambda\|^2_{L^\infty(0,\|V^{\varepsilon_p}\|_{L^\infty(\R^d)})}\\
&\qquad\mathbb{E}\Big[\max_{1\le i\le N}\Big|\frac{1}{N}\sum_{j=1}^NV^{\varepsilon_p}(\bar{X}^{i,\varepsilon,\sigma}_t-\bar{X}^{j,\varepsilon,\sigma}_t)
-V^{\varepsilon_p}\ast u^{\varepsilon,\sigma}(\bar{X}^{i,\varepsilon,\sigma}_t) \Big|^2 \Big]\\
\le& C\varepsilon_p^{-2(d+1)-2d|m-3|}\|V^{\varepsilon_p}\|^2_{L^\infty(\R^d)}
\mathbb{P}(\mathcal{A}_{\theta_{(6)}}^{(6)}(t))+C\varepsilon_p^{-2(d+1)-2d|m-3|}N^{-2\theta_{(6)}}\\
\le& C\varepsilon_p^{-2(d+1)-2d|m-3|-2d(1+\tilde{\kappa}_{(6)})}N^{2\tilde{\kappa}_{(6)}(\theta_{(6)}-\frac{1}{2})+1}
+C\varepsilon_p^{-2(d+1)-2d|m-3|}N^{-2\theta_{(6)}},
\end{align*}
where $\mathcal{A}_{\theta_{(6)}}^{(6)}(t) := \mathcal{A}_{\theta_{(6)}}^N(V^{\varepsilon_p} (\cdot),u^{\varepsilon, \sigma}(t, \cdot))$.
Due to the $L^\infty((0,T)\times\R^d)$ bound for $u^{\varepsilon,\sigma}$, we proceed similarly to infer that
\begin{align*}
I_{25}\le& C\|p'_\lambda\|^2_{L^\infty(0,\|V^{\varepsilon_p}\ast u^{\varepsilon,\sigma}\|_{L^\infty((0,T)\times\R^d)})}
\mathbb{E}\Big[\max_{1\le i\le N}\Big|\frac{1}{N}\sum_{j=1}^N\nabla V^{\varepsilon_p}(\bar{X}^{i,\varepsilon,\sigma}_t-\bar{X}^{j,\varepsilon,\sigma}_t)
-\nabla V^{\varepsilon_p}\ast u^{\varepsilon,\sigma}(\bar{X}^{i,\varepsilon,\sigma}_t) \Big|^2 \Big]\\
\le&C\|\nabla V^{\varepsilon_p}\|^2_{L^\infty(\R^d)}
\mathbb{P}(\mathcal{A}_{\theta_{(7)}}^{(7)}(t))+CN^{-2\theta_{(7)}}\\
\le& C\varepsilon_p^{-2(d+1)(1+\tilde{\kappa}_{(7)})}N^{2\tilde{\kappa}_{(7)}(\theta_{(7)}-\frac{1}{2})+1}+CN^{-2\theta_{(7)}},
\end{align*}
where $\mathcal{A}_{\theta_{(7)}}^{(7)}(t) := \mathcal{A}_{\theta_{(7)}}^N(\nabla V^{\varepsilon_p} (\cdot),u^{\varepsilon, \sigma}(t, \cdot))$.
Thus, in accordance with the estimates for $I_{21}$ to $I_{26}$, we derive that
\begin{align}\label{rela5}
I_2\le& C\varepsilon_p^{-2(dm-d+1)}N^{-\gamma}+C(\varepsilon_p^{-4-2d(|m-3|+2)}+\varepsilon_p^{-2(dm-d+2)})N^{-2a}\nonumber\\
&+ C\varepsilon_p^{-2(d+1)-2d|m-3|-2d(1+\tilde{\kappa}_{(6)})}N^{2\tilde{\kappa}_{(6)}(\theta_{(6)}-\frac{1}{2})+1}
+C\varepsilon_p^{-2(d+1)-2d|m-3|}N^{-2\theta_{(6)}}\nonumber\\
&+C\varepsilon_p^{-2(d+1)(1+\tilde{\kappa}_{(7)})}N^{2\tilde{\kappa}_{(7)}(\theta_{(7)}-\frac{1}{2})+1}+CN^{-2\theta_{(7)}}.
\end{align}
Noticing $\varepsilon_k\geq N^{-\beta_k}$ and $\varepsilon_p\geq \big( {\ln N} \big)^{-\frac{1}{2+d(2+|m-3|)}}$, we can choose $\tilde{\kappa}_{(5)}-\tilde{\kappa}_{(7)}$ big enough and $\gamma$, $a$, $\theta_{(5)}$, $\theta_{(6)}$, $\theta_{(7)}$ satisfy that
\begin{align*}
	2\beta_k(\vartheta-1)<\gamma,\quad \beta_k\vartheta<a<\frac{1}{2},\quad 0<\theta_{(5)}<\frac{1}{2}-\beta_k(\vartheta-1),\quad 0< \theta_{(6)},\theta_{(7)}<\frac{1}{2},
\end{align*} 
then insert \eqref{rela3} and \eqref{rela5} into \eqref{rela1} to obtain
\begin{align*}
	\frac{d}{dt}\mathcal{H}(u^{\varepsilon,\sigma}_N|{u^{\varepsilon,\sigma}}^{\otimes N})
	+\frac{\sigma}{2N}\sum_{i=1}^N\int_{\R^{dN}}u^{\varepsilon,\sigma}_N \Big|\nabla_{x_i}\log\frac{u^{\varepsilon,\sigma}_N}{{u^{\varepsilon,\sigma}}^{\otimes N}} \Big|^2 dx_1\cdots dx_N
	\le CN^{-2\mu},
\end{align*}
where $0<\mu<\min\{\gamma/2-\beta_k(\vartheta-1),a-\beta_k\vartheta,\theta_{(5)},\theta_{(6)},\theta_{(7)}\}$. 
This implies by using Csisz\'ar-Kullback-Pinsker inequality that for any $l\in\mathbb{N}$ it holds
\begin{align}\label{relative9}
\sup_{t\in (0,T)}\|u^{\varepsilon,\sigma}_{N,l}(t)-{u^{\varepsilon,\sigma}}^{\otimes l}(t)\|_{L^1(\mathbb{R}^{dl})}^2
&\leq \sup_{t\in (0,T)}2\mathcal{H}_l(u_{N,l}^{\varepsilon,\sigma}|{u^{\varepsilon,\sigma}}^{\otimes l})(t)
\leq \sup_{t\in (0,T)} 4l\mathcal{H}(u_{N}^{\varepsilon,\sigma}| {u^{\varepsilon,\sigma}}^{\otimes N})(t)\le CN^{-2\mu}.
\end{align}

\vskip5mm
Next we focus on the derivation of $L^1$ error estimate for the difference $u^{\varepsilon,\sigma}-u^\sigma$ under Assumption \ref{ass}. 
\begin{align*}
\int_{\R^d}|u^{\varepsilon,\sigma}-u^\sigma|dx
\le& \int_{B_R(0)}|u^{\varepsilon,\sigma}-u^\sigma|dx+ \int_{B_R(0)^c}|u^{\varepsilon,\sigma}-u^\sigma|dx\\
\le& CR^{d/2} \Big(\int_{\R^d}|u^{\varepsilon,\sigma}-u^\sigma|^2dx\Big)^{1/2}+\frac{1}{R^2}\int_{\R^d}|u^{\varepsilon,\sigma}-u^\sigma||x|^2dx\\
\le& CR^{d/2}(\varepsilon_k+\varepsilon_p)+CR^{-2}.
\end{align*}
Choosing $R^{-1}=(\varepsilon_k+\varepsilon_p)^{\frac{2}{d+4}}$, we deduce that
\begin{align*}
\|u^{\varepsilon,\sigma}-u^\sigma\|_{L^\infty(0,T;L^1(\R^d))}\le C(\varepsilon_k+\varepsilon_p)^{\frac{4}{d+4}}.
\end{align*}
Obviously, taking $\|u^{\varepsilon,\sigma}\|_{L^1(\R^{d})}=\|u^{\sigma}\|_{L^1(\R^{d})}=1$ into account, we can also obtain the estimate to the difference of the $l$ tensor products of them, namely
\begin{align*}
	\|{u^{\varepsilon,\sigma}}^{\otimes l}-{u^\sigma}^{\otimes l}\|_{L^\infty(0,T;L^1(\R^{dl}))}\le C(l)(\varepsilon_k+\varepsilon_p)^{\frac{4}{d+4}}.
\end{align*}
In accordance with \eqref{relative9}, it holds
\begin{align*}
\|u^{\varepsilon,\sigma}_{N,l}(t)-{u^{\sigma}}^{\otimes l}(t)\|_{L^\infty(0,T;L^1(\mathbb{R}^{dl}))}
&\leq CN^{-\mu}+ C(l)(\varepsilon_k+\varepsilon_p)^{\frac{4}{d+4}}\leq C(l)\varepsilon_p^{\frac{4}{d+4}}\quad {\mbox{for any }}l>0,
\end{align*}
where $\varepsilon_p\geq  (\ln N)^{-\frac{1}{2+d(2+|m-3|)}}$, finishing the proof of Theorem \ref{relativeL1}.
\end{proof}

Finally, we turn to the proof of Theorem \ref{relativeLq}.

\begin{proof}[Proof of Theorem \ref{relativeLq}]

As preparation, we investigate the $L^r$ estimate for $u^{\varepsilon,\sigma}_{N,l}$. From \eqref{relativeuN}, we obtain that the $l$-th marginal $u_{N,l}^{\varepsilon,\sigma}$ satisfies that
\begin{align}
\label{relative11}
\begin{cases}
\partial_t u_{N,l}^{\varepsilon,\sigma}(t,x_1,\cdots,x_l)=&\sigma\displaystyle\sum_{i=1}^l\Delta_{x_i}u^{\varepsilon,\sigma}_{N,l}
-\displaystyle\sum_{i=1}^l\nabla_{x_i}\cdot \Big(\int_{\R^{d(N-l)}}\frac{1}{N}\displaystyle\sum_{j=1}^N\nabla\Phi^{\varepsilon_k}_\vartheta(x_i-x_j)u^{\varepsilon,\sigma}_{N} dx_{l+1}\cdots dx_N \Big) \\
&+\displaystyle\sum_{i=1}^l\nabla_{x_i}\cdot\Big(\int_{\R^{d(N-l)}}\nabla p_\lambda\Big(\frac{1}{N}\displaystyle\sum_{j=1}^NV^{\varepsilon_p}(x_i-x_j) \Big)u^{\varepsilon,\sigma}_Ndx_{l+1}\cdots dx_N \Big),     \\
\quad u_{N,l}^{\varepsilon,\sigma}(0,x_1,\cdots,x_l)=&u_0^\sigma(x_1)\cdots u_0^\sigma(x_l).
\end{cases}
\end{align}
A multiplication of \eqref{relative11} by $r(u^{\varepsilon,\sigma}_{N,l})^{r-1}$ with $r\ge 2$ and integration over $\R^{dl}$ leads to
\begin{align*}
&\frac{d}{dt}\int_{\R^{dl}}|u^{\varepsilon,\sigma}_{N,l}|^r dx_1\cdots dx_l
+\sigma r(r-1)\displaystyle\sum_{i=1}^l\int_{\R^{dl}}(u^{\varepsilon,\sigma}_{N,l})^{r-2}|\nabla_{x_i} u^{\varepsilon,\sigma}_{N,l}|^2 dx_1\cdots dx_l\\
=& r\displaystyle\sum_{i=1}^l \int_{\R^{dl}}\Big(\int_{\R^{d(N-l)}}\frac{1}{N}\displaystyle\sum_{j=1}^N\nabla\Phi_\vartheta^{\varepsilon_k}(x_i-x_j)u^{\varepsilon,\sigma}_{N}dx_{l+1}\cdots dx_N \Big)\cdot\nabla_{x_i} (u^{\varepsilon,\sigma}_{N,l})^{r-1}dx_1\cdots dx_l\\
&-r\displaystyle\sum_{i=1}^l\int_{\R^{dl}}\Big(\int_{\R^{d(N-l)}}\nabla p_\lambda\Big(\frac{1}{N}\sum_{j=1}^NV^{\varepsilon_p}(x_i-x_j) \Big)u^{\varepsilon,\sigma}_N dx_{l+1}\cdots dx_N \Big)\cdot\nabla_{x_i}(u^{\varepsilon,\sigma}_{N,l})^{r-1} dx_1\cdots dx_l\\
\le&\frac{\sigma r(r-1)}{2}\displaystyle\sum_{i=1}^l\int_{\R^{dl}}(u^{\varepsilon,\sigma}_{N,l})^{r-2}|\nabla_{x_i} u^{\varepsilon,\sigma}_{N,l}|^2 dx_1\cdots dx_l\\
&+C\displaystyle\sum_{i=1}^l\int_{\R^{dl}}(u^{\varepsilon,\sigma}_{N,l})^{r-2}\Big(\int_{\R^{d(N-l)}}\frac{1}{N}\displaystyle\sum_{j=1}^N\nabla\Phi_\vartheta^{\varepsilon_k}(x_i-x_j)u^{\varepsilon,\sigma}_Ndx_{l+1}\cdots dx_N \Big)^2 dx_1\cdots dx_l\\
&+C\displaystyle\sum_{i=1}^l\int_{\R^{dl}}(u^{\varepsilon,\sigma}_{N,l})^{r-2}\Big(\int_{\R^{d(N-l)}}\nabla p_\lambda\Big(\frac{1}{N}\sum_{j=1}^NV^{\varepsilon_p}(x_i-x_j) \Big)u^{\varepsilon,\sigma}_{N}dx_{l+1}\cdots dx_N\Big)^2dx_1\cdots dx_l.
\end{align*}
H\"older's inequality gives
\begin{align*}
&\frac{d}{dt}\int_{\R^{dl}}|u^{\varepsilon,\sigma}_{N,l}|^r dx_1\cdots dx_l
+\frac{\sigma r(r-1)}{2}\displaystyle\sum_{i=1}^l\int_{\R^{dl}}(u^{\varepsilon,\sigma}_{N,l})^{r-2}|\nabla_{x_i} u^{\varepsilon,\sigma}_{N,l}|^2 dx_1\cdots dx_l\\
\le&C\|\nabla\Phi_\vartheta^{\varepsilon_k}\|^2_{L^\infty(\R^{d})}\int_{\R^{dl}}(u^{\varepsilon,\sigma}_{N,l})^{r-2}\Big(\int_{\R^{d(N-l)}}u^{\varepsilon,\sigma}_{N}dx_{l+1}\cdots dx_N \Big)^2 dx_1\cdots dx_l\\
&+C\|p'_\lambda\|^2_{L^\infty(0,\|V^{\varepsilon_p}\|_{L^\infty(\R^d)})}\|\nabla V^{\varepsilon_p}\|^2_{L^\infty(\R^d)}
\int_{\R^{dl}}(u^{\varepsilon,\sigma}_{N,l})^{r-2}\Big(\int_{\R^{d(N-l)}}u^{\varepsilon,\sigma}_{N}dx_{l+1}\cdots dx_N \Big)^2 dx_1\cdots dx_l\\
\le& C(\varepsilon_k^{-2(\vartheta-1)}+\varepsilon_p^{-2(dm-d+1)})\int_{\R^{dl}}|u^{\varepsilon,\sigma}_{N,1}|^r dx_1\cdots dx_l.
\end{align*}
In view of $u_0^\sigma\in L^r(\R^d)$ and Gronwall's inequality, we achieve that
\begin{align*}
\sup_{t\in (0,T)}\int_{\R^{dl}}|u^{\varepsilon,\sigma}_{N,l}|^r dx_1\cdots dx_l
\le& \hat{C}(r) \exp\{\hat{C}(r)T(\varepsilon_k^{-2(\vartheta-1)}+\varepsilon_p^{-2(dm-d+1)})\},
\end{align*}
where $\hat{C}$ is a positive constant independent of $N$, $\varepsilon_k$ and $\varepsilon_p$.
Next, by a proper choice of $\varepsilon_k$ and $\varepsilon_p$, we aim to control the right-hand side of the above inequality by a positive power of $N$. 
Without loss of generality, we take this power to be $1/4$.
Since $3(dm-d+1)\ge 2+d(2+|m-3|)$, we achieve that $(\frac{1}{8\hat{C}T}\ln N)^{\frac{1}{3(dm-d+1)}}\le (\ln N)^{\frac{1}{2+d(2+|m-3|)}}$,
which ensures us to take $\varepsilon_p\geq (\frac{1}{8\hat{C}T}\ln N)^{-\frac{1}{3(dm-d+1)}}\ge (\ln N)^{-\frac{1}{2+d(2+|m-3|)}}$.
Let $\varepsilon_k\geq(\frac{1}{8\hat{C}T}\ln N)^{-\frac{1}{2(\vartheta-1)}}\ge N^{-\beta_k}$. 
Therefore,
\begin{align}\label{relative12}
\sup_{t\in (0,T)}\int_{\R^{dl}}|u^{\varepsilon,\sigma}_{N,l}|^r dx_1\cdots dx_l
\le \hat{C} \exp\{\hat{C}T(\varepsilon_k^{-2(\vartheta-1)}+\varepsilon_p^{-3(dm-d+1)})\}\le CN^{1/4}.
\end{align}
The estimates \eqref{relative9}, \eqref{relative12} and interpolation yield, for any $1<q<\infty$,
\begin{align}\label{relative13}
\|u^{\varepsilon,\sigma}_{N,l}-{u^{\varepsilon,\sigma}}^{\otimes l}\|_{L^q(\R^{dl})}
\le& C\|u^{\varepsilon,\sigma}_{N,l}-{u^{\varepsilon,\sigma}}^{\otimes l}\|_{L^{2q}(\R^{dl})}^{\frac{2(q-1)}{2q-1}}
\|u^{\varepsilon,\sigma}_{N,l}-{u^{\varepsilon,\sigma}}^{\otimes l}\|_{L^1(\R^{dl})}^{\frac{1}{2q-1}}\nonumber\\
\le& CN^{\frac{1}{8q}\frac{2(q-1)}{2q-1}}N^{\frac{-\mu}{2q-1}}
=CN^{\frac{1}{2q-1}(\frac{q-1}{4q}-\mu)},
\end{align} 
where $\frac{q-1}{4q}<\frac{1}{4}<\mu$.

By Assumption \ref{ass}, $\|u^{\varepsilon,\sigma}-u^{\sigma}\|_{L^\infty(0,T;L^2\cap L^\infty(\R^d))}\le C(\varepsilon_k+\varepsilon_p)$ and thus, by interpolation, achieving
\begin{align*}
\|u^{\varepsilon,\sigma}-u^\sigma\|_{L^\infty(0,T;L^q(\R^d))}\le C(\varepsilon_k+\varepsilon_p)\quad{\mbox{for any }}2\le q\le\infty.
\end{align*}
Furthermore, it follows from the $L^\infty(0,T;L^1(\R^d))$ bound for $u^{\varepsilon,\sigma}$ and $u^\sigma$ that, for any $1<q<2$,
\begin{align*}
\|u^{\varepsilon,\sigma}-u^\sigma\|_{L^\infty(0,T;L^q(\R^d))}
\le \|u^{\varepsilon,\sigma}-u^\sigma\|_{L^\infty(0,T;L^1(\R^d))}^{\frac{1}{q}}\|u^{\varepsilon,\sigma}-u^\sigma\|_{L^\infty((0,T)\times\R^d)}^{1-\frac{1}{q}}
\le C(\varepsilon_k+\varepsilon_p)^{1-\frac{1}{q}}.
\end{align*}
We combine this result and \eqref{relative13} to complete the proof of Theorem \ref{relativeLq}.
\begin{remark}
In \eqref{relative13}, the condition $\frac{q-1}{4q}<\frac{1}{4}<\mu$ follows from the choices $\frac{1}{4}+\beta_k\vartheta<a<\frac{1}{2}$, $\frac{1}{4}<\theta_{(5)}<\frac{1}{2}-\beta_k(\vartheta-1)$ and $\frac{1}{4}< \theta_{(6)},\theta_{(7)}<\frac{1}{2}$
subject to $\beta_k<\frac{1}{4\vartheta}$.
Since $\varepsilon_p\geq (\frac{1}{8\hat{C}T}\ln N)^{-\frac{1}{3(dm-d+1)}}$
and $\varepsilon_k\geq(\frac{1}{8\hat{C}T}\ln N)^{-\frac{1}{2(\vartheta-1)}}$,
the additional condition $\beta_k<\frac{1}{4\vartheta}$ is no longer needed for this scaling.
\end{remark}
\end{proof}

\medskip
\indent
{\bf Acknowledgements:}
The authors would like to thank Paul Nikolaev for pointing out the mistake during the preparation of this paper. L. Chen is partially supported by Deutsche Forschungsgemeinschaft (DFG, German Research Foundation, Grant No. 547277619) and the National Natural Foundation of China ( Grant No. 12171218).
Y. Li is supported by NSFC (Grant No. 12501268), Taishan Scholars Foundation of Shandong Province (Grant No. tsqnz20250715) and the Shandong Provincial Natural Science Foundation (Grant No. ZR2025QC1508).

\end{document}